\documentclass[reqno,11pt]{amsart}
\usepackage[top=2.0cm,bottom=2.0cm,left=3cm,right=3cm]{geometry}
\usepackage{amsthm,amsmath,amssymb,dsfont}
\usepackage{mathrsfs,amsfonts,functan,extarrows,mathtools}
\usepackage[colorlinks]{hyperref}
\usepackage{marginnote}
\usepackage{xcolor}
\usepackage{stmaryrd,bm}
\usepackage{esint}
\usepackage{graphicx}
\usepackage{pgfplots}
\usepackage{tikz}
\usepackage[english]{babel}
\usepackage{appendix}
\usepackage{bbm}

\newtheorem{theorem}{Theorem}[section]
\newtheorem{proposition}{Proposition}[section]

\newtheorem{remark}{Remark}[section]
\newtheorem{lemma}{Lemma}[section]

\numberwithin{equation}{section}
\allowdisplaybreaks

\def\p{\partial}
\def\vr{\varrho}

\def\d{\mathrm{d}}
\def\no{\nonumber}
\def\R{\mathbb{R}}
\def\eps{\varepsilon}

\def\l{\langle}
\def\r{\rangle}

\def\T{\mathbb{T}}

\newcounter{wronumber}

\begin{document}
\title[Models on engineered {\em Escherichia coli} populations]
 { On kinetic and macroscopic models for the stripe formation in engineered bacterial populations}

\author[Ning Jiang]{Ning Jiang}
\address[Ning Jiang]{\newline School of Mathematics and Statistics, Wuhan University, Wuhan, 430072, P. R. China}
\email{njiang@whu.edu.cn}

\author[Jiangyan Liang]{Jiangyan Liang}
\address[Jiangyan Liang]{\newline School of Mathematics and Statistics, Wuhan University, Wuhan, 430072, P. R. China}
\email{ljymath@whu.edu.cn}

\author[Yi-Long Luo]{Yi-Long Luo}
\address[Yi-Long Luo]
{\newline School of Mathematics, South China University of Technology, Guangzhou, 510641, P. R. China}
\email{luoylmath@scut.edu.cn}

\author[Min Tang]{Min Tang}
\address[Min Tang]
{\newline School of Mathematics, Institute of Natural Sciences and MOE-LSC, Shanghai Jiao Tong University, Shanghai 200240, China}
\email{tangmin@sjtu.edu.cn}

\author[Yaming Zhang]{Yaming Zhang}
\address[Yaming Zhang]
{\newline School of Mathematics, Institute of Natural Sciences and MOE-LSC, Shanghai Jiao Tong University, Shanghai 200240, China}
\email{zymsasj@sjtu.edu.cn}

\thanks{\today}

\begin{abstract}
We study the well-posedness of the biological models with AHL-dependent cell mobility on engineered {\em Escherichia coli} populations. For the kinetic model proposed by Xue-Xue-Tang recently \cite{XXT-2018-CB}, the local existence for large initial data is proved first. Furthermore, the positivity and local conservation laws for density $\rho (t,x,z)$ and nutrient $n(t,x)$ with initial assumptions \eqref{IC-K2} and \eqref{IC-Average} are justified. Based on these properties, it can be extended globally in time near the equilibrium $(0,0,0)$. Considering the asymptotic behaviors of faster response CheZ turnover rate (i.e., $\eps \rightarrow 0$), one formally derives an anisotropic diffusion engineered Escherichia coli populations model (in short, AD-EECP) for which we find a key extra a priori estimate to overcome the difficulties coming from the nonlinearity of the diffusion structure. The local well-posedness and the positivity and local conservation laws for density and nutrient of the AD-EECP are justified. Furthermore, the global existence around the steady state $(\vr_a, h_a, 0)$ with $\vr_a \in [0, \Lambda_b)$ is obtained. \\


\end{abstract}

\maketitle

\section{Introduction}\label{Sec-Intro}

\subsection{PDE models on engineered {\em Escherichia coli} populations}

It is a fundamental problem in developmental biology to understand the formation of regularly spaced structures, such as vertebrate segments, hair follicles, fish pigmentation and animal coats \cite{BGCAW-2005, BSM-2009, EL-2010, PRL2012, H-1980, M-2002, MV-2009}.  These patterns involve the complex interaction of intracellular signaling, cell-cell communication, cell growth and cell migration. The overwhelmingly complex physiological context usually makes it difficult to uncover the interplay of these mechanisms. Synthetic biology has recently been used to extract essential components of complex biological systems and examine potential strategies for pattern formation.

One of these problems relate to the bacterium {\em Escherichia coli}. Recently in \cite{LFLRCL-2011-S}, the chemotaxis signaling pathway of {\em E. coli} has been engineered and coupled with a quorum sensing module, leading to cell-density suppressed cell motility. When a suspension of the engineered cells is inoculated at the center of a petri dish with semi-solid agar and rich nutrient, the colony grows, moves outward and sequentially establishes rings or ``stripes" with a high density of cells behind the colony front. These spatial patterns form in a strikingly similar way
as many periodic patterns in other biological systems. When the maximum density of the motile front reaches a threshold, an immotile zone is nucleated. The immotile zone then absorbs bacteria from its neighborhood to expand, forming alternating high and low density zones. These patterns do not form when using wild-type {\em E. coli}; instead, the colony simply expands outward and forms a uniform lawn.

In \cite{LFLRCL-2011-S}, the quorum-sensing module of bacterium  was used to control the transcription of {CheZ}. The engineered cell synthesizes and secretes acyl-homoserine lactone (AHL), a small molecule that is freely diffusible across the cell membrane and degrades rapidly. At high concentrations, AHL suppresses the transcription of { CheZ} in an ultra-sensitive manner. If {CheZ} is suppressed, CheZ protein becomes
diluted as the cell grows and divides. Because CheZ is a dephosphorylation kinase of CheYp, a reduction of CheZ protein can immediately lead to higher CheYp concentration and thus more persistent tumbles of the cell. This, in turn, causes changes to the chemoreceptors as well as to other proteins involved in chemotaxis, and triggers a non-classic chemotactic cellular response. To quantify the effect of AHL in single cell movement, one must take into account the whole chemotaxis pathway as well as CheZ turnover.

A phenomenological PDE model was used to explain the pattern formation process in \cite{LFLRCL-2011-S} and a simplified version was analyzed in \cite{PRL2012}. The model consists of a system of reaction-diffusion equations for the cell density, AHL and nutrient concentrations. The diffusion rate of the cell population is assumed to be a switch-like function of the local AHL concentration. More precisely, the PDE model in \cite{LFLRCL-2011-S} is:
\begin{equation}\label{2011-Science}
  \left\{
    \begin{array}{l}
      \p_t \vr = \Delta_x ( \mu(h) \vr ) +  \tfrac{\gamma n^2 \vr}{n^2 + K^2_n} \,, \\[1.5mm]
      \p_t h = D_h \Delta_x h + \alpha \vr - \beta h \,, \\[1.5mm]
      \p_t n = D_n \Delta_x n - \tfrac{k_n \gamma n^2\vr }{n^2 + K^2_n}\,,
    \end{array}
  \right.
\end{equation}
where $\vr(x,t)$ is the cell density, $h(x,t)$ is the AHL concentration, and $n(x,t)$ is the nutrient level. The key feature (also the analytical difficulty) of this model is that the AHL-dependent cell mobility $\mu(h)$ is modeled by a steep Hill function with an abrupt transition between two different values $D_\vr$ and $D_{\vr,0}$ at $h\approx K_h$, where $K_h>0$ is the AHL threshold of mobility regulation (See Figure \ref{Fig-AHL}). Moreover, $\gamma, k_n, K_n, \alpha, \beta$ are positive parameters. There are some mathematical analysis of the models proposed in \cite{LFLRCL-2011-S} and \cite{PRL2012}, such as \cite{JKW-SIAM2018, JSW-JDE2020, MPW-PhysicaD} in which the system \eqref{2011-Science} is named {\em reaction-diffusion system with density-dependent motility}.

\begin{figure}[h]
 \centering
 \begin{tikzpicture}[node distance=2cm]
 \draw[->](0,0)--(4,0) node[right,scale=1]{AHL concentration $h$};
 \draw[->](0,0)--(0,4) node[above,scale=1]{Cell motility $\mu (h)$};
 \draw[-, very thick, color=red](0,3)--(1.5,3) ;
 \draw[-] node[below left]{0} (2.3,0)--(4,0) ;
 \draw[scale=1, very thick, domain=1.5:2.5, smooth,variable=\t,color=red]
 plot (\t,{1.25*cos( 3.1415926 * ( \t  - 1.5 ) r ) + 1.75 });
 \draw (0,3) node[left]{$D_\rho$};
 \draw[-,thick,dashed] (2,3)--(2,0) node[below, scale=1]{$K_h$};
 \draw[-,very thick,color=red](2.5,0.5)--(3.9,0.5);
 \draw[-](0,0.5)--(0.1,0.5) node[left]{$D_{\rho,0}$};
 \end{tikzpicture}
 \caption{AHL-dependent cell mobility}\label{Fig-AHL}
\end{figure}
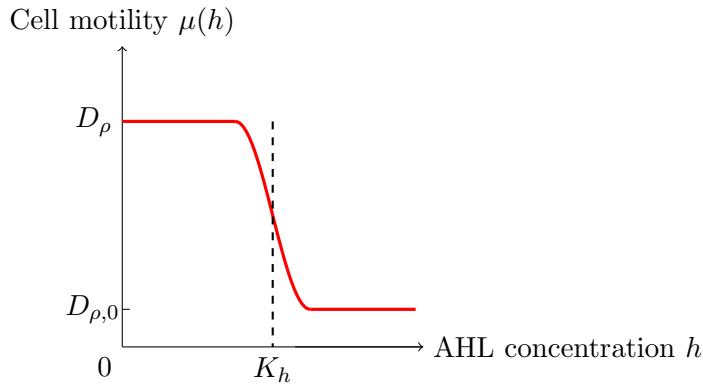

Since the whole chemotaxis pathway is involved in the pattern formation process, it is unclear how cell movement can be reduced to an anisotropic (or cross) diffusion process. Moreover, the model \eqref{2011-Science} does not address the role of intracellular signaling in stripe formation and cannot be used to understand how the spatial structure of the high-density and low-density regions depends on cell-level parameters. To address these questions, recently Xue, Xue and Tang \cite{XXT-2018-CB} proposed a hybrid model for the stripe formation which accounts for the behavior of individual cells. This hybrid model starts with a detailed description of intracellular signaling, single cell movement and cell division, and individual-based component is then coupled with reaction-diffusion equations for AHL and nutrient concentrations. However the hybrid models are very time-consuming due to the large number of cells involved in the pattern formation process. To overcome this computational challenge, in \cite{XXT-2018-CB} Xue-Xue-Tang then derived a mesoscopic model for the cell density from the hybrid model. Numerical comparisons of the hybrid model and the PDE model showed quantitative agreement in 1D under biologically-relevant parameter regimes. More precisely, they derived the following kinetic equation model for  engineered {\em Escherichia coli} populations (briefly named K-EECP system) \cite{XXT-2018-CB}:
\begin{equation}\label{EECPK}\tag{K-EECP}
  \left\{
    \begin{array}{c}
      \p_t \rho = D (z) \Delta_x \rho - \tfrac{1}{\eps} \p_z ( g (z, h) \rho ) + \gamma n \rho \,, \\[1.5mm]
      \p_t h = D_h \Delta_x h + \alpha {\vr} - \beta h \,, \\[1.5mm]
      \p_t n = D_n \Delta_x n - \xi \vr n \,, \\[1.5mm]
      \vr (t, x) = \int_{\T_w} \rho (t,x,z) d z \,,
    \end{array}
  \right.
\end{equation}
with initial data
\begin{equation}\label{IC-EECPK}
  \begin{aligned}
    ( \rho , h, n ) |_{t=0} = ( \rho^{\eps, in} (x, z) , h^{\eps, in} (x) , n^{\eps, in} (x) ) \,.
  \end{aligned}
\end{equation}
In the system \eqref{EECPK}, the scalar function $\rho (t,x,z) $ is the density of {\em Escherichia coli} cells at time $t \geq 0$, at position $x \in \T^3 : = \R^3 / [-\pi, \pi]^3$ with internal state $z \in \T_w : = \R / [0, Z_w ]$ (here $Z_w$ is the CheZ concentration of wild type Escherichia coli cells), namely, with the total concentration $z$ of CheZ protein. Here the period domains are considered for the sake of avoiding the mathematical troubles on the higher order derivatives values on the boundary. $\vr (t,x)$ is the total density of the {\em Escherichia coli} cells at time $t \geq 0$ and at position $x \in \T^3$. The scalar function $h (t,x)$ is the AHL concentration at time $t \geq 0$ and at position $x \in \T^3$. The scalar function $n (t,x)$ represents the local nutrient concentration at time $t \geq 0$ and at position $x \in \T^3$. The coefficient $\gamma > 0$ represents the cell growth rate, the positive constants $D_h$ and $D_n$ denote the diffusion coefficients of AHL and the nutrient, respectively. $\alpha > 0$ is the production rate of AHL, $\beta > 0$ stands for the degradation rate of AHL and $\xi > 0$ characterizes the consumption rate of nutrient. Moreover, the known scalar function $D (z) > 0$ is the diffusion coefficient with the form
\begin{equation}\label{D(z)}
  \begin{aligned}
    D ( z ) = \frac{ s_0^2 \mu_0 ( z ) }{ 3 \lambda_0 ( z ) [ \mu_0 ( z ) + \lambda_0 ( z ) ] } \,,
  \end{aligned}
\end{equation}
where $s_0$ represents that cells can move in any directions with constant speed $s_0$. $\lambda_0 (z)$ is the switching frequencies from the running to tumbling states, and $\mu_0 (z)$ is that from the tumbling to running states. In other words, $D (z)$ characterizes that cells with different intracellular CheZ concentration $z$ have different mobility coefficient. $g (z, h)$ is the intracellular dynamic of CheZ protein. In \cite{XXT-2018-CB}, $g (z, h)$ is given the form
\begin{equation}\label{g(zh)}
  \begin{aligned}
    g (z, h) = k_V L (h) - k_V z \,,
  \end{aligned}
\end{equation}
where the steady state $L(h)$ of the intracellular CheZ is
\begin{equation}\label{L(h)}
  L (h) =
    \left\{
      \begin{array}{l}
        Z_w \,, \quad \textrm{ if } h < \overline{h} \,, \\
        0 \,, \qquad  \textrm{ if } h \geq \overline{h} \,.
      \end{array}
    \right.
\end{equation}
Here $k_V > 0$ is the growth rate of the volume of a cell, $\overline{h}$ is the threshold AHL level for the suppression of CheZ, which is a non-dimensional constant.

The non-dimensional parameter $\tfrac{1}{\eps}$ preceding the $z$-flux term in the first $\rho$-equation of \eqref{EECPK} is introduced to investigate the effect of a slower or faster CheZ turnover rate on the spatial patterning. It parameterizes the convective speed in $z$, which corresponds to the response speed of intracellular CheZ to the external signal AHL. More precisely,
\begin{itemize}
	\item $\eps = 1$ represents the base-line model;
	\item $0 < \eps < 1$ represents faster response;
	\item $\eps > 1$ represents slower response.
\end{itemize}

Since numerical simulations of the kinetic model \eqref{EECPK} share similar behavior with models in \cite{LFLRCL-2011-S}, a natural question is: what's the relation between the kinetic model \eqref{EECPK} with the PDE model \eqref{2011-Science}? The answer is: it depends on the CheZ turnover rate, which is represented by the coefficient $\frac{1}{\eps}$ in the first equation of \eqref{EECPK}. The formal analysis is provided in the next subsection.

\subsection{Formal asymptotic behaviors of fast CheZ turnover rate}\label{Subsec:Formal}

It is natural to analyze the asymptotic behavior on the non-dimensional parameter $\eps$: the fast CheZ turnover rate corresponding to $\eps \rightarrow 0$, which is a singular limit problem. Assume $( \rho^\eps (t, x, z), h^\eps (t, x), n^\eps (t,x) )$ is a solution to the system \eqref{EECPK} with initial data \eqref{IC-EECPK}, which {\em formally} converges to $ ( \rho (t,x,z), h (t,x), n (t, x) )$ as $\eps \rightarrow 0$. Then
\begin{equation*}
  \begin{aligned}
    - \p_z ( ( k_V L (h^\eps) - k_V z ) \rho^\eps ) = \eps \p_t \rho^\eps - \eps D (z) \Delta_x \rho^\eps - \eps \gamma n^\eps \rho^\eps \longrightarrow 0
  \end{aligned}
\end{equation*}
as $\eps \rightarrow 0$, which means that for any test function $\phi (z) \in C_c^\infty (\T_w)$,
\begin{equation*}
  \begin{aligned}
    \int_{\T_w} ( k_V L (h^\eps) - k_V z ) \rho^\eps \phi (z) \d z = \int_{\T_w} ( k_V L (h^\eps) - k_V z ) \rho^\eps \Phi' (z) \d z \\
    = - \int_{\T_w} \p_z ( ( k_V L (h^\eps) - k_V z ) \rho^\eps ) \Phi (z) \d z \rightarrow 0
  \end{aligned}
\end{equation*}
as $\eps \rightarrow 0$, where $\Phi (z)$ is the primitive function of $\phi (z)$. Moreover,
\begin{equation*}
  \begin{aligned}
    \int_{\T_w} ( k_V L (h^\eps) - k_V z ) \rho^\eps \phi (z) \d z \rightarrow \int_{\T_w} ( k_V L (h) - k_V z ) \rho \phi (z) \d z \ \textrm{ as } \eps \rightarrow 0 \,.
  \end{aligned}
\end{equation*}
We then know that for any $\phi (z) \in C_c^\infty (\T_w)$,
\begin{equation*}
  \begin{aligned}
    \int_{\T_w} ( k_V L (h) - k_V z ) \rho \phi (z) \d z = 0 \,,
  \end{aligned}
\end{equation*}
which mean that
\begin{equation}\label{Steady-state-z}
  \begin{aligned}
    z  = L (h) \,.
  \end{aligned}
\end{equation}
This indicates that the limit $\rho (t, x ,z) = \vr (t,x) \delta (z - L (h))$, where $\vr (t,x) = \int_{\T_w} \rho (t,x,z) \d z$ and $\delta (\cdot)$ denotes the $\delta
$-function such that $\delta (0) = \infty$ and $\delta (z) = 0$ for $z \neq 0$. Then for the first equation of \eqref{EECPK}, i.e.,
\begin{equation*}
  \begin{aligned}
    \p_t \rho^\eps (t, x ,z) = \Delta_x ( D(z) \rho^\eps (t,x,z) ) - \tfrac{k_V}{\eps} \p_z \big( ( L (h^\eps (t,x)) - z ) \rho^\eps (t,x,z) \big) + \gamma n^\eps (t,x) \rho^\eps (t,x,z) \,,
  \end{aligned}
\end{equation*}
we integrate by parts over $z \in \T_w$ and obtain
\begin{equation*}
  \begin{aligned}
    \p_t \int_{\T_w} \rho^\eps (t, x ,z) \d z = \Delta_x ( \int_{\T_w} D(z) \rho^\eps (t,x,z) \d z ) + \gamma n^\eps (t,x) \int_{\T_w} \rho^\eps (t,x,z) \d z \,.
  \end{aligned}
\end{equation*}
By letting $\eps \rightarrow 0$ in the above equation, we {\em formally} deduce that
\begin{equation}\label{Formal-1}
  \begin{aligned}
    \p_t \int_{\T_w} \rho (t, x ,z) \d z = \Delta_x ( \int_{\T_w} D(z) \rho (t,x,z) \d z ) + \gamma n (t,x) \int_{\T_w} \rho (t,x,z) \d z \,.
  \end{aligned}
\end{equation}
We plug $\rho (t, x ,z) = \vr (t,x) \delta (z - L (h))$ into \eqref{Formal-1}, so that
\begin{equation}\label{Formal-2}
  \begin{aligned}
    \p_t \vr (t,x) = \Delta_x ( D (L (h (t,x))) \vr (t,x) ) + \gamma n (t,x) \vr (t,x) \,,
  \end{aligned}
\end{equation}
where the calculation
\begin{equation*}
  \begin{aligned}
    \int_{\T_w} D(z) \rho (t,x,z) \d z = \int_{\T_w} D(z) \vr (t,x) \delta (z - L (h (t,x))) \d z = D ( L ( h (t, x) ) ) \vr (t,x)
  \end{aligned}
\end{equation*}
is utilized. Furthermore, from taking limit in the second and the third equations in \eqref{EECPK} as $\eps \rightarrow 0$ and combining with the equation \eqref{Formal-2}, one formally derives the following anisotropic diffusion engineered {\em Escherichia coli} populations model (in short, AD-EECP system)
\begin{equation}\label{AD-EECP}\tag{AD-EECP}
  \left\{
    \begin{array}{l}
      \p_t \vr = \Delta_x ( \widetilde{D} (h) \vr ) + \gamma n \vr \,, \\[1.5mm]
      \p_t h = D_h \Delta_x h + \alpha \vr - \beta h \,, \\[1.5mm]
      \p_t n = D_n \Delta_x n - \xi \vr n \,,
    \end{array}
  \right.
\end{equation}
where
\begin{equation}\label{Tilde-D}
  \begin{aligned}
    \widetilde{D} (h) : = D (L (h)) \,.
  \end{aligned}
\end{equation}
Note, the model \eqref{AD-EECP} is almost the same as \eqref{2011-Science}. Now the the AHL-dependent cell mobility becomes $\widetilde{D} (h)$. In this sense, the models proposed in \cite{XXT-2018-CB} is more precise and more general than the models in \cite{LFLRCL-2011-S} and \cite{PRL2012}. Remark that the previous formal derivation sensitively requires the discontinuity of $\rho (t,x,z)$ (involving the Dirac function $\delta (z)$).

The previous formal limit mechanism indicates that when the response speed of intracellular CheZ to the external signal AHL is very fast, the intracellular CheZ concentration $z$ of cells will be approximated by its steady state $L(h)$, as shown in \eqref{Steady-state-z}. Finally, the initial data of the system \eqref{AD-EECP} is imposed on
\begin{equation}\label{IC-AD-EECP}
  \begin{aligned}
    {\vr} (0,x) = {\vr}^{in}_a (x)  \,, \quad h (0,x) = h^{in}_a (x) \,, \quad n (0,x) = n^{in}_a (x) \,.
  \end{aligned}
\end{equation}
Formally, it can be assumed that as $\eps \to 0$,
\begin{equation*}
	\begin{aligned}
		\int_{ \T_w } \rho^{\eps, in} (x,z) \d z \to \vr_a^{in} (x) \,, \ h^{\eps, in} (x) \to h_a^{in} (x) \,, \ n^{\eps, in} (x) \to n_a^{in} (x) \,.
	\end{aligned}
\end{equation*}

It is natural to analytically study the kinetic and macroscopic models mentioned in the above subsections. The following mathematical analysis problems should be addressed:
\begin{enumerate}
  \item Mathematically prove the existence of pattern formation, which should be represented as periodic travelling waves solutions of \eqref{EECPK} and \eqref{AD-EECP};
  \item Justify the limit from the hybrid model proposed in \cite{XXT-2018-CB} to the kinetic model \eqref{EECPK};
  \item Justify the limit from the kinetic model \eqref{EECPK} to the macroscopic model \eqref{AD-EECP};
\end{enumerate}
To address the above problems, the first step is the existence of the solutions to these systems. This is the goal of this paper:  the well-posedness of the kinetic system \eqref{EECPK}, the macroscopic system \eqref{AD-EECP}  in the {\em classical solutions regime}. For both kinetic and macroscopic models, the main analytical difficulties come from the AHL-dependent mobility.

\subsection{Some smoothness hypotheses}
 We first consider the mollified steady state $L_\ell  (h)$ of the intracellular CheZ $z$:

Let $\varphi (\varsigma) = \frac{1}{\kappa_0} \exp ( \frac{1}{|\varsigma|^2 - 1} )$ if $|\varsigma| < 1$ and $\varphi (\varsigma) = 0$ if $|\varsigma| \geq 1$, where $\kappa_0 = \int_{|\varsigma| < 1} \exp ( \frac{1}{|\varsigma|^2 - 1} ) d \varsigma > 0$. For any $\ell > 0$, we denote $\varphi_\ell (\varsigma) = \frac{1}{\ell} \varphi ( \frac{\varsigma}{\ell} )$. We then define the smoothly mollified function $L_\ell (h)$ of the step function $L(h)$ given in \eqref{L(h)} such that
\begin{equation}\label{L(h)-m}
  \begin{aligned}
    L_\ell  (h) = \int_\R L (\varsigma) \varphi_\ell ( h - \varsigma ) \d \varsigma \,,
  \end{aligned}
\end{equation}
whose diagram is as shown in Figure \ref{Fig-MAAI-CheZ}.
\begin{figure}[h]
	\centering
	\begin{tikzpicture}[node distance=2cm]
	  \draw[->](-2,0)--(4,0) node[right,scale=1]{$h$};
	  \draw[->](0,-1)--(0,3.5) node[above,scale=1]{$L_\ell (h)$};
	  \draw[-](-1.5,2)--(1.3,2) ;
	  \draw[-] node[below left]{0} (2.3,0)--(4,0) ;
	  \draw[scale=1,domain=1.3:2.3,smooth,variable=\t]
	  plot (\t,{cos( 3.1415926 * ( \t  - 1.3 ) r ) + 1 });
	  \draw (0,2)--(0.1,2) node[above left]{$Z_w$};
	  \draw (1.3,0.1)--(1.3,0) node[below left= -2pt, scale=1]{$\bar{h}-\ell$};
	  \draw (1.8,0.1)--(1.8,0) node[below=-2pt, scale=1]{$\bar{h}$};
	  \draw (2.3,0.1)--(2.3,0) node[below right= -2pt, scale=1]{$\bar{h}+\ell$};
	\end{tikzpicture}
	\caption{Mollified steady state of the intracellular CheZ.}\label{Fig-MAAI-CheZ}
\end{figure}
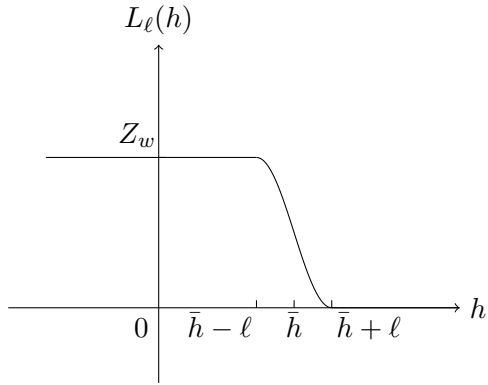

Notice that $L_\ell (h) \rightarrow L(h)$ as $\ell \rightarrow 0$, and for any integer $k \geq 1$,
\begin{equation}\label{equation 1.11}
  \begin{aligned}
    | \varphi_\ell^{(k)} (\varsigma) | \leq \tfrac{C}{\ell^k} \quad ( \, \forall \, \varsigma \in \R ) \,,
  \end{aligned}
\end{equation}
where $C > 0$ is some positive constant depending on $\kappa_0 > 0$. Thus, we easily derive that the scalar function $L_\ell : \omega \in \R \mapsto L_\ell (\omega) \in \R^+$ satisfies
\begin{equation}
  \begin{aligned}
    \sup_{\omega \in \R} | L_\ell^{(k)} (\omega) | \leq \tfrac{C}{\ell^k} <  \infty \,.
  \end{aligned}
\end{equation}
Here and in what following, the symbol $f^{(k)} (x)$ means the $k$-order derivative $\tfrac{\d^k}{\d x^k} f (x)$ for single variable functions. Moreover, for the smooth scalar function
\begin{equation}\label{g(z,w)}
  \begin{aligned}
    g_\ell (z, \omega) := k_V L_\ell (\omega) - k_V z \,,
  \end{aligned}
\end{equation}
set
\begin{equation}\label{equation 1.3}
\begin{aligned}
  \qquad \quad c_1 : = \sum_{i=1}^s \sup_{(z,\omega) \in \T_w \times \mathbb{R} } | \p_\omega^i g_\ell (z,\omega)  | \leq \sum_{i=1}^s C \ell^{-i} < \infty \,,
\end{aligned}
\end{equation}
where the constant $C > 0$ is independent of $\ell > 0$. Moreover, there holds
\begin{equation}\label{equation 1.4}
\begin{aligned}
	\p_z g_\ell (z, \omega) \equiv - k_V \,, \quad \p_z^l \p_\omega^m g_\ell (z, \omega) \equiv 0 \ \ \textrm{for any } l \geq 1 \ \textrm{ and } m \geq 0 \,,
\end{aligned}
\end{equation}

Furthermore, the diffusion coefficient $D (z)$ given in \eqref{D(z)} should also be imposed on some smooth hypotheses and positive lower assumptions. More precisely,
\begin{equation}\label{equation 1.5}
  \begin{aligned}
    a : = \sum_{i=0}^s \sup_{ z \in \T_w } | D^{(i)} (z) | < \infty \,,
  \end{aligned}
\end{equation}
\begin{equation}\label{equation 1.6}
  \begin{aligned}
    \qquad \qquad \qquad d : = \inf_{ z \in \T_w } D(z) > 0 , \, \ \int_{ \T_w } | D(z) - d |^2 \d z < \infty \,.
  \end{aligned}
\end{equation}
We then define a number $\eta$ as
\begin{equation}\label{equation 1.7}
  \begin{aligned}
    \eta : = \tfrac{\|D(z)-d\|_{L_z^2}^2}{d^2} \geq 0 \,.
  \end{aligned}
\end{equation}
Consequently, the function $\widetilde{D}_\ell (\zeta) : = D ( L_\ell (\zeta) )$ (the original discontinuous form of $\widetilde{D}$ given in \eqref{Tilde-D}) automatically satisfies
\begin{equation}\label{equation 1.12}
  \begin{aligned}
    b : = \sup_{ \omega \in \mathbb{R} } \sum_{i=0}^s | \widetilde{D}_\ell^{(i)} ( \omega ) | < \infty \,, \quad d = \inf_{\zeta \in \mathbb{R} } \widetilde{D}_\ell (\zeta) > 0 \,.
  \end{aligned}
\end{equation}
Since $L_\ell (\zeta)$ ranges in $[0, Z_w]$, exactly one period of $\T_w$, $\widetilde{D}_\ell$ and $D(z)$ have the same infimum $d > 0$. Remark that $b > 0$ depends on $\ell > 0$ and $b \to + \infty $ as $\ell \to 0$. Summarily, in this paper, we impose the following two hypothesis on the systems \eqref{EECPK} and \eqref{AD-EECP}:
\begin{enumerate}
	\item[\bf (H1)] The steady state $L (h)$ is replaced by its mollified form $L_\ell (h)$ given in \eqref{L(h)-m}. Therefore the scalar function $g (z,h)$ satisfies the bound \eqref{equation 1.3}.
	\item[\bf (H2)] The diffusion coefficient $D (z)$ in \eqref{EECPK} satisfies the assumptions \eqref{equation 1.5} and \eqref{equation 1.6}. Consequently, the diffusion coefficient $\widetilde{D}_\ell (\zeta) : = D ( L_\ell (\zeta) )$ subjects to \eqref{equation 1.12}.
\end{enumerate}
For simplicity, the subscript $\ell$ will be dropped in the rest of this paper.

\subsection{Notations and main results}
Before presenting our main results, we first gather the all notations and conventions used throughout this paper. We employ the notations $A \lesssim B$ to represent $A \leq C B$ for some harmless positive constant $C$. The symbol $A \thicksim B$ is employed if $C_1 B \leq A \leq C_2 B$ for some positive constants $C_1, C_2$. For any $p,q\in [1,\infty]$, we introduce the Banach spaces $L^p_x$, $L^q_z$ and $L^p_x L^q_z$ equipped with the norms
\begin{align*}
  & \| f \|_{L^{p}_x} = ( \int_{ \mathbb{T}^3 } |f|^p \d x )^{\frac{1}{p}} \,, \quad \| f \|_{L^{q}_z} = ( \int_{ \T_w } |f|^q \d z )^{\frac{1}{q}} \,, \\
  & \| f \|_{L^{p}_xL_z^q} = ( \int_{ \mathbb{T}^3 } \| f(x,\cdot) \|_{L_z^q}^p \d x )^{\frac{1}{p}}
\end{align*}
for $1 \leq p , q < \infty$, and if $p = \infty$ or $q = \infty$, the corresponding modifications shall be taken. For $p = q = 2$, we denote $L^2_{x,z} : = L^2_x L^2_z$ and use the notation $\l \cdot,\cdot \r_{L^2_x}$, $\l \cdot,\cdot \r_{L^2_z}$ and $\l \cdot,\cdot \r_{L^2_{x,z}}$ to represent the inner product on the Hilbert space $L_x^2$, $L_z^2$ and $L_{x,z}^2$, respectively.

In this paper, the symbol $\nabla_x$ stands for the gradient operator and $\Delta_x$ is the Laplacian operator. Moreover, for any multi-index $k=(k_1, k_2, k_3) \in \mathbb{N}^3$ and integer $ l \in \mathbb{N}$, we denote the higher order derivative operators
\begin{align*}
  \p_x^k : = \frac{ \p^{ |k| } }{ \p x_1^{k_1} \p x_2^{k_2} \p x_3^{k_3} } \,, \quad \p_z^l \p_x^k : = \frac{ \p^{ |k| + l } }{ \p x_1^{k_1} \p x_2^{k_2} \p x_3^{k_3} \p z^l} \,,
\end{align*}
where $|k|=k_1+k_2+k_3$. If each component of $k \in \mathbb{N}^3$ is not greater than that of $\tilde{k} \in \mathbb{N}^3$, we denote $k \leq \tilde{k}$. We further employ the symbol $k < \tilde{k}$ to indicate that $k \leq \tilde{k}$ and $|k| < |\tilde{k}|$. We then define the Sobolev spaces $H^s_x$, $H_{x}^{s} L^2_z$ and $H^s_{x,z}$ endowed with the norms
\begin{equation*}
  \begin{aligned}
    & \| \, \cdot \, \|_{H^s_x} : = ( \sum_{|k| \leq s} \| \p^k_x \, \cdot \, \|_{L^2_x}^2 )^{\frac{1}{2}} \,, \ \| \, \cdot \, \|_{H_{x}^{s} L^2_z} : = ( \sum_{ |k| \leq s } \| \p_x^k \, \cdot \, \|_{L_{x,z}^2}^2 )^{\frac{1}{2}} \,, \\
    & \| \, \cdot \, \|_{H_{x,z}^{s}} : = ( \sum_{ l + |k| \leq s } \| \p_z^l \p_x^k \, \cdot \, \|_{L_{x,z}^2}^2 )^{\frac{1}{2}} \,,
  \end{aligned}
\end{equation*}
respectively. Moreover, for the functions $D (z) \geq d > 0$ and $\widetilde{D} (x) = D \circ h (x) \geq d > 0$, we define the following weighted spaces
\begin{align*}
  \| \, \cdot \, \|_{L_{x,z}^2 (D)} & : = \| \cdot D^\frac{1}{2} (z) \|_{L^2_{x,z}} \,, \qquad \| \, \cdot \, \|_{L_x^2 ( \widetilde{D} ) } : = \| \, \cdot \, \widetilde{D} (h(x))^\frac{1}{2} \|_{L^2_x} \,,  \\
  \| \, \cdot \, \|_{H_{x,z}^s (D) } & : = ( \sum_{l + |k| \leq s} \| \p_z^l \p_x^k \, \cdot \, \|^2_{L^2_{x,z}} )^{\tfrac{1}{2}} \,, \quad \| \, \cdot \, \|_{H_x^s L_z^2(D) } : = ( \sum_{|k| \leq s} \| \p^k_x \, \cdot \, \|_{L^2_{x,z} (D)}^2 )^{\tfrac{1}{2}} \,, \\
  \| \, \cdot \, \|_{H_x^s(\widetilde{D})} & : = ( \sum_{|k| \leq s} \| \p_x^k \, \cdot \, \|^2_{L^2_x (\widetilde{D})} )^{\tfrac{1}{2}} \,.
\end{align*}

In this paper, we will mainly investigate the well-posedness of the kinetic system \eqref{EECPK} with initial data \eqref{IC-EECPK} and the macroscopic system \eqref{AD-EECP} with initial data \eqref{IC-AD-EECP}, including the local existence with large initial data and the global existence near the constant equilibria, which are such that the conservation laws of mass holds. In order to state the results on the system \eqref{EECPK}, we first introduce the energy functional
\begin{equation}\label{Es}
	\begin{aligned}
		\mathscr{E}_s (\rho, h, n) : = \| \rho \|^2_{H^s_{x,z}} + \| \int_{\T_w} \rho (\cdot, \cdot, z) \d z \|^2_{H^s_x} + \| h \|^2_{H^s_x} + \| n \|^2_{H^s_x} \,,
	\end{aligned}
\end{equation}
and the energy dissipative rate
\begin{equation}\label{Ds}
	\begin{aligned}
		\mathscr{D}_s (\rho, h, n) : = \| \nabla_x \rho \|^2_{H^s_{x,z} (D)} + \| \nabla_x \int_{\T_w} \rho (\cdot, \cdot, z) \d z \|^2_{H^s_x} + \| \nabla_x h \|^2_{H^s_x} + \| \nabla_x n \|^2_{H^s_x} + \| h \|^2_{H^s_x} \,.
	\end{aligned}
\end{equation}
For simplicity, we also employ the symbol
\begin{equation*}
	\begin{aligned}
		  \vr (t,x) : = \int_{\T_w} \rho (t, x, z) \d z \,.
	\end{aligned}
\end{equation*}
Now we precisely state our main theorem as follows:

\begin{theorem}[Well-posedness of \eqref{EECPK}]\label{Thm1}
  Let $s \geq 3$ be any fixed integer. Assume $g (z, \omega)$ and $D(z)$ satisfy the smoothness hypotheses $(\rm H1)$-$(\rm H2)$. All coefficients are assumed to be positive.
  \begin{enumerate}
  	\item {\bf (Local well-posedness with large data)} If the initial energy
       \begin{align}\label{IC-K1}
         E_L^{\eps,in} : = \mathscr{E}_s (\rho^{\eps, in}, h^{\eps, in}, n^{\eps, in}) < \infty \,,
       \end{align}
      then for any $0 < T < \tfrac{1}{C(1+\eps^{-1})} \ln \tfrac{[1 + (E_L^{\eps, in})^\frac{s}{2}]^\frac{2}{s}}{E_L^{\eps, in}} $, the Cauchy problem \eqref{EECPK}-\eqref{IC-EECPK} admits a unique solution $ ( \rho (t,x,z), h(t,x), n (t,x) ) $ with
      \begin{align*}
        & \rho \in L^\infty(0,T;H_{x,z}^{s}) \,, \ \nabla_x \rho \in L^2 (0,T; H_{x,z}^s (D)) \,, \\
        & \vr \,, h \,, n \in L^\infty(0, T; H^s_x) \cap L^2 (0, T; H^{s+1}_x) \,.
      \end{align*}
      Moreover, the following energy bound holds:
      \begin{align*}
        \sup_{0 \leq t \leq T} \mathscr{E}_s (\rho, h, n) + \int_0^T \mathscr{D}_s (\rho, h, n) \d t \leq \widetilde{B}_\eps (T, E_L^{\eps, in}, C) < \infty \,.
      \end{align*}
      Here the constant $C=C(k_V, c_1, c_2, a, d, \alpha, \beta, \gamma, \xi, s)>0$ is independent of $\eps$ (given in Lemma \ref{lemma 2.1} below).

      \item {\bf (Positivity and conservation laws)} If the initial data further satisfy
      \begin{equation}\label{IC-K2}
      	\begin{aligned}
      		\rho^{\eps, in} (x,z) \,, h^{\eps, in} \,, n^{\eps, in} \geq 0 \,,
      	\end{aligned}
      \end{equation}
      then the solution $(\rho, h, n)$ constructed above obeys the positivity
      \begin{equation}
      	  \begin{aligned}
      	  	  \rho (t, x,z), h(t,x), n(t,x) \geq 0 \,.
      	  \end{aligned}
      \end{equation}
      Furthermore, if the initial data are also assumed
      \begin{equation}\label{IC-Average}
      	\begin{aligned}
      		\iint_{\T^3 \times \T_w} \rho^{\eps, in} (x,z) \d z \d x = (2 \pi)^3 Z_w \rho_0 \geq 0 \,, \ \int_{\T^3} n^{\eps, in} (x) \d x =  0
      	\end{aligned}
      \end{equation}
      for any constant $\rho_0 \geq 0$, then there hold
      \begin{equation}\label{Consv-KEECP}
      	\begin{aligned}
      		\int_{\T^3} \vr (t,x) \d x = (2 \pi)^3 Z_w \rho_0 \geq 0 \,, \ \int_{\T^3} n (t,x) \d x = 0 \quad \textrm{ for all } t \in [0,T] \,.
      	\end{aligned}
      \end{equation}

      \item {\bf (Global existence around $(0,0,0)$)} Under the initial assumptions \eqref{IC-K1}-\eqref{IC-K2}-\eqref{IC-Average} with $\rho_0 = 0$, the solution $(\rho, h, n)$ around the trivial steady state $(0,0,0)$ globally exists. Moreover, the solution $(\rho, h, n)$ satisfies $\rho (t,x,z) = n (t,x) \equiv 0$ and
      \begin{equation}
      	\begin{aligned}
      		\sup_{ t \geq 0 } \, \mathscr{E}_s (\rho, h, n) (t) + \int_0^\infty \mathscr{D}_s (\rho , h , n) (\varsigma) \d \varsigma \leq \mathcal{C}_1 \mathscr{E}_s (\rho^{\eps, in}, h^{\eps, in}, n^{\eps, in})
      	\end{aligned}
      \end{equation}
      for some uniform constant $\mathcal{C}_1 > 0$.
  \end{enumerate}
\end{theorem}

%\begin{remark}\label{Rmk1}
%	When proving the small global existence, the further assumption \eqref{CH-2} is required to deal with difficulties coming from the term $- \tfrac{1}{\eps} \p_z ( g(z, h) \rho )$ in the right-hand side of the first equation of \eqref{EECPK} system. More precisely, since $g (z, h) = k_V L (h) - k_V z$, we see that
%	\begin{equation*}
%	  \begin{aligned}
%	    - \tfrac{1}{\eps} \p_z ( g(z, h) \rho ) = \tfrac{k_V}{\eps} \p_z ( z \rho ) - \tfrac{k_V}{\eps} L (h) \p_z \rho \,,
%	  \end{aligned}
%	\end{equation*}
%	where the linear term $\tfrac{k_V}{\eps} \p_z ( z \rho ) = \tfrac{k_V}{\eps} \rho + \tfrac{k_V}{\eps} z \p_z \rho $, appearing in the right-hand side of first equation of \eqref{EECPK}, can only be absorbed by the diffusion term $D (z) \Delta_x \rho$ provided that the assumption \eqref{CH-2} holds.
%\end{remark}

\begin{remark}\label{Rmk1-1}
	There is a natural goal to prove the global-in-time solution around the steady space-homogeneous states to the system \eqref{EECPK}. Now we try to seek the steady space-homogeneous solutions with form $(\rho_* (z), h_*, n_*)$, where $\rho_* (z) \geq 0$ and the nonnegative constants $h_*$, $n_*$ are to be determined. Since the smooth solutions (at least continuous solutions) are considered in this paper, $\rho_* (z)$ is also assumed to be continuous. Plugging $(\rho_* (z), h_*, n_*)$ into \eqref{EECPK} reduces to
	\begin{equation}\label{Steady-*}
		\left\{
		  \begin{aligned}
		  	  - \tfrac{k_V}{\eps} \partial_z [ (L(h_*) - z) \rho_* (z) ] + \gamma n_* \rho_* (z) = 0 \,, \\
		  	  \alpha \int_{ \T_w } \rho_* (z) \d z - \beta h_* = 0 \,, \\
		  	  - \xi \int_{ \T_w } \rho_* (z) \d z \cdot n_* = 0 \,.
		  \end{aligned}
		\right.
	\end{equation}
The last equality means $n_* = 0$ or $\int_{ \T_w } \rho_* (z) \d z = 0$.

{\bf Case 1. $n_* = 0$.} The first equation in \eqref{Steady-*} means
\begin{equation*}
	\begin{aligned}
		( L(h_*) - z ) \rho_* (z) = c_*
	\end{aligned}
\end{equation*}
for some undetermined constant $c_*$ and any $z \in \T_w$. Recall $L (h_*) \in [0, Z_w]$ (see \eqref{L(h)} above). Choosing $z = L(h_*) \in \T_w$ in the previous equality yields that $c_* = 0$. Namely, $( L(h_*) - z ) \rho_* (z) = 0$ holds for any $z \in \T_w$, which means that $\rho_* (z) \equiv 0$. The second equality in \eqref{Steady-*} thereby implies $h_* = 0$.

{\bf Case 2. $\int_{ \T_w } \rho_* (z) \d z = 0$.} The positivity and continuity of $\rho_* (z) \geq 0$ imply $\rho_* (z) \equiv 0$. There therefore holds $h_* = 0$ via the second equality in \eqref{Steady-*}.

In summary, the steady space-homogeneous solutions are of the forms $(0,0,n_*)$, where the constant $n_* \geq 0$.

The goal is to analyze the long time existence near the steady space-homogeneous solutions $(0, 0, n_*)$ with $n_* \geq 0$. As shown in Appendix \ref{Appendix} below, the system \eqref{EECPK} is linearly unstable around $(0,0,n_*)$ while $n_* > 0$. It remains to consider the long time existence near the trivial state $(0,0,0)$. As the similar arguments on the long time existence around the system \eqref{AD-EECP} in Theorem \ref{Thm3} below in the energy method framework, there are two quantities $\gamma \iint_{\T^3 \times \T_w} [\rho (t,x,z)]^2 n (t,x) \d z \d x$ and $- \xi \iint_{\T^3 \times \T_w} \rho (t,x,z) [ n (t,x) ]^2 \d z \d x$, which should be controlled by $\mathscr{E}_s^\frac{1}{2} (\rho, h, n) \mathscr{D}_s (\rho, h, n)$. In this process, the averages $ \frac{1}{|\T^3 \times \T_w|} \iint_{\T^3 \times \T_w} \rho (t,x,z) \d z \d x$ and $ \frac{1}{|\T^3|} \int_{ \T^3 } n (t,x) \d x$ shall be coincided with the constant steady state $(\rho_* (z)=0, h_* = 0, n_* = 0)$, which ensure the validity of Poincar\'e inequality. Namely, one needs
\begin{equation} \label{Average}
	\begin{aligned}
		\iint_{\T^3 \times \T_w} \rho (t,x,z) \d z \d x = \int_{ \T^3 } n (t,x) \d x = 0 \,.
	\end{aligned}
\end{equation}
The previous equality is derived from \eqref{Consv-KEECP} with $\rho_0 = 0$ under the initial conditions \eqref{IC-K1}, \eqref{IC-K1} and \eqref{IC-Average}.

However, once the zero-average condition \eqref{Average} holds, the positivity of $(\rho, h, n)$ yields that $\rho (t,x,z) = n(t,x) \equiv 0$ and $h (t,x)$ subjects to
\begin{equation}\label{Linear-h}
	\begin{aligned}
		\partial_t h = D_h \Delta_x h - \beta h \,, \ h (0,x) = h^{\eps, in} (x) \geq 0 \,.
	\end{aligned}
\end{equation}
Consequently, under the initial conditions \eqref{IC-K1}-\eqref{IC-K2}-\eqref{IC-Average} with $\rho_0 = 0$, the result of long time existence around the trivial steady solution $(0,0,0)$ is reasonable and foreseeable. Nevertheless, the global-in-time solution around $(0,0,0)$ without initial assumption \eqref{IC-Average} is still open.
\end{remark}

\begin{remark}\label{Rmk1-2}
	We now focus on the limit from \eqref{EECPK} to \eqref{AD-EECP} system as $\eps \to 0$. The formal derivation of the limit has been given in Subsection \ref{Subsec:Formal} in the weak distribution sense. In the part (1) of Theorem \ref{Thm1}, the local existence time $T$ has a upper bound
	\begin{equation*}
		\begin{aligned}
			T_\eps : = \tfrac{1}{C(1+\eps^{-1})} \ln \tfrac{[1 + (E_L^{\eps, in})^\frac{s}{2}]^\frac{2}{s}}{E_L^{\eps, in}} > 0 \,.
		\end{aligned}
	\end{equation*}
    If the initial energy $E_L^{\eps, in}$ is uniformly bounded in $\eps > 0$,
    \begin{equation*}
    	\begin{aligned}
    		T_\eps \to 0 \quad \textrm{as} \ \eps \to 0 \,.
    	\end{aligned}
    \end{equation*}
    This fails to pass the limit from \eqref{EECPK} system to \eqref{AD-EECP} system as $\eps \to 0$ in the smooth solution regime. The reason is that the formal derivation above depends on the Dirac distribution function. However, the smooth regime excludes exactly the all discontinuous cases. In this sense, the rigorous justification of the limit must be in the weak solution regime. This will be appeared in our separated work.
\end{remark}

Next, we investigate the well-posedness of the system \eqref{AD-EECP} with initial data \eqref{IC-AD-EECP} and study the long time existence of the solutions. In order to state simply the results on the system \eqref{AD-EECP}, one defines the following energy functional
\begin{equation}
	\begin{aligned}
		\mathbf{E}_s (\vr, h, n) : = \| \vr \|_{H^s_x}^2 + \| h \|_{H^s_x}^2 + \| n \|_{H^s_x}^2
	\end{aligned}
\end{equation}
and energy dissipative rate
\begin{equation}
	\begin{aligned}
		\mathbf{D}_s (\vr, h, n) : = \| \nabla_x \vr \|_{H^s_x ( \widetilde{D} ) }^2 + \| \nabla_x h \|_{H^s_x}^2 + \| h \|_{H^{s}_{x}}^2 + \| \nabla_x n \|_{H^s_x}^2 + \| n \|_{H^s_x}^2 \,.
	\end{aligned}
\end{equation}
We also define an additional energy functional
\begin{equation}
	\begin{aligned}
		\mathbb{E}_s (h) : = \| \widetilde{D} (h) \|^2_{H^s_x} \,.
	\end{aligned}
\end{equation}
It is easy to see that $\mathbb{E}_s (h) \leq C \mathbf{E}_s (\vr, h, n)$ under the hypothesis (H2).

We also should clarify the steady space-homogeneous solutions $(\vr_a, h_a, n_a)$, where the constants $\vr_a, h_a, n_a \geq 0$ is to be determined. Substituting $(\vr_a, h_a, n_a)$ into \eqref{AD-EECP} yields that
\begin{equation*}
	\left\{
	  \begin{aligned}
	  	  \gamma n_a \vr_a = 0 \,, \\
	  	  \alpha \vr_a - \beta h_a = 0 \,, \\
	  	  - \xi \vr_a n_a = 0 \,.
	  \end{aligned}
	\right.
\end{equation*}
Hence, the steady space-homogeneous states $(\vr_a, h_a, n_a)$ of \eqref{AD-EECP} system must satisfy
\begin{equation*}
	\begin{aligned}
		n_a = 0 \,, h_a = \frac{\alpha}{\beta} \vr_a \geq 0 \,, \ \textrm{ or } \vr_a = h_a = 0 \,, n_a \geq 0 \,.
	\end{aligned}
\end{equation*}
As shown in Appendix, the solution $(\vr, h, a)$ around $(0,0,n_a)$ with $n_a > 0$ is linear unstable. Thus, we will prove the long time existence of the solutions around $(\vr_a, h_a, 0)$ with $h_a = \frac{\alpha}{\beta} \vr_a \geq 0$. However, not all $\vr_a \geq 0$ makes the \eqref{AD-EECP} system dissipative near $(\vr_a, h_a, 0)$. As shown in Lemma \ref{Lmm-Dissi} below, only if $\vr_a \in [0, \Lambda_b)$, the system will be dissipated near $(\vr_a, h_a, 0)$, where
\begin{equation}\label{Lambda-b}
	\begin{aligned}
		\Lambda_b : = \tfrac{4 \beta d^\frac{3}{2} D_h}{\alpha^2 \mathcal{C}_p^2 b} > 0 \,.
	\end{aligned}
\end{equation}
Here the constant $\mathcal{C}_p > 0$ is given in Lemma \ref{Lmm-Dissi}. The similar ideas can be found in \cite{JLZ-2021-CMS}.

Now we state our main results on the \eqref{AD-EECP} system as follows.

\begin{theorem}[Well-posedness of \eqref{AD-EECP}]\label{Thm3}
  Let $s \geq 4$ be any fixed integer. Assume that $\widetilde{D} (\cdot)$ satisfy the hypothesis $({\bf H2})$. All coefficients are assumed to be positive.
  \begin{enumerate}
  	\item {\bf (Local well-posedness with large data)} If the initial energy
      \begin{equation}\label{IC-A1}
        \begin{aligned}
          E_l^{in} : = \mathbf{E}_s (\vr_a^{in}, h_a^{in}, n_a^{in}) < \infty \,,
        \end{aligned}
     \end{equation}
    then there exists a number $T = T ( E_l^{in}, s, d, b, \alpha, \beta, \gamma, \xi ) > 0 $ such that the Cauchy problem \eqref{AD-EECP}-\eqref{IC-AD-EECP} admits a unique solution $ (\vr, h, n ) (t,x)$ on the interval $[0,T]$ satisfying
    \begin{align*}
      \vr, h, n \in L^\infty (0,T; H_{x}^s ) \cap L^2(0,T; H_x^{s+1} ) \,.
    \end{align*}
    Moreover, there holds
    \begin{align*}
      \sup_{ 0 \leq t \leq T } \mathbf{E}_s (\vr, h, n) + \int_0^T \mathbf{D}_s (\vr, h, n) \d t \leq B (E_l^{in}, T) \,.
    \end{align*}

    \item {\bf (Positivity and conservation laws)} If the initial data further satisfy
    \begin{equation}\label{IC-A2}
    	\begin{aligned}
    		\vr_a^{in} (x) \,, h_a^{in} (x) \,, n_a^{in} (x) \geq 0 \,,
    	\end{aligned}
    \end{equation}
    then the solution $(\vr, h, n)$ constructed above obeys the positivity
    \begin{equation}
    	\begin{aligned}
    		\vr (t, x), h(t,x), n(t,x) \geq 0 \,.
    	\end{aligned}
    \end{equation}
    Furthermore, if the initial data are also assumed
    \begin{equation}\label{IC-Average-a}
    	\begin{aligned}
    		\int_{\T^3} \vr_a^{in} (x) \d x = (2 \pi)^3 \vr_0 \geq 0 \,, \ \int_{\T^3} n_a^{in} (x) \d x =  0
    	\end{aligned}
    \end{equation}
    for any constant $\vr_0 \geq 0$, then there hold
    \begin{equation}\label{Consv-ADEECP}
    	\begin{aligned}
    		\int_{\T^3} \vr (t,x) \d x = (2 \pi)^3 \vr_0 \geq 0 \,, \ \int_{\T^3} n (t,x) \d x = 0 \quad \textrm{ for all } t \in [0,T] \,.
    	\end{aligned}
    \end{equation}

    \item {\bf (Long time existence around $(\vr_a, h_a, 0)$)} Consider the steady states $(\vr_a, h_a, 0)$ of \eqref{AD-EECP} with $h_a = \frac{\alpha}{\beta} \vr_a$ and $0 \leq \vr_a < \Lambda_b$, where the constant $\Lambda_b > 0$ is given in \eqref{Lambda-b}. Under the initial assumptions \eqref{IC-A1}-\eqref{IC-A2}-\eqref{IC-Average-a} with $\vr_0 = \vr_a \in [0, \Lambda_b )$, there is a small constant $\vartheta_0 > 0$, depending only on the all coefficients and $s$, such that if the initial energy
    \begin{align}\label{equation 1.18}
      \mathbf{E}_s^{in} : = \mathbf{E}_s ( \vr^{in}_a - \vr_a, h^{in}_a - h_a , n^{in}_a) \leq \vartheta_0 \,,
    \end{align}
    then the solution $(\vr, h, n) (t, x)$ to the Cauchy problem \eqref{AD-EECP}-\eqref{IC-AD-EECP} constructed in the part $(1)$ can be globally extended. Moreover, $n(t,x) \equiv 0$ and there holds
   \begin{align*}
	  \sup_{ t \geq 0 } \mathbf{E}_s ( \vr - \vr_a, h - h_a, n ) + \int_0^\infty \mathbf{D}_s ( \vr - \vr_a, h - h_a, n ) \d t \leq \mathcal{C}^*_2 \mathbf{E}_s^{in}
   \end{align*}
   for some constant $ \mathcal{C}^*_2 > 0 $, depending only on the all coefficients and $s$.
  \end{enumerate}
\end{theorem}

\begin{remark}\label{Rmk1-3}
	When constructing the large local existence of the system \eqref{AD-EECP}, we need an extra energy inequality associated with the evolution of $\widetilde{D} (h)$ (see \eqref{equation 5.5}). In deriving the extra energy inequality, there is quantity $K_{43} = 2 D_h \l \widetilde{D} '' (h) \p^k_x h \Delta_x h , \p^k_x \widetilde{D} (h) \r$ (see \eqref{K43}), which will be bounded by $\| h \|_{H^s_x} \| \Delta_x h \|_{H^2_x} ( 1 + \| h \|^{s-1}_{H^{s-1}_x} ) \| \nabla_x h \|_{H^{s-1}_x} \lesssim (1 + \| h \|^{s-1}_{H^s_x}) \| h \|^2_{H^s_x}$. Here we will use the inequality $\| \Delta_x h \|_{H^2_x} \leq \| h \|_{H^s_x}$, so the Sobolev index $s \geq 4$ is required.
\end{remark}

\begin{remark}\label{Rmk1-4}
    As shown in Lemma \ref{Lmm-AD-PhiPsi} below, there are two quantities $\gamma \int_{ \T^3 } [\vr (t,x) - \vr_a]^2 n(t,x) \d x$ and $- \xi \int_{ \T^3 } [\vr (t,x) - \vr_a] [n (t,x)]^2 \d x$, which should be bounded by $\mathbf{E}_s^\frac{1}{2} (\vr - \vr_a, h-h_a, n) \mathbf{D}_s (\vr - \vr_a, h-h_a, n)$ when proving the long time existence of \eqref{AD-EECP} around the steady states $(\vr_a, h_a, 0)$ with $h_a = \frac{\alpha}{\beta} \vr_a$ and $\vr_a \in [0, \Lambda_b )$ in the energy method framework. In this process, the averages $\frac{1}{|\T^3|} \int_{ \T^3 } \vr (t,x) \d x$ and $\frac{1}{|\T^3|} \int_{ \T^3 } n (t,x) \d x$ should be consistent with the steady states $(\vr_a, h_a, 0)$ above, which ensure the validity of Poincar\'e inequality applied in functions $\vr (t,x) - \vr_a$ and $n(t,x)$. Hence one needs
    \begin{equation}\label{Average-a}
    	\begin{aligned}
    		\int_{ \T^3 } \vr (t,x) \d x = (2 \pi)^3 \vr_a \geq 0 \,, \int_{ \T^3 } n (t,x) \d x = 0 \,.
    	\end{aligned}
    \end{equation}
The previous conserved equality is implied by \eqref{Consv-ADEECP} with $\vr_0 = \vr_a \geq 0$ under the initial conditions \eqref{IC-A1}, \eqref{IC-A2} and \eqref{IC-Average-a}. But the zero-average condition $\int_{ \T^3 } n (t,x) \d x = 0$ in \eqref{Average-a} and the positivity $n (t,x) \geq 0$ imply that $n(t,x) \equiv 0$ and $(\vr, h) (t,x)$ subjects to
\begin{equation}\label{rho-h}
	\left\{
	  \begin{aligned}
	  	\p_t \vr = \Delta_x ( \widetilde{D} (h) \vr )  \,, \\
	  	\p_t h = D_h \Delta_x h + \alpha \vr - \beta h \,.
	  \end{aligned}
	\right.
\end{equation}
Therefore, under the initial conditions \eqref{IC-A1}-\eqref{IC-A2}-\eqref{IC-Average-a} with $\vr_0 = \vr_a \in [0, \Lambda_b )$, the long time existence of \eqref{AD-EECP} around the steady states $(\vr_a, h_a, 0)$ above is reasonable. However, the long time stability of $(\vr_a, h_a, 0)$ without initial assumption \eqref{IC-Average-a} is still open.

Furthermore, if one sets $\vr_a = 0$, which means $h_a = 0$, then positivity and conservation laws in Theorem \ref{Thm3} also imply that $\int_{ \T^3 } \vr (t,x) \d x = 0$ and the continuous function $\vr (t,x) \geq 0$. It thereby infers that $\vr (t,x) \equiv 0$ also holds and $h (t,x)$ obeys
\begin{equation*}
	\begin{aligned}
		\p_t h = D_h \Delta_x h - \beta h \,.
	\end{aligned}
\end{equation*}
This corresponds to the case of \eqref{EECPK} system around the steady state $(0,0,0)$, see Theorem \ref{Thm1}.
\end{remark}

\begin{remark}\label{Rmk1-5}
	Formally, as in Subsection \ref{Subsec:Formal}, \eqref{AD-EECP} system is the limit of \eqref{EECPK} system when the response speed $\eps$ of intracellular CheZ to the external signal AHL goes to zero. In Theorem \ref{Thm1}, the global-in-time solution to \eqref{EECPK} equations under initial assumptions \eqref{IC-K1}-\eqref{IC-K2}-\eqref{IC-Average} around ONLY the steady state $(0,0,0)$ is proved. However, in Theorem \ref{Thm3}, the limit \eqref{AD-EECP} equations admit the global-in-time solutions under initial hypotheses \eqref{IC-A1}-\eqref{IC-A2}-\eqref{IC-Average-a} around the steady states $(\vr_a, h_a, 0)$ with $h_a = \frac{\alpha}{\beta} \vr_a$ and $ \vr_a \in [0, \Lambda_b ) $. It is intuitively confused that, compared with \eqref{EECPK} equations, the \eqref{AD-EECP} model admits some more global solutions near non-zero steady states $(\vr_a, h_a, 0)$ with $h_a = \frac{\alpha}{\beta} \vr_a$ and $0 < \vr_a < \Lambda_b $. The reason is the assumed smooth approximate of $L(h)$ in \eqref{L(h)-m}, which leads to the well-defined constant $b > 0$ in \eqref{equation 1.12}. Actually, $b > 0$ depends on $\ell > 0$, the modified length near the jumped point $\bar{h}$ of $L(h)$, see \eqref{L(h)-m} or Figure \ref{Fig-MAAI-CheZ}. Moreover, $b \to + \infty$ as $\ell \to 0$, which implies that $\Lambda_b \to 0$ as $\ell \to 0$. Together with $L_\ell (h) \to L (h)$ as $\ell \to 0$, the global solutions of \eqref{AD-EECP} near non-zero steady states $(\vr_a, h_a , 0)$ with $h_a = \frac{\alpha}{\beta} \vr_a$ and $\vr_a \in (0, \Lambda_b)$ are consistent with that of \eqref{EECPK} near $(0,0,0)$ in the sense of asymptotic behavior $\ell \to 0$.
\end{remark}

%%%%%%%%%%%%%%%%%%%%%%%%%%%%%%%%%%%%%%%%%%%%%%%%%%%%%%%%%%%%%%%%%%%%%%%%%%%%%%%%%%%%%%%%%%%%%%%%%%%%%%%%%%%%%%%%%%%%

\subsection{Main ideas and sketch of proofs}

For the \eqref{EECPK} model, one of the main difficulties come from the $z$-flux term $- \tfrac{1}{\eps} \p_z (g (z, h) \rho)$ associated with the intracellular dynamic of CheZ protein. This term is the main novelty of the model proposed in \cite{XXT-2018-CB}. In particular, we need $(x,v)$-derivative energy estimates. This is different with many kinetic equation, for example, the classical Boltzmann equation, for which the pure spacial derivative energy estimate is available, see \cite{Guo-Indiana} and \cite{JXZ-Indiana}.

In the proof of positivity, the similar approaches to justify the positivity of weak solutions of reaction kinetic equations in Perthame's book \cite{Perthame-2015-BOOK} are employed. More precisely, based on the continuous solution constructed in part (1) of Theorem \ref{Thm1}, let $\Omega_-$ be the set of the points $(t,x,z)$ such that $\rho (t,x,z) < 0$ (see \eqref{Omega-}, Subsection \ref{Subsec-2-3} below). By contradiction arguments, the continuity of $\rho (t,x,z)$ with positive initial data $\rho^{\eps, in} \geq 0$ and the Gr\"onwall inequality imply $\Omega_- = \emptyset$, which concludes $\rho (t,x,z) \geq 0$. The positivity of $h (t,x)$ and $n (t,x)$ can be proved by directly employing Lemma \ref{Lmm-Perthame} constructed in \cite{Perthame-2015-BOOK}.

A basic physical view is to study the conservation laws of a system, such as mass, moment and energy. In the \eqref{EECPK} system, $\iint_{ \T^3 \times \T_w } \rho (t,x,z) \d z \d x + \frac{\gamma}{\xi} \int_{ \T^3 } n (t,x) \d x$ is the only conserved quantity without any assumption, i.e.,
\begin{equation*}
	\begin{aligned}
		\tfrac{\d}{\d t} \left( \iint_{ \T^3 \times \T_w } \rho (t,x,z) \d z \d x + \tfrac{\gamma}{\xi} \int_{ \T^3 } n (t,x) \d x \right) = 0 \,,
	\end{aligned}
\end{equation*}
derived from the constitutive of system. However, this summation cannot be separated directly. Under the positivity of $\rho$ and $n$, $\frac{\d}{\d t} \int_{ \T^3 } n (t,x) \d x \leq 0$. Together with the initial hypothesis $\int_{ \T^3 } n^{\eps, in} (x) \d x = 0$, one concludes $\int_{ \T^3 } n (t,x) \d x = 0$. Then the total mass of $\rho (t,x,z)$ is conserved.

As illustrated in Remark \ref{Rmk1-1}, $(0,0,0)$ is the unique continuous steady state such that the linearized model of \eqref{EECPK} system near $(0,0,0)$ is stable (see Appendix \ref{Appendix} below). However, while proving the global-in-time solution of \eqref{EECPK} near $(0,0,0)$, the conditions $\iint_{\T^3 \times \T_w} \rho (t,x,z) \d z \d x = \int_{ \T^3 } n (t,x) \d x = 0$ are required, see Remark \ref{Rmk1-1} above. Together with positivity of the continuous functions $\rho$ and $h$, $\rho (t,x,z) = n (t,x) = 0$. In this situation, the \eqref{EECPK} system reduces to the linear parabolic equation \eqref{Linear-h} of $h$. Therefore, the global existence of \eqref{EECPK} equations around $(0,0,0)$ can be constructed under initial hypotheses \eqref{IC-K1}-\eqref{IC-K2}-\eqref{IC-Average} with $\rho_0 = 0$.

For the Cauchy problem \eqref{AD-EECP}-\eqref{IC-AD-EECP}, the nonlinearity of the diffusion term $\Delta_x ( \widetilde{D} (h) \vr )$ will generate some difficulties in proving the local existence with large initial data. Generally speaking, in order to prove the local well-posedness with large initial data by employing the Gr\"onwall inequality, one often should obtain the following type energy differential inequality
\begin{equation*}
  \begin{aligned}
    \tfrac{\d}{\d t} E_l (t) + D_l (t) \leq P (E_l (t))
  \end{aligned}
\end{equation*}
for some strictly increasing function $P (\cdot) : \R^+ \rightarrow \R^+$ with $P (0) = 0$. However, it is difficult to obtain this type of inequality to the \eqref{AD-EECP} equations. One will specifically show this difficulty based on the $L^2$-estimates, since the higher order derivatives estimates correspond to the similar structures. More precisely, when taking $L^2$-inner product of the diffusion term $\Delta_x ( \widetilde{D} (h) \vr )$ by dot with $\vr$, there will generate a unsigned quantity $\l \nabla_x \widetilde{D} (h) \cdot \vr, \nabla_x \vr \r_{L^2_x}$ (see \eqref{l1}). A simple observation implies that the unsigned quantity can be bounded by $\| \vr \|_{L^\infty_x} \| \nabla_x h \|_{L^2_x} \| \nabla_x \vr \|_{L^2_x (\widetilde{D})}$, which corresponds to the quantity $\| \vr \|_{L^\infty_x} \| \nabla_x \p^k_x h \|_{L^2_x} \| \nabla_x \p^k_x \vr \|_{L^2_x (\widetilde{D})}$ in the higher order derivatives estimates for $1 \leq |k| \leq s$. The norm $\| \rho \|_{L^\infty_x}$ can be dominated by the energy $E_l^\frac{1}{2} (t)$ and $\| \nabla_x h \|_{L^2_x} \| \nabla_x \vr \|_{L^2_x (\widetilde{D})}$ can only be controlled by the dissipative rate $D_l (t)$ (see \eqref{AD-Loc-Ener}). If one does as the way above, the following type energy differential inequality can be obtained
\begin{equation}\label{Abst-0}
  \begin{aligned}
    \tfrac{\d}{\d t} E_l (t) + D_l (t) \leq E_l^\frac{1}{2} (t) D_l (t) + P (E_l (t)) \,.
  \end{aligned}
\end{equation}
In order to control the quantity $E_l^\frac{1}{2} (t) D_l (t)$, the small size of $E_l^\frac{1}{2} (t) $ must be required. Therefore, the small initial data should be imposed, which is NOT what we expected in the proof of local existence. One of our {\em novelties} is to seek an extra energy inequality to overcome this difficulty. To be more precise, we introduce extra energy $\mathbb{E}_l (t) = \| \widetilde{D} (h) \|^2_{H^s_x}$ and extra dissipative rate $\mathbb{D}_l (t) = D_h \| \nabla_x \widetilde{D} (h) \|^2_{H^s_x}$. Then the unsigned quantity $\l \nabla_x \widetilde{D} (h) \cdot \vr, \nabla_x \vr \r_{L^2_x}$ can be bounded by $E_l^\frac{1}{2} (t) D_l^\frac{1}{2} (t) \mathbb{D}_l^\frac{1}{2} (t)$ ($\leq \tfrac{1}{2} D_l (t) + \tfrac{1}{2} E_l (t) \mathbb{D}_l (t)$), so that we obtain the following type energy inequality
\begin{equation}\label{Abst-1}
  \begin{aligned}
    \tfrac{\d}{\d t} E_l (t) + D_l (t) \leq E_l (t) \mathbb{D}_l (t) + P (E_l (t)) \,.
  \end{aligned}
\end{equation}
One further focuses on the $\widetilde{D} (h)$-evolution in \eqref{equation 5.5}, i.e.,
\begin{equation*}
  \begin{aligned}
    \p_t \widetilde{D} (h) = D_h \Delta_x \widetilde{D} (h) - D_h \widetilde{D}''(h) |\nabla_x h|^2 - \beta \widetilde{D}' (h) h + \alpha \vr \widetilde{D}' (h) \,.
  \end{aligned}
\end{equation*}
One can thereby derive the extra energy inequality
\begin{equation}\label{Abst-2}
  \begin{aligned}
    \tfrac{\d}{\d t} \mathbb{E}_l (t) + \mathbb{D}_l (t) \leq Q (E_l (t)) \,.
  \end{aligned}
\end{equation}
The Sobolev index $s \geq 4$ is actually required here. Thanks to the energy inequalities \eqref{Abst-1} and \eqref{Abst-2}, one can derive the local existence with large initial data, for details see Subsection \ref{Sec: Local-Result to the limit equation} below. We emphasize that the inequalities \eqref{Abst-1} and \eqref{Abst-2} cannot be simply added together. Otherwise, we will get an inequality similar to the structure of the inequality \eqref{Abst-0}.

For the positivity and conservation laws of the \eqref{AD-EECP} system, the arguments are the almost same as that in the \eqref{EECPK} system. The continuous steady states of \eqref{AD-EECP} are $(\vr_a, h_a, 0)$ with $h_a = \frac{\alpha}{\beta} \vr_a$ and $(0,0,n_a)$ with $n_a > 0$. As shown in Appendix \ref{Appendix}, the linearized model around $(0,0,n_a)$ is unstable. Our goal is therefore to prove the global existence of the \eqref{AD-EECP} system near $(\vr_a, h_a, 0)$ with $h_a = \frac{\alpha}{\beta} \vr_a \geq 0$. As shown in Remark \ref{Rmk1-4}, the conditions $\int_{ \T^3 } \vr (t,x) - \vr_a \d x = \int_{ \T^3 } n (t,x) \d x = 0$ are needed. These conditions hold under initial hypotheses \eqref{IC-A2}-\eqref{IC-Average-a} with $\vr_0 = \vr_a \geq 0$. By the further continuity and positivity of $n (t,x)$, one has $n (t,x) \equiv 0$. Then the \eqref{AD-EECP} model reduces to \eqref{rho-h} for $(\vr, h)$. By defining the fluctuation \eqref{Fluctuation}, i.e., $\vr (t,x) = \vr_a + \Phi (t,x) \,, \ h (t,x) = h_a + \Psi (t,x) $, $( \Phi, \Psi )$ subjects to the equations \eqref{AD-PhiPsi}. However, not all $\vr_a \geq 0$ makes the $(\Phi, \Psi)$-equations dissipative. Lemma \ref{Lmm-Dissi} indicates that only if $0 \leq \vr_a < \Lambda_b$ ($\Lambda_b$ is defined in \eqref{Lambda-b}), the equations \eqref{AD-PhiPsi} will be dissipated. At the end, in Lemma \ref{Lmm-AD-PhiPsi}, the induction approach for the orders of derivatives shall be employed to close the energy differential inequality, namely to obtain the inequality
\begin{equation*}
  \begin{aligned}
    \tfrac{1}{2} \tfrac{\d}{\d t} E_g (t) + D_g (t) \leq C ( 1 + E_g^\frac{s-1}{2} (t) ) E_g^\frac{1}{2} (t) D_g (t)
  \end{aligned}
\end{equation*}
for some instant energy functional $E_g (t)$ and dissipative rate $D_g (t)$. Based on the continuity arguments, one can thereby verify the small global well-posedness of \eqref{AD-EECP} system near $(\vr_a, h_a, 0)$ with corresponding small initial perturbed data under the hypotheses \eqref{IC-A1}-\eqref{IC-A2}-\eqref{IC-Average-a} with $\vr_0 = \vr_a \in [0, \Lambda_b)$.

%%%%%%%%%%%%%%%%%%%%%%%%%%%%%%%%%%%%%%%%%%%%%%%%%%%%%%%%%%%%%%%%%%%%%%%%%%%%%%%%%%%%%%%%%%%%%%%%%%%%%%%%%%%%%%%%%%%%

\subsection{Organizations of current paper}

In the next section, the goal is to study the local well-posedness of the \eqref{EECPK} model with initial data \eqref{IC-EECPK}. The a priori estimates in Lemma \ref{lemma 2.1} is first derived. Based on the a priori estimates, one proves the large local solution by continuity arguments. Then the positivity and conservation laws of the \eqref{EECPK} system is justified. The global existence near $(0,0,0)$ is finally constructed in the hypotheses \eqref{IC-K1}-\eqref{IC-K2}-\eqref{IC-Average} with $\rho_0 = 0$. In Section \ref{Sec:AD-EECP}, the well-posedness of the \eqref{AD-EECP} with initial data \eqref{IC-AD-EECP} is proved. There are two separated a priori estimates first derived in Lemma \ref{Lmm-AD}. The continuity arguments can conclude the local existence of large local solution. Then the positivity and conservation laws of \eqref{AD-EECP} system is verified. Finally, the global-in-time solution near $(\vr_a, h_a, 0)$ with $\vr_a \in [0, \Lambda_b)$ is constructed under the hypotheses \eqref{IC-A1}-\eqref{IC-A2}-\eqref{IC-Average-a} with $\vr_0 = \vr_a$. In Appendix \ref{Appendix}, one shows that near $(0, 0, n_0)$ with $n_0 > 0$, the linearized models of \eqref{EECPK} and \eqref{AD-EECP} systems are both unstable.

%%%%%%%%%%%%%%%%%%%%%%%%%%%%%%%%%%%%%%%%%%%%%%%%%%%%%%%%%%%%%%%%%%%%%%%%%%%%%%%%%%%%%%%%%%%%%%%%%%%%%%%%%%%%%%%%%%%%

\section{Well-posedness of \eqref{EECPK} model: Proof of Theorem \ref{Thm1}}

In this section, we devote to prove Theorem \ref{Thm1}. More precisely, we will employ the energy method to prove the local in time existence in Sobolev space with large initial data. Then the positivity of the solution will be checked. Under the positivity, the conservation law of mass \eqref{Consv-KEECP} will be proved provided that the initial condition \eqref{IC-Average} holds. Finally, based on the conservation law \eqref{Consv-KEECP} and under the further coefficients constraint \eqref{IC-K1}-\eqref{IC-K2}-\eqref{IC-Average} with $\rho_0 = 0$ before, the long time stability of the trivial steady state $(0,0,0)$ will be verified.

\subsection{A priori estimates for \eqref{EECPK} model}\label{A_Priori}
In this subsection, the a priori estimate of the system \eqref{EECPK} will be accurately derived from employing the energy method. We now introduce the following energy functional $E_L(t)$ and energy dissipative rate functional $D_L(t)$:
\begin{equation}\label{Loc-Energ}
\begin{aligned}
   E_{L} (t) & = \| \rho \|^2_{H^s_{x,z}} + \tfrac{1}{\eta+1} \| \vr \|^2_{H^s_{x}} + \| h \|^2_{H^s_{x}} + \| n \|_{H^{s}_{x}}^2 \,, \\
    D_{L} (t) & = \| \nabla_x \rho \|^2_{H^s_{x,z}(D)} + \tfrac{d}{2(\eta+1)} \| \nabla_x \vr \|^2_{H^s_{x}} +  D_h \| \nabla_x h \|^2_{H^s_{x}} + \tfrac{\beta}{2} \| h \|^2_{H^s_{x}} + \ D_n \| \nabla_x n \|_{H^{s}_{x}}^2 \,,
\end{aligned}
\end{equation}
where the constant $\eta \geq 0$ is given in \eqref{equation 1.7}. It is easy to see that
\begin{equation*}
	\begin{aligned}
		E_L (t) \thicksim \mathscr{E}_s (\rho, h, n) (t) \,, \ D_L (t) \thicksim \mathscr{D}_s (\rho, h, n) (t) \,,
	\end{aligned}
\end{equation*}
where the functional $\mathscr{E}_s (\rho, h, n)$ and $ \mathscr{D}_s (\rho, h, n)$ are defined in \eqref{Es} and \eqref{Ds}, respectively.

\begin{lemma}\label{lemma 2.1}
Let $s \geq 3$ be any fixed integer. Assume that $(\rho(t,x,z), h(t,x), n(t,x) )$ is a sufficiently smooth solution to system \eqref{EECPK} on the interval $[0,T]$. Then there is a positive constant $C=C(k_V, c_1, c_2, a, d, \alpha, \beta, \gamma, \xi, s)>0$, independent of $\eps$, such that
\begin{align*}
  \tfrac{\d}{\d t} E_{L} (t) + D_{L} (t) \leq C (1 + \tfrac{1}{\eps} ) ( 1 + E_L^{\frac{s}{2}} (t) ) E_{L} (t)
\end{align*}
 holds for all $t\in[0,T]$.
\end{lemma}

Before proving the conclusions of Lemma \ref{lemma 2.1}, we first introduce the following lemma, which will be frequently used latter.
\begin{lemma}\label{Lmm-g-norms}
	Assume $g (z,\omega)$ satisfies the hypothesis $(\mathrm{H}1)$. Let integer $s \geq 3$. Then there hold
	\begin{equation}\label{g-derivative}
		\p_z^l \p^k_x g (z, h(t,x)) =
		\left\{
		  \begin{aligned}
		  	0 \,, \qquad & \quad l \geq 2 \,, k \in \mathbb{N}^3 \textrm{ or } l = 1 \,, k \in \mathbb{N}^3 \setminus \{0\} \,, \\
		  	- k_V \,, \quad & \quad l = 1 \,, k = 0 \in \mathbb{N}^3 \,, \\
		  	k_V \partial^k_x L (h) \,, & \quad l = 0 \,, k \in \mathbb{N}^3 \,,
		  \end{aligned}
		\right.
	\end{equation}
	and moreover,
	\begin{equation}
	  \begin{aligned}
	    & \| \p^k_x g (z, h(t,x)) \|_{L^2_x L^\infty_z} \lesssim \| h \|_{H^s_x} + \| h \|^s_{H^s_x} \,, \\
	    & \| \p^k_x g (z, h(t,x)) \|_{L^4_x L^\infty_z} \lesssim \| h \|_{H^{s+1}_x} + \| h \|^s_{H^{s+1}_x} \,, \\
	    & \| \p^k_x g (z, h(t,x)) \|_{L^\infty_{x,z}} \lesssim \| h \|_{H^{s+2}_x} + \| h \|^s_{H^{s+2}_x}
	  \end{aligned}
	\end{equation}
	for $1 \leq |k| \leq s$.
\end{lemma}

Based on the definition of $g (z, \omega)$ in \eqref{g(z,w)}, this lemma can be proved directly by the Sobolev embedding theory. For simplicity of presentation, the details will be omitted here.

\begin{proof}[Proof of Lemma \ref{lemma 2.1}]
We first derive the $L^2$-estimate, which will contain the major structures of the energy functional. Then we estimate the higher order energy bound, which shall be consistent with the structures of $L^2$-estimate. We also emphasize that the properties \eqref{g-derivative} of $g(z,\omega)$ in Lemma \ref{Lmm-g-norms} will also be used frequently. For simplicity, the symbols
$$\mathscr{E}_s : = \mathscr{E}_s (\rho, h, n) \,, \quad \mathscr{D}_s : = \mathscr{D}_s (\rho, h, n) $$
will be employed in the later proof, where $\mathscr{E}_s (\rho, h, n)$ and $\mathscr{D}_s (\rho, h, n)$ are defined in \eqref{Es} and \eqref{Ds}, respectively.

\vspace*{2mm}

{\bf Step 1. $L^2$ estimates.}

\vspace*{2mm}

We take $L^2$-inner product with $\rho(t,x,z)$ on the first equation of system \eqref{EECPK}, and integrate by parts over $x\in \mathbb{T}^3$ and $z\in \T_w$. We thereby have
\begin{align}\label{L2-1}
  \tfrac{1}{2} \tfrac{\d}{\d t} \| \rho \|_{L^2_{x,z}}^2 + \| \nabla_x \rho \|_{L^2_{x,z}(D)}^2 = \underbrace{ - \tfrac{1}{\eps} \l \p_z ( g \rho ) , \rho \r_{L_{x,z}^2}}_{I_1} + \underbrace{ \l \gamma n \rho , \rho \r_{L_{x,z}^2}}_{I_2} \,.
\end{align}
Based on the integration by parts, the H\"older inequality and the property of $g(z,h(t,x))$ in \eqref{g-derivative}, we estimate $I_1$ and $I_2$ as follows:
\begin{align*}
  & I_1 = - \tfrac{1}{\eps} \l \p_z ( g \rho ) , \rho \r_{L_{x,z}^2} = \tfrac{1}{\eps} \l g \rho , \p_z \rho \r_{L_{x,z}^2} = - \tfrac{1}{\eps} \l \p_z g , \tfrac{1}{2} \rho^2 \r_{L_{x,z}^2} = \tfrac{k_V}{2\eps} \| \rho \|_{L^2_{x,z}}^2 \,, \\
  & I_2 = \l \gamma n \rho, \rho \r_{L_{x,z}^2} \leq \gamma \| n \|_{L^2_x} \| \rho \|_{L^4_xL^2_z}^2 \leq C \gamma \| n \|_{L^2_x} \| \rho \|_{H^1_xL^2_z}^2 \,.
\end{align*}
Therefore, we have
\begin{align}\label{equation 2.1}
  \tfrac{1}{2} \tfrac{\d}{\d t} \| \rho \|_{L^2_{x,z}}^2 + \| \nabla_x \rho \|_{L^2_{x,z}(D)}^2 \leq \tfrac{k_V}{2\eps} \| \rho \|_{L^2_{x,z}}^2 + C \gamma \| n \|_{L^2_x} \| \rho \|_{H^1_x L^2_z}^2 \,.
\end{align}
Observe that the mass $\vr (t,x) = \int_{\T_w} \rho (t,x,z) \d z$ obeys the evolution
\begin{equation}\label{Eq-mass}
  \begin{aligned}
    \p_t \vr = \Delta_x \int_{\T_w} D(z) \rho (t, x, z) \d z + \gamma n \vr \,.
  \end{aligned}
\end{equation}
From taking $L^2$-inner product with ${\vr}(t,x)$, integrating by parts over $x\in \mathbb{T}^3$ and splitting $D(z) = d + (D(z) - d)$ (here $d > 0$ is given in \eqref{equation 1.6}), we derive that
\begin{align}\label{L2-2}
  \no& \,\, \tfrac{1}{2} \tfrac{\d}{\d t} \| {\vr} \|_{L^2_{x}}^2 + d \| \nabla_x {\vr} \|_{L^2_{x}}^2 \\
  \no = & - \l \nabla_x \int_{\T_w} ( D(z) - d ) \rho(t,x,z) \d z , \nabla_x \int_{\T_w} \rho(t,x,z) \d z \r_{L_{x}^2} + \gamma \l n {\vr} , {\vr} \r_{L_x^2} \\
  \leq & \| \nabla_x {\vr} \|_{L^2_{x}} \underbrace{ \| \nabla_x \int_{\T_w} (D(z)-d) \rho(t,x,z) \d z \|_{L^2_{x}}}_{I_3} + \underbrace{ \gamma \l n {\vr} , {\vr} \r_{L_x^2}}_{I_4} \,.
\end{align}
From the hypotheses of $D(z)$ in \eqref{equation 1.6}, we easily know that
\begin{align*}
  I_3 & \leq \| \| D(z)-d \|_{L_z^2} \| \nabla_x \rho \|_{L_z^2} \|_{L_x^2} \leq \tfrac{1}{\sqrt{d}} \| D(z)-d \|_{L_z^2} \| \nabla_x \rho \|_{L_{x,z}^2(D)} \,, \\
  I_4 & = \gamma \l n {\vr} , {\vr} \r_{L_{x}^2} \leq \gamma \| n \|_{L^2_x} \| {\vr} \|_{L^4_x}^2 \leq C \gamma \| n \|_{L^2_x} \| {\vr} \|_{H^1_x}^2 \,.
\end{align*}
Moreover, the H\"older inequality implies
\begin{align*}
  \tfrac{1}{\sqrt{d}} \| D(z)-d \|_{L_z^2} \| \nabla_x \rho \|_{L_{x,z}^2(D)} \| \nabla_x {\vr} \|_{L_{x}^2} \leq \tfrac{d}{2} \| \nabla_x {\vr} \|_{L_{x}^2}^2 + \tfrac{\eta + 1}{2}\| \nabla_x \rho \|_{L_{x,z}^2(D)}^2 \,,
\end{align*}
where $\eta : = \tfrac{1}{d^2} \| D(z)-d \|_{L_z^2}^2 \geq 0 $ (see also \eqref{equation 1.7}).
Consequently, we have
\begin{align}\label{equation 2.2}
  \tfrac{1}{2 (\eta + 1) } \tfrac{\d}{\d t} \| {\vr} \|_{L^2_{x}}^2 + \tfrac{ d}{2 (\eta + 1) } \| \nabla_x {\vr} \|_{L^2_{x}}^2 \leq \tfrac{1}{2} \| \nabla_x \rho \|_{L_{x,z}^2(D)}^2 + \tfrac{C \gamma}{ (\eta + 1) } \| n \|_{L^2_x} \| {\vr} \|_{H^1_x}^2 \,.
\end{align}

We then take $L^2$-inner product with $h(t,x)$ on the second equation of system \eqref{EECPK}, and integrate by parts over $x\in \mathbb{T}^3$. We thereby have
\begin{align}\label{equation 2.3}
  \tfrac{1}{2} \tfrac{\d}{\d t} \| h \|_{L^2_{x}}^2 + D_h \| \nabla_x h \|_{L^2_{x}}^2 + \beta \| h \|_{L^2_{x}}^2 = \alpha \l {\vr} , h \r_{L_{x}^2} \leq \tfrac{\alpha^2}{2\beta} \| {\vr} \|_{L^2_{x}}^2 + \tfrac{\beta}{2} \| h \|_{L^2_{x}}^2 \,.
\end{align}
 We further take $L^2$-inner product with $n(t,x)$ on the third equation of system \eqref{EECPK}, and integrate by parts over $x\in \mathbb{T}^3$. We have
\begin{align}\label{equation 2.4}
  \tfrac{1}{2} \tfrac{\d}{\d t} \| n \|_{L^2_x}^2 + D_n \| \nabla_x n \|_{L^2_x}^2 = -\xi \l {\vr} n , n \r_{L_x^2} \leq \xi \| \vr \|_{L^2_x} \| n \|_{H^1_x}^2 \,.
\end{align}
Finally, together with the inequalities $\eqref{equation 2.1}$, $\eqref{equation 2.2}$, $\eqref{equation 2.3}$ and $\eqref{equation 2.4}$, we see
\begin{align}\label{equation 2.5}
  \no \tfrac{1}{2} & \tfrac{\d}{\d t} \Big( \| \rho \|_{L^2_{x,z}}^2 + \tfrac{1}{\eta + 1} \| \vr \|_{L^2_{x}}^2 + \| h \|_{L^2_{x}}^2 + \| n \|_{L^2_{x}}^2 \Big) \\
  + &  \tfrac{1}{2} \| \nabla_x \rho \|_{L^2_{x,z}(D)}^2 + \tfrac{ d}{2(\eta + 1)} \| \nabla_x \vr \|_{L^2_{x}}^2 + D_h \| \nabla_x h \|_{L^2_{x}}^2 + \tfrac{\beta}{2} \| h \|_{L^2_{x}}^2 + D_n \| \nabla_x n \|_{L^2_{x}}^2 \\
  \no \leq & \tfrac{k_V}{2\eps} \| \rho \|_{L^2_{x,z}}^2 +  \tfrac{\alpha^2}{2\beta} \| \vr \|_{L^2_{x}}^2 + C \gamma \| n \|_{L^2_x} \| \rho \|_{H^1_xL_z^2}^2 + \tfrac{C \gamma}{\eta + 1} \| n \|_{L^2_x} \| \vr \|_{H^1_x}^2 + C \xi \| \vr \|_{L^2_x} \| n \|_{H^1_x}^2 \\
  \no \leq & C ( \tfrac{1}{\eps} + \mathscr{E}_s^\frac{1}{2} ) \mathscr{E}_s \,.
\end{align}

\vspace*{2mm}

{\bf Step 2. Higher order derivatives estimates.}

\vspace*{2mm}

The estimates of the higher order derivatives will be divided into two parts: (1) the pure spatial derivatives estimates; (2) the mixed $(x,z)$-derivatives estimates.

\vspace*{2mm}

{\em (1) The pure spatial derivatives estimates.}

\vspace*{2mm}

First, we act $k$-order derivative operator $\p^k_x$ on \eqref{EECPK}${}_1$ for all $1 \leq |k|\leq s$, take $L^2$-inner product by dot with $\p_x^k \rho$ and integrate by parts over $x \in \mathbb{T}^3$ and $z\in \T_w$. We thereby have
\begin{align}\label{H-rho}
  \tfrac{1}{2} \tfrac{\d}{\d t} \| \p_x^k \rho \|_{L^2_{x,z}}^2 + \| \p_x^{k} \nabla_x \rho \|_{L_{x,z}^2(D)}^2 = \underbrace{ - \tfrac{1}{\eps} \l \p_x^k ( \p_z ( g \rho ) ), \p_x^k \rho \r_{L_{x,z}^2}}_{II_1} + \underbrace{ \gamma \l \p_x^k (n \rho), \p_x^k \rho \r_{L_{x,z}^2}}_{II_2} \,.
\end{align}
We split the term $II_1$ as follows:
\begin{align}\label{II1=II11+II12}
  \no II_1 = & \tfrac{1}{\eps} \l \p_x^k (g \rho), \p_z (\p_x^k \rho) \r_{L_{x,z}^2} = \tfrac{1}{\eps} \l g \p_x^k \rho, \p_z (\p_x^k \rho) \r_{L_{x,z}^2} + \tfrac{1}{\eps} \sum_{ \substack{ a+b=k \\ |a| \geq 1}} \l \p_x^a g \p_x^b \rho, \p_z (\p_x^k \rho) \r_{L_{x,z}^2} \\
  = & \underbrace{ \tfrac{1}{\eps} \l g \p_x^k \rho, \p_z (\p_x^k \rho) \r_{L_{x,z}^2}}_{II_{11}} \ \underbrace{ -\tfrac{1}{\eps} \sum_{ \substack{ a+b=k \\ |a| \geq 1 } } \l \p_x^a g \cdot \p_z \p_x^b \rho, \p_x^k \rho \r_{L_{x,z}^2} }_{II_{12}} \,.
\end{align}
Here the property of $g(z,h(t,z))$ in $\eqref{g-derivative}$ is used. Furthermore, the property of $g(z,h(t,z))$ in $\eqref{g-derivative}$ infers that
\begin{align}\label{II11}
  II_{11} = & - \tfrac{1}{2\eps} \l \p_z g, ( \p_x^k \rho )^2 \r_{L_{x,z}^2} = \tfrac{k_V}{2\eps} \| \p_x^k \rho \|_{L_{x,z}^2}^2 \lesssim \tfrac{1}{\eps} \mathscr{E}_s \,.
\end{align}
By Lemma \ref{Lmm-g-norms}, the H\"older inequality, the Sobolev embedding $H_x^1(\mathbb{T}^3)\hookrightarrow L_x^4(\mathbb{T}^3)$ and $H_x^2(\mathbb{T}^3)\hookrightarrow L_x^\infty(\mathbb{T}^3)$, the term $II_{12}$ can be bounded by
\begin{align}\label{II12}
   II_{12} \lesssim \tfrac{1}{\eps} ( \| h \|_{H_x^{s}} + \| h \|_{H_x^s}^s ) \| \rho \|_{H_{x,z}^s} \| \nabla_x \rho \|_{H_x^{s-1} L^2_z} \lesssim \tfrac{1}{\eps} (\mathscr{E}_s^\frac{1}{2} + \mathscr{E}_s^\frac{s}{2}) \mathscr{E}_s \,.
\end{align}
We emphasize that there is a mixed $(x,z)$-derivative term uncontrolled in the quantity $II_{12}$, which indicates us that the estimates of the mixed $(x,z)$-derivatives are required to close the energy bounds. Furthermore, combining with the Sobolev embedding $H_x^1(\mathbb{T}^3)\hookrightarrow L_x^4(\mathbb{T}^3)$ and $H_x^2(\mathbb{T}^3)\hookrightarrow L_x^\infty(\mathbb{T}^3)$, the term $II_2$ can be bounded by
\begin{align}\label{II2}
  \no II_2 = & \gamma \l \p_x^k n \cdot \rho, \p_x^k \rho \r_{L_{x,z}^2} + \gamma \l n \cdot \p_x^k \rho, \p_x^k \rho \r_{L_{x,z}^2} + \gamma \sum_{ \substack{ a+b=k \\ |a| \geq 1, |b| \geq 1 }} \l \p_x^a n \cdot \p_x^b \rho, \p_x^k \rho \r_{L_{x,z}^2} \\
  \lesssim & ( \| n \|_{H^s_{x}} + \| \rho \|_{H^s_x L^2_z} ) ( \| \nabla_x n \|_{H^{s-1}_{x}}^2 + \| \nabla_x \rho \|_{H^{s-1}_{x}L^2_z}^2 ) \lesssim (\mathscr{E}_s^\frac{1}{2} + \mathscr{E}_s^\frac{s}{2}) \mathscr{E}_s \,.
\end{align}
Consequently, collecting the above bounds, summing up for $1 \leq |k| \leq s$ and combining with the $L^2$ estimate \eqref{equation 2.1} reduce to
\begin{align}\label{equation 2.13}
  \tfrac{1}{2} \tfrac{\d}{\d t} \| \rho \|_{H^s_x L^2_z}^2 + \| \nabla_x \rho \|_{H^s_{x}L^2_z (D)}^2 \lesssim \tfrac{1}{\eps} ( 1 + \mathscr{E}_s^\frac{1}{2} + \mathscr{E}_s^\frac{s}{2} ) \mathscr{E}_s + \mathscr{E}_s \,.
\end{align}

We next act $k$-order derivative operator $\p^k_x$ on the equation \eqref{Eq-mass} for all $1 \leq |k| \leq s$, take $L^2$-inner product by dot with $\p_x^k{\vr}$ and integrate by parts over $x\in \mathbb{T}^3$ . We thereby have
\begin{align}\label{H-varrho}
  \no &\tfrac{1}{2} \tfrac{\d}{\d t} \| \p_x^k \vr \|_{L^2_{x}}^2 + d \| \p_x^k \nabla_x \vr \|_{L^2_{x}}^2 \\
  \no &= - \l \p_x^k \nabla_x \int_{\T_w} (D(z)-d) \rho(t,x,z) \d z, \p_x^k \nabla_x \vr \r_{L_{x}^2} + \gamma \l \p_x^k ( n \vr ), \p_x^k \vr \r_{L_{x}^2} \\
  & \leq \| \p_x^k \nabla_x \vr \|_{L^2_{x}} \underbrace{ \| \p_x^k \nabla_x \int_{\T_w} (D(z)-d) \rho(t,x,z) \d z \|_{L^2_{x}}}_{II_3} + \underbrace{ \gamma \l \p_x^k ( n \vr ), \p_x^k \vr \r_{L_{x}^2}}_{II_4} \,.
\end{align}
Here the splitting $D (z) = d + \big(D (z) - d\big)$ is also used. By the similar arguments in estimating the term $I_3$, the quantity $II_3$ can be bounded by
\begin{align}\label{II3-1}
  II_3 \leq \tfrac{\| D(z)-d\|_{L^2_{z}}}{\sqrt{d}} \| \p_x^k \nabla_x \rho \|_{L^2_{x,z}(D)} \,.
\end{align}
For the term $II_4$, combining with the Sobolev embedding $H_x^1(\mathbb{T}^3)\hookrightarrow L_x^4(\mathbb{T}^3)$ and $H_x^2(\mathbb{T}^3)\hookrightarrow L_x^\infty(\mathbb{T}^3)$, we have
\begin{align}\label{II4}
  \no II_4 = & \gamma \l \p_x^k n \cdot \vr, \p_x^k \vr \r_{L_x^2} + \gamma \l n \cdot \p_x^k \vr, \p_x^k \vr \r_{L_x^2} + \gamma \sum_{ \substack{ a+b=k \\ |a| \geq 1, |b| \geq 1 }} \l \p_x^a n \cdot \p_x^b \vr, \p_x^k \vr \r_{L_x^2} \\
  \lesssim & ( \| n \|_{H^s_{x}} + \| \vr \|_{H^{s}_x} ) \big( \| \nabla_x n \|^2_{H^{s-1}_x} + \| \nabla_x \vr \|^2_{H^{s-1}_x} \big) \lesssim \mathscr{E}_s^\frac{3}{2} \,.
\end{align}
Due to the Cauchy inequality, we have
\begin{align}\label{II3-2}
  \tfrac{\| D(z)-d\|_{L^2_{z}}}{\sqrt{d}} \| \p_x^k \nabla_x \rho \|_{L^2_{x,z}(D)} \| \p_x^k \nabla_x \vr \|_{L^2_{x}} \leq \tfrac{d}{2} \| \p_x^k \nabla_x \vr \|_{L^2_{x}}^2 + \tfrac{\eta + 1}{2} \| \p_x^k \nabla_x \rho \|_{L^2_{x,z}(D)}^2
\end{align}
for $\eta = \tfrac{\| D (z) - d \|^2_{L^2_z}}{d^2} \geq 0$. Collecting the above inequalities and combining the $L^2$ estimate \eqref{equation 2.2}, we therefore have
\begin{align}\label{equation 2.14}
  \tfrac{1}{2(\eta + 1)} \tfrac{\d}{\d t} \| \vr \|_{H^s_{x}}^2 + \tfrac{ d}{2(\eta + 1)} \| \nabla_x \vr \|_{H^s_{x}}^2 - \tfrac{1}{2} \| \nabla_x \rho \|_{H^s_{x} L^2_z(D)}^2 \lesssim \mathscr{E}_s^\frac{3}{2} \,.
\end{align}

We next act $k$-order derivative on the second equation of system \eqref{EECPK} for all $1 \leq |k|\leq s$, take $L^2$-inner product by dot with $\p_x^k h$ and integrate by parts over $x\in \mathbb{T}^3$. We thereby have,
\begin{align}\label{H-h}
  \tfrac{1}{2} \tfrac{\d}{\d t} \| \p_x^k h \|_{L^2_{x}}^2 + D_h \| \p_x^k \nabla_x h \|_{L_{x}^2}^2 + \beta \| \p_x^{k} h \|_{L_{x}^2}^2 = \alpha \l \p_x^k \vr, \p_x^k h \r_{L_{x}^2} \leq \tfrac{\beta}{2} \| \p^k_x h \|^2_{L^2_x} + \tfrac{\alpha^2}{2 \beta} \| \p^k_x \vr \|^2_{L^2_x} \,,
\end{align}
which, together with the $L^2$ estimate \eqref{equation 2.3}, implies that
\begin{align}\label{equation 2.15}
  \tfrac{1}{2} \tfrac{\d}{\d t} \| h \|_{H^s_{x}}^2 + D_h \| \nabla_x h \|_{H_{x}^s}^2 + \tfrac{\beta}{2} \| h \|_{H_{x}^s}^2 \leq \tfrac{\alpha^2}{2\beta} \| \nabla_x \vr \|_{H^{s-1}_{x}}^2 \lesssim \mathscr{E}_s \,.
\end{align}

Next, from acting $\p_x^k$ on the third equation of system \eqref{EECPK} for all $1 \leq |k|\leq s $, taking $L^2$-inner product by dot with $\p_x^k n$ and integrating by parts over $x\in \mathbb{T}^3$, we deduce that
\begin{align}\label{H-n}
  \no \tfrac{1}{2} & \tfrac{\d}{\d t} \| \p_x^k n \|_{L^2_{x}}^2 + D_n \| \p_x^k \nabla_x n \|_{L^2_{x}}^2 = -\xi \l \p_x^k ( \vr n), \p_x^k n \r_{L_x^2} \\
  \no = & \xi \l \p_x^k n \cdot \vr, \p_x^k n \r_{L_x^2} + \xi \l n \cdot \p_x^k \vr, \p_x^k n \r_{L_x^2} + \xi \sum_{ \substack{ a+b=k \\ |a| \geq 1, |b| \geq 1 }} \l \p_x^a n \cdot \p_x^b \vr, \p_x^k n \r_{L_x^2} \\
  \lesssim & ( \| n \|_{H^s_{x}} + \| \vr \|_{H^{s}_x} ) \big( \| \nabla_x n \|^2_{H^{s-1}_x} + \| \nabla_x \vr \|^2_{H^{s-1}_x} \big) \lesssim \mathscr{E}_s^\frac{3}{2} \,.
\end{align}
Here the Sobolev embedding $H_x^1(\mathbb{T}^3)\hookrightarrow L_x^4(\mathbb{T}^3)$ and $H_x^2(\mathbb{T}^3)\hookrightarrow L_x^\infty(\mathbb{T}^3)$ are also utilized. Together with the $L^2$ estimate \eqref{equation 2.4}, we infer that
\begin{align}\label{equation 2.16}
  \tfrac{1}{2} \tfrac{\d}{\d t} \| n \|_{H^s_{x}}^2 + D_n \| \nabla_x n \|_{H^s_{x}}^2 \lesssim \mathscr{E}_s^\frac{3}{2}  \,.
\end{align}
It is therefore derived from summing the inequalities \eqref{equation 2.13}, \eqref{equation 2.14}, \eqref{equation 2.15} and \eqref{equation 2.16} that
\begin{equation}\label{Spatial-Est}
  \begin{aligned}
    & \tfrac{1}{2} \tfrac{\d}{\d t} \big( \| \rho \|^2_{H^s_x L^2_z} + \tfrac{1}{\eta + 1} \| \vr \|^2_{H^s_x} + \| h \|^2_{H^s_x} + \| n \|^2_{H^s_x} \big) \\
    & + \tfrac{1}{2} \| \nabla_x \rho \|^2_{H^s_x L^2_z (D)} + \tfrac{d}{2 (\eta + 1)} \| \nabla_x \vr \|^2_{H^s_x} + D_h \| \nabla_x h \|^2_{H^s_x} + \tfrac{\beta}{2} \| h \|^2_{H^s_x} + D_n \| \nabla_x n \|^2_{H^s_x} \\
    & \lesssim \tfrac{1}{\eps} ( 1 + \mathscr{E}_s^\frac{1}{2} + \mathscr{E}_s^\frac{s}{2} ) \mathscr{E}_s + \mathscr{E}_s + \mathscr{E}_s^\frac{3}{2} \,.
  \end{aligned}
\end{equation}

\vspace*{2mm}

{\em (2) The mixed $(x,z)$-derivative estimates.}

\vspace*{2mm}

One notices that there is a term $\tfrac{1}{\eps} ( \| h \|_{H_x^s} + \| h \|_{H_x^s}^s ) \| \rho \|_{H_{x,z}^s}^2$ including the mixed $(x,z)$-derivatives of $\rho (t,x,z)$ in \eqref{II12}, which cannot be controlled by the energy and dissipative rate in the left hand side of \eqref{Spatial-Est}. In other words, the energy inequality \eqref{Spatial-Est} is not closed. The mixed $(x,z)$-derivative estimates are therefore required. We act the mixed derivative operator $\p_z^l \p_x^m$ on the first equation of system \eqref{EECPK} for all $(m, l) \in \mathbb{N}^3 \times \mathbb{N}$ with $l+|m|\leq s $ and $l \neq 0$, take $L^2$-inner product by dot with $\p_z^l\p_x^m{\rho}$ and integrate by parts over $x\in \mathbb{T}^3$ and $z \in \T_w$. We thereby have
\begin{align}\label{M-rho}
  \tfrac{1}{2} \tfrac{\d}{\d t} \| \p_z^l \p_x^m \rho \|_{L^2_{x,z}}^2 = & \underbrace{ \l \p_z^l \p_x^m \Delta_x (D(z)\rho), \p_z^l \p_x^m \rho \r_{L_{x,z}^2}}_{III_1} \underbrace{ - \tfrac{1}{\eps} \l \p_z^{l+1} \p_x^m (g \rho), \p_z^l \p_x^m \rho \r_{L_{x,z}^2}}_{III_2} \no\\
  & + \underbrace{ \gamma \l \p_z^l \p_x^m (n \rho), \p_z^l \p_x^m \rho \r_{L_{x,z}^2}}_{III_3}  \,.
\end{align}

Considering the hypotheses of $D(z)$ in \eqref{equation 1.5} and \eqref{equation 1.6}, the term $III_1$ can be estimated as
\begin{equation}\label{III1}
\begin{aligned}
  III_1 & = - \l D(z) \p_z^l \p_x^m \nabla_x \rho, \p_z^l \p_x^m \nabla_x \rho \r_{L_{x,z}^2} - \sum_{ \substack{ a+b=l \\ |a| \geq 1 }} \l \p_z^a ( D(z) ) \p_z^b ( \p_x^m \nabla_x \rho ), \p_z^l \p_x^m \nabla_x \rho \r_{L_{x,z}^2} \\
  & \leq - \| \nabla_x \p_z^l \p_x^m \rho \|_{L_{x,z}^2(D)}^2 + \tfrac{1}{\sqrt{d}} \sum_{ \substack{ a+b=l \\ |a| \geq 1 }} \| \p_z^a (D(z)) \|_{L^{\infty}_z} \| \p_z^b ( \p_x^m \nabla_x \rho ) \|_{L^{2}_{x,z}} \| \p_z^l \p_x^m \nabla_x \rho \|_{L^{2}_{x,z}(D)} \\
  & \leq - \tfrac{1}{2} \| \nabla_x \p_z^l \p_x^m \rho \|_{L_{x,z}^2(D)}^2 + \tfrac{b^2}{2d} \| \nabla_x \rho \|_{H_{x,z}^{s-1}}^2 \leq - \tfrac{1}{2} \| \nabla_x \p_z^l \p_x^m \rho \|_{L_{x,z}^2(D)}^2 + C \mathscr{E}_s \,.
\end{aligned}
\end{equation}

In order to dominate the term $III_2$, we split it as three parts:
\begin{align*}
  III_2 & = \underbrace{ \tfrac{1}{\eps} \l \p_z^{l} \p_x^m \rho \cdot g, \p_z ( \p_z^l \p_x^m \rho ) \r_{L_{x,z}^2}}_{III_{21}} + \underbrace{ \tfrac{1}{\eps} \l \p_z^{l-1} \p_x^m \rho \cdot \p_z g, \p_z ( \p_z^l \p_x^{m} \rho ) \r_{L_{x,z}^2}}_{III_{22}} \\
  & \quad + \underbrace{ \tfrac{1}{\eps} \sum_{ \substack{ a+b=m \\ |b| \geq 1 }} \l \p_z^l ( \p_x^a \rho \cdot \p_x^b g ), \p_z ( \p_z^l \p_x^{m} \rho ) \r_{L_{x,z}^2}}_{III_{23}} \,.
\end{align*}
Considering the hypothesis of $g(z,h(t,x))$ in \eqref{g-derivative}, the term $III_{21}$ can be bounded by
\begin{align*}
   III_{21} & = \tfrac{1}{2\eps} \l g, \p_z ( \p_z^l \p_x^m \rho )^2 \r_{L_{x,z}^2} = - \tfrac{1}{2\eps} \l \p_z g, ( \p_z^l \p_x^m \rho)^2 \r_{L_{x,z}^2} = \tfrac{k_V}{2\eps} \| \p_z^l \p_x^m \rho \|_{L^2_{x,z}}^2 \lesssim \tfrac{1}{\eps} \mathscr{E}_s \,.
\end{align*}
Similarly, the term $III_{22}$ can be bounded by
\begin{align*}
  III_{22} & = - \tfrac{1}{\eps} \l \p_z^l \p_x^m \rho \cdot \p_z g, \p_z^l \p_x^{m} \rho \r_{L_{x,z}^2} = \tfrac{k_V}{\eps} \| \p_z^l \p_x^m \rho \|_{L^2_{x,z}}^2 \lesssim \tfrac{1}{\eps} \mathscr{E}_s \,.
\end{align*}
Now we estimate the term $III_{23}$. From the property of $g(z, \omega)$ in \eqref{g-derivative} in Lemma \ref{Lmm-g-norms}, we see that
\begin{align*}
  III_{23} & = \underbrace{ - \tfrac{1}{\eps} \sum_{ \substack{ a+b=m \\ |b| = 1 }} \l \p_z^{l+1} \p_x^a \rho \cdot \p_x^b g, \p_z^l \p_x^{m} \rho \r_{L_{x,z}^2}}_{III_{231}} \ \underbrace{ - \tfrac{1}{\eps} \sum_{ \substack{ a+b=m \\ |b| \geq 2 }} \l \p_z^{l+1} \p_x^a \rho \cdot \p_x^b g, \p_z^l \p_x^{m} \rho \r_{L_{x,z}^2}}_{III_{232}} \,.
\end{align*}
Lemma \ref{Lmm-g-norms} and the H\"older inequality tell us that
\begin{equation*}
  \begin{aligned}
    III_{231} \lesssim \tfrac{1}{\eps} \big( \| h \|_{H^s_x} + \| h \|^s_{H^s_x} \big) \| \rho \|_{H^s_{x,z}} \| \nabla_x \rho \|_{H^{s-1}_{x,z}} \lesssim \tfrac{1}{\eps} (\mathscr{E}_s^\frac{1}{2} + \mathscr{E}_s^\frac{s}{2}) \mathscr{E}_s \,,
  \end{aligned}
\end{equation*}
where $s \geq 3$ is required. For the term $III_{232}$, one easily derives from $l + |m| \leq s$, $l \geq 1$, $a+b=m$ and $|b| \geq 2$ that $2 \leq |b| \leq s - 2$, $l + 1 + |a| \leq s - 1$. It is therefore derived from Lemma \ref{Lmm-g-norms}, the H\"older inequality and the Sobolev embedding  $H_x^1(\mathbb{T}^3)\hookrightarrow L_x^4(\mathbb{T}^3)$ that
\begin{equation*}
  \begin{aligned}
    III_{232} \lesssim \tfrac{1}{\eps} \big( \| h \|_{H^s_x} + \| h \|^s_{H^s_x} \big) \| \rho \|_{H^s_{x,z}} \| \nabla_x \rho \|_{H^{s-1}_{x,z}} \lesssim \tfrac{1}{\eps} (\mathscr{E}_s^\frac{1}{2} + \mathscr{E}_s^\frac{s}{2}) \mathscr{E}_s \,.
  \end{aligned}
\end{equation*}
We thereby have
\begin{equation*}
  \begin{aligned}
    III_{23} = III_{231} + III_{232} \lesssim \tfrac{1}{\eps} (\mathscr{E}_s^\frac{1}{2} + \mathscr{E}_s^\frac{s}{2}) \mathscr{E}_s \,.
  \end{aligned}
\end{equation*}
Consequently, one obtains
\begin{equation}\label{III2}
  \begin{aligned}
    III_2 \lesssim \tfrac{1}{\eps} (\mathscr{E}_s^\frac{1}{2} + \mathscr{E}_s^\frac{s}{2}) \mathscr{E}_s \,.
  \end{aligned}
\end{equation}

It remains to control the term $III_3$. The Sobolev embedding $H_x^1(\mathbb{T}^3)\hookrightarrow L_x^4(\mathbb{T}^3)$ and $H_x^2(\mathbb{T}^3)\hookrightarrow L_x^\infty(\mathbb{T}^3)$ reduce to
\begin{align}\label{III3}
   III_{3} \lesssim \| n \|_{H^s_{x}} \| \rho \|_{H^{s}_{x,z}}^2 \lesssim \mathscr{E}_s^\frac{3}{2} \,.
\end{align}

We therefore infer from the bounds of $III_1$, $III_2$ and $III_3$ that
\begin{align}\label{Mixed-Est}
  \tfrac{1}{2} \tfrac{\d}{\d t} \| \p_z^l \p_x^{m} \rho \|_{L^{2}_{x,z}}^2 + \tfrac{1}{2} \| \nabla_x \p^l_z \p^m_x \rho \|^2_{L^2_{x,z} (D)} \lesssim \tfrac{1}{\eps} (1 + \mathscr{E}_s^\frac{1}{2} + \mathscr{E}_s^\frac{s}{2}) \mathscr{E}_s + \mathscr{E}_s + \mathscr{E}_s^\frac{3}{2}
\end{align}
for all $ l + |m| \leq s $ with $l \neq 0$. Finally, together with the inequalities $\eqref{Spatial-Est}$ and $\eqref{Mixed-Est}$, we have
\begin{equation}\label{Closed-Est}
\begin{aligned}
  & \tfrac{1}{2} \tfrac{\d}{\d t} \big( \| \rho \|^2_{H^s_{x,z}} + \tfrac{1}{\eta + 1} \| \vr \|^2_{H^s_x} + \| h \|^2_{H^s_x} + \| n \|^2_{H^s_x} \big) \\
  & +  \| \nabla_x \rho \|^2_{H^s_{x,z} (D)} + \tfrac{d}{2 (\eta + 1)} \| \nabla_x \vr \|^2_{H^s_x} + D_h \| \nabla_x h \|^2_{H^s_x} + \tfrac{\beta}{2} \| h \|^2_{H^s_x} + D_n \| \nabla_x n \|^2_{H^s_x} \\
  & \lesssim \tfrac{1}{\eps} (1 + \mathscr{E}_s^\frac{1}{2} + \mathscr{E}_s^\frac{s}{2}) \mathscr{E}_s + \mathscr{E}_s + \mathscr{E}_s^\frac{3}{2}  \,.
\end{aligned}
\end{equation}
Recalling the definitions of the energy functionals $E_L (t)$ and $D_L (t)$ in \eqref{Loc-Energ}, we finish the proof of Lemma \ref{lemma 2.1} from the inequality \eqref{Closed-Est}.
\end{proof}

\subsection{Local well-posedness to the \eqref{EECPK} system with large initial data}\label{Sec: Local-Result to the EECP model}

In this subsection, we will prove the local well-posedness of the system \eqref{EECPK} with large initial data \eqref{IC-EECPK}, namely, prove the first part of Theorem \ref{Thm1}. We first construct the approximate system by iterative scheme as follows: for all integer $l \geq 0$,
\begin{align}\label{equation 3.1}
  \left\{
    \begin{array}{c}
      \p _t \rho^{l+1} = D(z) \Delta_x \rho^{l+1} - \tfrac{1}{\eps} \p_z ( g(z, h^l) \rho^{l+1} ) + \gamma n^l \rho^{l+1}   \,, \\[2mm]
      \p _t h^{l+1} = D_h \Delta_x h^{l+1} + \alpha {\vr}^{l+1} - \beta h^{l+1} \,, \\[2mm]
      \p _t n^{l+1} = D_n \Delta_x n^{l+1} - \xi {\vr}^l n^{l+1}  \,, \\[2mm]
      (\rho^{l+1}, h^{l+1}, n^{l+1})|_{t=0}=(\rho^{\eps, in}, h^{\eps, in},  n^{\eps, in}) \,.
    \end{array}
  \right.
\end{align}
The iteration starts from
\begin{equation}\label{equation 3.2}
  ( \rho^{0} (t,x,z), h^{0} (t,x), n^{0} (t,x) ) = ( \rho^{\eps, in} (x,z), h^{\eps, in} (x), n^{\eps, in} (x) ) \,.
\end{equation}

In the arguments proving the convergence ($l \rightarrow \infty$) of the approximate solutions \eqref{equation 3.1}, it is essential to obtain uniform (in $l \geq 0$) energy estimates of \eqref{equation 3.1} in a uniform lower bound lifespan time, whose derivations are the almost same as the derivations of the a priori estimates for the system \eqref{EECPK} with initial data \eqref{IC-EECPK}. The arguments of the uniform lower bound lifespan time can be referred to \cite{JL-SIAM-2019}, for instance. The convergence arguments are a standard process. For simplicity, we will only consider the a priori estimates in Lemma \ref{lemma 2.1} for the smooth solutions of \eqref{EECPK}-\eqref{IC-EECPK} on some time interval.

From Lemma \ref{lemma 2.1}, we see that
\begin{equation}\label{Loc-1}
  \begin{aligned}
    \tfrac{\d}{\d t} E_{L} (t) + D_{L} (t) \leq C (1 + \tfrac{1}{\eps} ) ( 1 + E_L^{\frac{s}{2}} (t) ) E_{L} (t) \,,
  \end{aligned}
\end{equation}
where the energy functional $E_L (t)$ and $D_L (t)$ are defined in \eqref{Loc-Energ}. Then \eqref{Loc-1} implies
\begin{equation*}
  \begin{aligned}
    \frac{\tfrac{\d}{\d t} E_L (t) }{( 1 + E_L^\frac{s}{2} (t) ) E_L (t)} \leq C (1 + \tfrac{1}{\eps}) \,,
  \end{aligned}
\end{equation*}
which means that
\begin{equation}\label{Loc-2}
  \begin{aligned}
    \frac{\d}{\d t} \Big[ \ln \frac{E_L (t)}{(1 + E_L^\frac{s}{2} (t) )^\frac{2}{s}} \Big] \leq C (1 + \tfrac{1}{\eps}) \,.
  \end{aligned}
\end{equation}
Noticing that
\begin{equation*}
  \begin{aligned}
    E_L (0) = E_L^{\eps, in} : = \| \rho^{\eps, in} \|^2_{H^s_{x,z}} + \tfrac{1}{2 (\eta + 1)} \| \int_{ \T_w } \rho^{\eps, in} (\cdot, z) \d z \|^2_{H^s_x} + \| h^{\eps, in} \|^2_{H^s_x} + \| n^{\eps, in} \|^2_{H^s_x} < \infty \,,
  \end{aligned}
\end{equation*}
we derive from integrating the inequality \eqref{Loc-2} over $[0,t]$ that
\begin{equation}\label{Loc-3}
  \begin{aligned}
    \frac{E_L (t)}{ \big[ 1 + E_L^\frac{s}{2} (t) \big]^\frac{2}{s} } \leq \frac{E_L^{\eps, in} }{ \big[ 1 + (E_L^{\eps, in})^\frac{s}{2} \big]^\frac{2}{s} } e^{C (1 + \tfrac{1}{\eps}) t} : = A(t) \,.
  \end{aligned}
\end{equation}
Consider the function
$$H (y) = \frac{y}{(1 + y^\frac{s}{2})^\frac{2}{s}} $$
for $y \geq 0$. It is easy to see that $H (0) = 0$, $\lim\limits_{y \rightarrow \infty} H (y) = 1$, and
\begin{equation*}
  \begin{aligned}
    H' (y) = \frac{1}{(1 + y^\frac{s}{2})^{\frac{2}{s}+1}} > 0 \,, \ y \geq 0
  \end{aligned}
\end{equation*}
is strictly decreasing in $[0,\infty)$. Thus the behavior of $H(y)$ reads as follows (Figure \ref{Fig-H}).
\begin{figure}[h]
	\centering
	\begin{tikzpicture}[node distance=2cm]
	\draw[->,thick](-1,0)--(4,0) node[right,scale=1]{$y$};
	\draw[->,thick](0,-1)--(0,4) ;
	\draw[-] node[below left]{0} (2.3,0)--(4,0) ;
	\draw[scale=1, very thick, domain=0:3.9, smooth,variable=\t]
	plot (\t,{-3/(( 3 * \t  + 1 ) ) + 3 }) node[right] {$H(y)$};
	\draw[-,thick,dashed] (4,3)--(-1,3) (0,3) node[above left, scale=1]{1};
	\draw[-,thick] (4,2)--(-1,2) (0,2) node[above left, scale=1]{$A(t)$};
	\draw[-,thick,dashed] (2/3,2)--(2/3,0) (0.2,0) node[below right, scale=1]{$H^{-1} (A(t))$};
	\end{tikzpicture}
	\caption{Behavior of $H (y)$.}\label{Fig-H}
\end{figure}
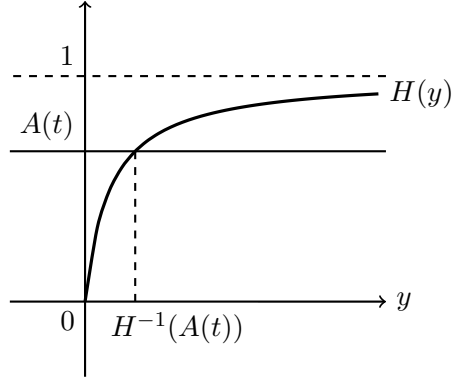
We therefore see that if $A(t) < 1$, the nonlinear inequality \eqref{Loc-3} on the functional $E_L (t)$ can be solved as
\begin{equation}\label{Loc-4}
  \begin{aligned}
    E_L (t) \leq H^{-1} (A(t)) = H^{-1} \left( \tfrac{E_L^{\eps, in} }{ \big[ 1 + (E_L^{\eps, in})^\frac{s}{2} \big]^\frac{2}{s} } e^{C (1 + \tfrac{1}{\eps}) t} \right) : = B_\eps (t) \,.
  \end{aligned}
\end{equation}
Notice that $B_\eps (t)$ is strictly increasing on $t \geq 0$ and $A(t) < 1$ implies that
$$t < T_0 = \tfrac{1}{C(1 +\eps^{-1})} \ln \tfrac{[1 + (E_L^{\eps, in})^\frac{s}{2}]^\frac{2}{s}}{E_L^{\eps, in}} \,. $$
Consequently, from \eqref{Loc-1} and \eqref{Loc-4}, we derive that for any $0 < T < T_0 $,
\begin{equation*}
  \begin{aligned}
    E_L (t) + \int_0^t D_L (\tau) \d \tau \leq E_L^{\eps, in} + CT (1 + \eps^{-1}) B_\eps (T) \big[ 1 + B_\eps^\frac{s}{2} (T) \big] : = \widetilde{B}_\eps (T, E_L^{\eps, in}, C)
  \end{aligned}
\end{equation*}
holds for all $t \in [0,T]$. Consequently, we conclude the local existence of the system \eqref{EECPK} with large initial data \eqref{IC-EECPK}, namely, the first part of Theorem \ref{Thm1}.

\subsection{\bf Positivity and conservation laws}
\label{Subsec-2-3}

In this subsection, the first goal is to prove the positivity of the solution $(\rho, h, n)$ to \eqref{EECPK} system constructed in the first part of Theorem \ref{Thm1}. The approach is inspired by Perthame's work \cite{Perthame-2015-BOOK}. Based on the positivity of $\rho (t,x,z)$ and $n (t,x)$, the second goal is to derive their conservation laws  under the initial assumption \eqref{IC-Average}. Remark that the \eqref{EECPK} system admits a unconditionally conservation law
\begin{equation*}
	\begin{aligned}
		\frac{\d}{\d t} \left( \iint_{ \T^3 \times \T_w } \rho (t,x,z) \d z \d x + \frac{\gamma}{\xi} \int_{ \T^3 } n (t,x) \d x \right) = 0 \,.
	\end{aligned}
\end{equation*}
However, the previous conservation law can not be separated without any additional conditions. Thus the initial assumption \eqref{IC-Average} is required to separately obtain the conservation laws of $\rho (t,x,z)$ and $n (t,x)$.

\begin{proof}[\bf Proof of part (2) in Theoerm \ref{Thm1}] We split it into two steps.

	\vspace*{3mm}
	{\em Step 1. Positivity.}
	\vspace*{3mm}
	
	We first verify the positivity of the \eqref{EECPK} system. More precisely, for the solution $(\rho, h, n)$ constructed in part (1) of Theorem \ref{Thm1}, if the initial data are further assumed $\rho^{\eps, in}, h^{\eps, in}, n^{\eps, in} \geq 0$, then there hold $\rho, h, n \geq 0$. For $h (t,x)$ and $n (t,x)$, the following lemma constructed by Perthame's book \cite{Perthame-2015-BOOK} can be directly employed:
	
	\begin{lemma}[Lemma 3.10 of \cite{Perthame-2015-BOOK}]\label{Lmm-Perthame}
		Let $u \in C(\R^+; L^2 (\R^d))$ be a weak solution to the parabolic equation
		\begin{equation*}
			\left\{
			  \begin{aligned}
			  	& \partial_t u - \Delta_x u + R(t,x) u = Q(t,x) \quad \textrm{ in } \R^d \,, \\
			  	& u (0,x) = u^{in} (x) \,.
			  \end{aligned}
			\right.
		\end{equation*}
	Assume $|R(t,x)| \leq \Gamma (t)$ with $\Gamma \in L^\infty_{\mathrm{loc}} (\R^+)$ and $Q \in C(\R^+; L^2 (\R^d))$. If $Q \geq 0$ and $u^{in} \geq 0$, then $u \geq 0$.
	\end{lemma}
	Observe that $n (t,x)$ subjects to $\partial_t n - D_n n + \xi \vr n = 0$ with $n (0,x) = n^{\eps, in} (x) \geq 0$. Here $Q = 0$ and by Sobolev embedding theory
	\begin{equation*}
		\begin{aligned}
			|R(t,x)| = \xi |\vr (t,x)| \leq C \| \vr (t) \|_{H^s_x} : = \Gamma (t) \in L^\infty (0,T) \,.
		\end{aligned}
	\end{equation*}
	Lemma \ref{Lmm-Perthame} thereby implies $n (t,x) \geq 0$.
	
	Recall that $h(t,x)$ obeys $\partial_t h - D_h \Delta_x h + \beta h = \alpha \vr$ with $h (0,x) = h^{\eps, in} (x) \geq 0$. Here $R(t,x) \equiv \beta \in L^\infty (0,T)$ and $Q(t,x) = \alpha \vr (t,x) \in C(0,T; H^{s'}_x)$ for any $s' < s$. Once $\vr (t,x) \geq $ holds, Lemma \ref{Lmm-Perthame} therefore yields $h (t,x) \geq 0$.
	
	It remains to prove $\rho (t,x,z) \geq 0$, which ensures that $\vr (t,x) = \int_{ \T_w } \rho (t,x,z) \d z \geq 0$. However, one fails to directly apply Lemma \ref{Lmm-Perthame} to verify the positivity of $\rho (t,x,z)$, since it obeys
	\begin{equation}\label{rho-Equ}
		\left\{
		\begin{aligned}
			& \partial_t \rho = D (z) \Delta_x \rho - \tfrac{1}{\eps} (g (z,h) \rho) + \gamma n \rho \,, \\
			& \rho (0,x,z) = \rho^{\eps,in} (x,z) \geq 0 \,,
		\end{aligned}
	    \right.
	\end{equation}
    which is not a parabolic equation. However, one can employ the similar arguments of proving Lemma \ref{Lmm-Perthame} in \cite{Perthame-2015-BOOK}. As shown in part (1) of Theorem \ref{Thm1}, there hold
    \begin{equation*}
    	\begin{aligned}
    		\rho \in C ([0,T) \times \T^3 \times \T_w) \,, \ h, n \in C ([0,T) \times \T^3) \,.
    	\end{aligned}
    \end{equation*}
    Define a set
    \begin{equation}\label{Omega-}
    	\begin{aligned}
    		\Omega_- : = \big\{(t,x,z) \in [0,T) \times \T^3 \times \T_w ; \rho (t,x,z) < 0 \} \,.
    	\end{aligned}
    \end{equation}
    Our goal is to prove $\Omega_- = \emptyset$. Since $\rho (0,x,z) = \rho^{\eps,in} (x,z) \geq 0$, one has $\Omega_-^0 : = \big\{(t,x,z) \in \{0\} \times \T^3 \times \T_w ; \rho (0,x,z) < 0 \} = \emptyset$, which is consistent with our goal.

    Assume that $\Omega_- \neq \emptyset$, the continuity of $\rho$ implies that $\Omega_-$ is relatively open in $[0,T) \times \T^3 \times \T_w$. Then, for any $R \gg 1$ there is a compact set $K_R \subseteq \Omega_-$ such that $\mathbf{dist} (K_R, \partial \Omega_-) = \frac{1}{R}$. Take a cutoff function $\chi_R (t,x,z) \in C^\infty_0 (\Omega_-)$ satisfying
    \begin{equation}\label{Cutoff}
    	\begin{aligned}
    		& 0 \leq \chi_R (t,x,z) \leq 1 \,, \ \chi_R (t,x,z) \equiv 1 \ \textrm{ on } K_R \,, \\
    		& |\partial_t \chi_R| + |\nabla_x \chi_R| + |\partial_z \chi_R| \leq \frac{C}{R} \,, \ |\Delta_x \chi_R| \leq \frac{C}{R^2}
    	\end{aligned}
    \end{equation}
    for some constant $C > 0$. Multiplying \eqref{rho-Equ} by $\rho \chi_R$ and integrating by parts over $(t,x,z) \in [0,t] \times \T^3 \times \T_w$ infer that
    \begin{equation}\label{Pos-1}
    	\begin{aligned}
    		& \tfrac{1}{2} \iint_{\T^3 \times \T_w} \rho^2 \chi_R (t) \d z \d x + \int_0^t \iint_{\T^3 \times \T_w} D(z) |\nabla_x \rho|^2 \chi_R \d x \d z \d t' \\
    		= & \tfrac{1}{2} \int_0^t \iint_{\T^3 \times \T_w} \rho^2 \partial_t \chi_R \d z \d x \d t' + \tfrac{1}{2} \int_0^t \iint_{\T^3 \times \T_w} D(z) \rho^2 \Delta_x \chi_R \d z \d x \d t' \\
    		& + \gamma \int_0^t \iint_{\T^3 \times \T_w} n \rho^2 \chi_R \d z \d x \d t' - \tfrac{k_V}{2 \eps} \int_0^t \iint_{\T^3 \times \T_w} \rho^2 \chi_R \d z \d x \d t' \\
    		& - \tfrac{1}{2 \eps} \int_0^t \iint_{\T^3 \times \T_w} g(z,h) \rho^2 \partial_z \chi_R \d z \d x \d t'
    	\end{aligned}
    \end{equation}
    for all $t \in [0, T)$. Together with \eqref{Cutoff} and $\rho \,, g (z,h) \in L^\infty (0,T; H^s_{x,z}) \,, n \in L^\infty (0,T; H^s_x)$ derived from the part (1) of Theorem \ref{Thm1}, one infers that the right-hand side of the previous equality \eqref{Pos-1} can be bounded by $C (1 + \frac{1}{\eps}) \frac{1}{R}$. Therefore, using the Lebesgue Dominated Convergence theorem, we can pass to the limit at $R \to \infty$ and obtain
    \begin{equation}
    	\begin{aligned}
    		\tfrac{1}{2} & \iint_{\T^3 \times \T_w} \rho^2 \mathbbm{1}_{\Omega_-} (t) \d z \d x \\
    		\leq & \gamma \int_0^t \iint_{\T^3 \times \T_w} n \rho^2 \mathbbm{1}_{\Omega_-} (t') \d z \d x \d t' - \tfrac{k_V}{2 \eps} \int_0^t \iint_{\T^3 \times \T_w} \rho^2 \mathbbm{1}_{\Omega_-} (t') \d z \d x \d t' \\
    		\leq & C (1 + \tfrac{1}{\eps}) ( 1 + \| n \|_{L^\infty(0,T; H^s_x)} ) \int_0^t \iint_{\T^3 \times \T_w} \rho^2 \mathbbm{1}_{\Omega_-} (t') \d z \d x \d t' \,.
    	\end{aligned}
    \end{equation}
    The Gr\"onwall inequality implies that $\iint_{\T^3 \times \T_w} \rho^2 \mathbbm{1}_{\Omega_-} (t) \d z \d x \leq 0$ for all $t \in [0,T)$, which means that
    \begin{equation*}
    	\begin{aligned}
    		\iiint_{\Omega_-} \rho^2 \d z \d x \d t = 0 \,.
    	\end{aligned}
    \end{equation*}
	Since $\Omega_-$ is a relatively open set in $[0,T) \times \T^3 \times \T_w$ and $\rho \in C(\Omega_-)$, we have $\rho \equiv 0$ on $\Omega_-$. This is contradicted to the  definition of the set $\Omega_-$ in \eqref{Omega-}. Namely, $\Omega_- = \emptyset$. As a result, $\rho (t,x,z) \geq 0$ for all $(t,x,z) \in [0,T) \times \T^3 \times \T_w $.
	
	\vspace*{3mm}
	{\em Step 2. Conservation laws.}
	\vspace*{3mm}
	
	The next is aimed at justifying the conservation laws of $\rho$ and $n$ under the further initial assumption \eqref{IC-Average}. One integrates, respectively, the first and the third equations of \eqref{EECPK} over $\T^3 \times \T_w$ and $\T^3$. Then there hold
	\begin{equation*}
		\begin{aligned}
			& \frac{\d}{\d t} \iint_{ \T^3 \times \T_w } \rho \d z \d x = \gamma \iint_{ \T^3 \times \T_w } n \rho \d z \d x \,, \\
			& \frac{\d}{\d t} \int_{ \T^3 } n \d x = - \xi \int_{ \T^3 } \vr n \d x = - \xi \iint_{ \T^3 \times \T_w } n \rho \d z \d x \,,
		\end{aligned}
	\end{equation*}
    which means that
	\begin{equation*}
		\begin{aligned}
			\frac{\d}{\d t} \left( \iint_{ \T^3 \times \T_w } \rho (t,x,z) \d z \d x + \frac{\gamma}{\xi} \int_{ \T^3 } n (t,x) \d x \right) = 0 \,.
		\end{aligned}
	\end{equation*}
	Recall that $\rho, n \geq 0$ provided that $\rho^{\eps,in}, n^{\eps, in} \geq 0$. It therefore infers that
	\begin{equation*}
		\begin{aligned}
			\frac{\d}{\d t} \int_{ \T^3 } n \d x = - \xi \iint_{ \T^3 \times \T_w } n \rho \d z \d x \leq 0 \,,
		\end{aligned}
	\end{equation*}
    hence, it is a decreasing function in $t \geq 0$. Together with $\int_{ \T^3 } n \d x \geq 0$ and $\int_{ \T^3 } n^{\eps,in} \d x = 0$ (see \eqref{IC-Average}), one sees $\int_{ \T^3 } n \d x \equiv 0$. In other words, $\frac{\d}{\d t} \iint_{ \T^3 \times \T_w } \rho (t,x,z) \d z \d x = 0$. This finishes the proof of part (2) in Theorem \ref{Thm1} under the initial assumption \eqref{IC-Average}.
\end{proof}

\subsection{\bf Long time existence around $(0,0,0)$}\label{Subsec2.4}

In this subsection, the goal is to prove the global-in-time existence of the \eqref{EECPK} around the steady state $(0,0,0)$ under the initial assumptions \eqref{IC-K1}-\eqref{IC-K2}-\eqref{IC-Average} with $\rho_0 = 0$. Hence the proof of part (3) of Theorem \ref{Thm1} will be finished. Especially, the necessity of the initial hypothesis \eqref{IC-Average} will be illustrated later.

Roughly speaking, while proving the global-in-time existence around a steady state, one should derive a form of energy differential inequality
\begin{equation*}
	\begin{aligned}
		\tfrac{\d}{\d t} \mathcal{E}_{nergy} (t) + \mathcal{D}_{issipative} (t) \leq \mathbb{P} ( \mathcal{E}_{nergy} (t) ) \mathcal{D}_{issipative} (t) \,,
	\end{aligned}
\end{equation*}
where $\mathbb{P} (\cdot)$ is a strictly increasing function and $\mathbb{P} ( 0 ) = 0$. Here $\mathcal{E}_{nergy} (t)$ and $\mathcal{D}_{issipative} (t)$ represent the abstract forms of energy functional and dissipative rate, respectively. Once the initial energy $\mathcal{E}_{nergy} (0)$ is sufficiently small such that so $\mathbb{P} ( \mathcal{E}_{nergy} (0) )$ is, one can derive the quantity $\mathcal{E}_{nergy} (t) + \int_0^t \mathcal{E}_{nergy} (t') \d t'$ admits a upper bound uniform in time $t \geq 0$ by continuity arguments. This concludes our global-in-time existence around a steady state.

In the \eqref{EECPK} system, as shown in Lemma \ref{lemma 2.1} above,
\begin{equation}
	\begin{aligned}
		\mathcal{E}_{nergy} (t) \thicksim & \| \rho \|^2_{H^s_{x,z}} + \| \vr \|^2_{H^s_{x}} + \| h \|^2_{H^s_{x}} + \| n \|^2_{H^s_{x}} \,, \\
		\mathcal{D}_{issipative} (t) \thicksim & \| \nabla_x \rho \|^2_{H^s_{x,z}(D)} + \| \nabla_x h \|^2_{H^s_{x}} + \| h \|^2_{H^s_{x}} + \| \nabla_x \vr \|^2_{H^s_{x}} + \| \nabla_x n \|^2_{H^s_{x}} \,.
	\end{aligned}
\end{equation}
While doing $L^2$ estimates on the $\rho$ and $n$ equations in \eqref{EECPK}, there hold
\begin{align*}
	\tfrac{1}{2} \tfrac{\d}{\d t} \| \rho \|_{L^2_{x,z}}^2 + \| \nabla_x \rho \|_{L^2_{x,z}(D)}^2 = & - \tfrac{1}{\eps} \l \p_z ( g \rho ) , \rho \r_{L_{x,z}^2} + \gamma \l n \rho , \rho \r_{L_{x,z}^2} \,, \\
	\tfrac{1}{2} \tfrac{\d}{\d t} \| n \|_{L^2_{x}}^2 + D_n \| \nabla_x n \|_{L^2_{x}}^2 = & - \xi \l \vr n, n \r_{L_{x}^2} = - \xi \l n \rho , n \r_{L_{x,z}^2} \,.
\end{align*}
Here the quantities $\gamma \l n \rho , \rho \r_{L_{x,z}^2}$ and $- \xi \l n \rho , n \r_{L_{x,z}^2}$ should be hoped to be bounded by
$$\| n \|_{H^s_x} \| \nabla_x \rho \|^2_{H^s_{x,z} (D)} \quad \textrm{and} \quad \| \rho \|_{H^s_{x,z}} \| \nabla_x n \|^2_{H^s_x} \,,$$
respectively, so that they can be both controlled by $[\mathcal{E}_{nergy} (t) ]^\frac{1}{2} \mathcal{D}_{issipative} (t)$. However, this shall be yielded by the Poincar\'e inequality with zero-average conditions, i.e., $\iint_{\T^3 \times \T_w} \rho \d z \d x = \int_{ \T^3 } n \d x = 0$. But these two conditions can not be assumed in advance. They should be derived from some initial assumptions. Thus the initial conditions \eqref{IC-K2}-\eqref{IC-Average} with $\rho_0 = 0$ are necessary.

As indicated in part (2) of Theorem \ref{Thm1}, the initial conditions \eqref{IC-K2} and \eqref{IC-Average}, i.e., $\rho^{\eps,in}, h^{\eps, in}, n^{\eps,in} \geq 0$ and $\iint_{ \T^3 \times \T_w } \rho^{\eps, in} \d z \d x = \int_{ \T^3 } n^{\eps, in} \d x = 0$, imply
\begin{equation*}
	\begin{aligned}
		\rho (t,x,z) \geq 0 \,, \ n (t,x) \geq 0 \,, \ \iint_{ \T^3 \times \T_w } \rho (t,x,z) \d z \d x = 0 \,, \ \int_{ \T^3 } n (t,x) \d x = 0 \,.
	\end{aligned}
\end{equation*}
These mean that $\rho (t,x,z) \equiv 0$ and $n (t,x) \equiv 0$. The \eqref{EECPK} system thereby reduces to
\begin{equation*}
	\left\{
	  \begin{aligned}
	  	\partial_t h = & D_h \Delta_x h - \beta h \,, \\
	  	h (0,x) = & h^{\eps,in} (x) \geq 0 \,.
	  \end{aligned}
	\right.
\end{equation*}
Once the notations $\mathscr{E}_s (\rho, h, n) $ and $\mathscr{D}_s (\rho, h, n) $ are still employed even $\rho = n = 0$, one can easily derive
\begin{equation*}
	\begin{aligned}
		\sup_{ t \geq 0 } \, \mathscr{E}_s (\rho, h, n) (t) + \int_0^\infty \mathscr{D}_s (\rho , h , n) (\varsigma) \d \varsigma \leq \mathcal{C}_1 \mathscr{E}_s (\rho^{\eps, in}, h^{\eps, in}, n^{\eps, in})
	\end{aligned}
\end{equation*}
with some constant $C_1 > 0$ uniform in $t \geq 0$. Consequently, the proof of part (3) in Theorem \ref{Thm1} is finished.

\section{Well-posedness of \eqref{AD-EECP} model: Proof of Theorem \ref{Thm3}}\label{Sec:AD-EECP}

This section is aimed at proving Theorem \ref{Thm3}. Namely, we will first employ the energy method to verify the local in time existence with large initial data. Then the positivity of the solution constructed in part (1) of Theorem \ref{Thm3} will be proved under the condition \eqref{IC-A2}. Finally, under the further initial condition \eqref{IC-Average-a} with $\vr_0 = \vr_a \geq 0$, one will justify the global-in-time existence of \eqref{AD-EECP} system around the steady state $(\vr_a, h_a, 0)$ with $h_a = \frac{\alpha}{\beta} \vr_a \geq 0$, i.e., the part (3) of Theorem \ref{Thm3}.

\subsection{A priori estimates to \eqref{AD-EECP} model}\label{A_Priori to the limit equation}

In this subsection, the a priori estimate to the system \eqref{AD-EECP} will be accurately derived from employing the energy method. We now introduce the following energy functional $E_l (t)$, $\mathbb{E}_l (t)$, and energy dissipative rate $D_l (t)$, $\mathbb{D}_l (t)$:
\begin{equation}\label{AD-Loc-Ener}
  \begin{aligned}
    E_l (t) & = \| \vr \|_{H^s_x}^2 + \| h \|_{H^{s}_{x}}^2 + \| n \|_{H^{s}_{x}}^2 \,, \\
    D_l (t) & = \tfrac{1}{2} \| \nabla_x \vr \|^2_{H^s_x (\widetilde{D})} + D_h \| \nabla_x h \|^2_{H^s_x} + \tfrac{\beta}{2} \| h \|^2_{H^s_x} + D_n \| \nabla_x n \|^2_{H^s_x} \,, \\
    \mathbb{E}_l (t) & = \| \widetilde{D} (h) \|_{H^s_x}^2 \,, \ \mathbb{D}_l (t) = D_h \| \nabla_x \widetilde{D} (h) \|_{H^s_x}^2 \,.
  \end{aligned}
\end{equation}

Then we have the following lemma.

\begin{lemma}\label{Lmm-AD}
  Let fixed integer $s\geq 4$ and $\widetilde{D} (\cdot) $ satisfy the assumption $(\mathrm{H}2)$. Assume that $( \vr , h )$ is a sufficiently smooth solution to system \eqref{AD-EECP} on the interval $[0,T]$. Then
\begin{align*}
  & \tfrac{\d}{\d t} E_l (t) + D_l (t) \leq C E_l (t) \mathbb{D}_l (t) + C \big( 1 + E_l^s (t) \big) E_l (t) \,, \\
  & \tfrac{\d}{\d t} \mathbb{E}_l (t) + \mathbb{D}_l (t) \leq C ( 1 + E_l^{s+1} (t) )
\end{align*}
holds for all $t\in[0,T]$ and some constant $C > 0$.
\end{lemma}

Moreover, the following lemma will be frequently used when estimating the quantities $\p^k_x \widetilde{D} (h)$ with $1 \leq |k| \leq s$.

\begin{lemma}\label{Lmm-D(h)}
	Assume that $\widetilde{D} (\cdot) $ satisfies the assumption $(\mathrm{H}2)$. Then, for all $1 \leq |k| \leq s \, (s \geq 3)$, there hold
	\begin{equation*}
	  \begin{aligned}
	    \| \p_x^k \widetilde{D} (h_0+h) \|_{L^2_x} & \lesssim ( 1 + \| h \|^{s-1}_{H_x^{s-1}} ) \| \nabla_x h \|_{H_x^{s-1}} \,, \\
	    \| \p_x^k \widetilde{D} (h_0+h) \|_{L^4_x} & \lesssim ( 1 + \| h \|^{s-1}_{H_x^s} ) \| \nabla_x h \|_{H_x^s} \,, \\
	    \| \p_x^k \widetilde{D} (h_0+h) \|_{L^\infty_x} & \lesssim ( 1 + \| h \|^{s-1}_{H_x^{s+1}} ) \| \nabla_x h \|_{H_x^{s+1}} \,.
	  \end{aligned}
	\end{equation*}
    for any constant $h_0 \geq 0$.
\end{lemma}

The proof of Lemma \ref{Lmm-D(h)} can be directly verified by the Sobolev theory. The details are omitted for simplicity.

\begin{proof}[Proof of Lemma \ref{Lmm-AD}]

The $L^2$ estimates will be first derived, which indicates the major structures of the energy functional. Then the higher order energy bounds will be obtained, which shall be consistent with the structures of $L^2$-estimates.

\vspace*{2mm}

{\bf Step 1. $L^2$ estimates.}

\vspace*{2mm}

From taking $L^2$-inner product with ${\vr}(t,x)$ on the first equation of system \eqref{AD-EECP} and integrating by parts over $x\in \T^3$, one has
\begin{align}\label{l1}
   \tfrac{1}{2} \tfrac{\d}{\d t} \| \vr \|^2_{L^2_x} + \| \nabla_x \vr \|^2_{L^2_x (\widetilde{D})} = \l \vr \nabla_x \widetilde{D} (h), \nabla_x \vr \r_{L_x^2} + \gamma \l \vr n , \vr \r_{L_x^2}  \,,
\end{align}
By the Sobolev embedding $H_x^2 (\mathbb{T}^3) \hookrightarrow L_x^\infty (\mathbb{T}^3)$, $H_x^1 (\mathbb{T}^3) \hookrightarrow L_x^4 (\mathbb{T}^3)$, the fact $\widetilde{D} (h) \geq d > 0$ and the H\"older inequality, the right-hand side of \eqref{l1} can be bounded by
\begin{equation*}
  \begin{aligned}
    \tfrac{1}{2}\| \nabla_x \vr \|_{L^2_{x}(\widetilde{D})}^2 + C \| \vr\|_{H^2_{x}}^2 \| \nabla_x \widetilde{D} (h) \|_{L^2_{x}}^2 + C \| n \|_{L^2_{x}} \| \vr \|_{H^1_{x}}^2 \,.
  \end{aligned}
\end{equation*}
Consequently,
\begin{align}\label{equation 5.1}
  \tfrac{1}{2} \tfrac{\d}{\d t} \| \vr \|_{L^2_{x}}^2 + \tfrac{1}{2} \| \nabla_x \vr \|_{L^2_{x}(\widetilde{D})}^2 \leq C \| \vr\|_{H^2_{x}}^2 \| \nabla_x \widetilde{D} (h) \|_{L^2_{x}}^2 + C \| n \|_{L^2_{x}} \| \vr \|_{H^1_{x}}^2 \,.
\end{align}
The next is to take $L^2$-inner product with $h(t,x)$ on the second equation of system \eqref{AD-EECP} and to integrate by parts over $x \in \T^3$. It infers that
\begin{align*}
  \tfrac{1}{2} \tfrac{\d}{\d t} \| h \|_{L^2_{x}}^2 + D_h \| \nabla_x h \|_{L^2_{x}}^2 + \beta \| h \|_{L^2_{x}}^2 = \alpha \l \vr, h \r_{L^2_x} \leq \tfrac{\beta}{2} \| h \|_{L^2_{x}}^2 + \tfrac{\alpha^2}{2\beta} \| \vr \|_{L^2_{x}}^2 \,.
\end{align*}
Namely,
\begin{align}\label{equation 5.2}
  \tfrac{1}{2} \tfrac{\d}{\d t} \| h \|_{L^2_{x}}^2 + D_h \| \nabla_x h \|_{L^2_{x}}^2 + \tfrac{\beta }{2} \| h \|_{L^2_{x}}^2 \leq \tfrac{\alpha^2}{2\beta} \| \vr \|_{L^2_{x}}^2 \,.
\end{align}
Finally, from taking $L^2$-inner product with $n(t,x)$ on the third equation of system \eqref{AD-EECP} and integrating by parts over $x\in \mathbb{T}^3$,
\begin{align}\label{equation 5.3}
  \tfrac{1}{2} \tfrac{\d}{\d t} \| n \|_{L^2_x}^2 + D_n \| \nabla_x n \|_{L^2_x}^2 = - \xi \l \vr n, n \r_{L_x^2} \leq C \| \vr \|_{L^2_x} \| n \|_{H^1_x}^2 \,,
\end{align}
where the last inequality is derived from the Sobolev embedding theory. Together with the inequalities \eqref{equation 5.1}-\eqref{equation 5.2}-\eqref{equation 5.3}, there holds
\begin{align}\label{equation 5.4}
  & \tfrac{1}{2} \tfrac{\d}{\d t} ( \| \vr \|_{L^2_{x}}^2 + \| h \|_{L^2_{x}}^2 + \| n \|_{L^2_{x}}^2)\no \\
   & + \tfrac{1}{2} \| \nabla_x \vr \|_{L^2_{x} (\widetilde{D})}^2 + D_h \| \nabla_x h \|_{L^2_{x}}^2 + \tfrac{\beta}{2} \| h \|_{L^2_{x}}^2 + D_n \| \nabla_x n \|_{L^2_{x}}^2 \no \\
   & \lesssim E_l (t) \mathbb{D}_l (t) + E_l (t) + E_l^\frac{3}{2} (t) \,.
\end{align}
where the symbols $E_l (t)$ and $\mathbb{D}_l (t)$ are defined in \eqref{AD-Loc-Ener}.

There is a trouble quantity $C \| \nabla_x \widetilde{D} (h) \|_{L^2_{x}}^2 \| \vr \|_{H^2_{x}}^2$ coming from the diffusion term $\Delta_x ( \widetilde{D} (h) \vr )$ when proving the local solution with large initial data. The natural way to control this quantity is to dominate it by $C \sup_{\zeta \in \R } |\widetilde{D}' (\zeta)| \| \nabla_x h \|^2_{L^2_x} \| \vr \|_{H^2_x}^2 $, where the norm $\| \vr \|_{H^2_x}^2$ will be regarded as a part of energy. However, the norm $\| \nabla_x h \|^2_{L^2_x}$ can only be absorbed by the dissipative term $D_h \| \nabla_x h \|^2_{L^2_x}$ under the small size of the norm $\| \vr \|^2_{H^2_x}$, in which the small initial data are required. This is inconsistent with our goal of proving a large initial local solution. In order to overcome this difficulty, we try to seek an extra energy inequality associated with the composed function $\widetilde{D} (h)$. More precisely, we multiply the second $h$-equation of \eqref{AD-EECP} by $\widetilde{D}' (h)$ and then obtain
\begin{align}\label{equation 5.5}
  \p_t \widetilde{D} (h) = D_h \Delta_x \widetilde{D} (h) - D_h \widetilde{D}''(h) |\nabla_x h|^2 - \beta \widetilde{D}' (h) h + \alpha \vr \widetilde{D}' (h) \,,
\end{align}
where the relation
\begin{align*}
  \Delta_x \widetilde{D} (h) = \widetilde{D}' (h) \Delta_x h + \widetilde{D}'' (h) |\nabla_x h|^2
\end{align*}
is utilized. By taking $L^2$-inner product with $\widetilde{D}(h)$ on the equation \eqref{equation 5.5} and integrating by parts over $x\in \T^3$, one has
\begin{equation*}
  \begin{aligned}
    \tfrac{1}{2} \tfrac{\d}{\d t} \| \widetilde{D} (h)  \|_{L^2_{x}}^2 & + D_h \| \nabla_x \widetilde{D} (h) \|_{L^2_{x}}^2 \\
    & = - D_h \l \widetilde{D}'' (h) |\nabla_x h|^2 , \widetilde{D} (h) \r_{L^2_x} - \beta \l \widetilde{D}' (h) h, \widetilde{D} (h) \r_{L^2_x} + \alpha \l \vr \widetilde{D}' (h) , \widetilde{D} (h) \r_{L^2_x} \,.
  \end{aligned}
\end{equation*}
From the H\"older inequality, the Sobolev embedding theory and the hypotheses \eqref{equation 1.12} (i.e., (H2)), the right-hand side of the previous equality can be bounded by
\begin{equation*}
  \begin{aligned}
    \| \nabla_x h \|^2_{L^2_x} + \| h \|_{L^2_x} + \| \vr \|_{L^2_x} \lesssim E_l^\frac{1}{2} (t) + E_l (t) \lesssim 1 + E_l (t) \,.
  \end{aligned}
\end{equation*}
Therefore,
\begin{equation}\label{equation 5.6}
  \begin{aligned}
    \tfrac{1}{2} \tfrac{\d}{\d t} \| \widetilde{D} (h) \|_{L^2_{x}}^2 + D_h \| \nabla_x \widetilde{D}  (h) \|_{L^2_{x}}^2 \lesssim 1 + E_l (t) \,.
  \end{aligned}
\end{equation}

\vspace*{2mm}

{\bf Step 2. Higher order energy estimates.}

\vspace*{2mm}

First, we act $k$-order derivative operators $\partial_x^k$ on the first equation of system \eqref{AD-EECP} for all $1 \leq |k|\leq s$, take $L^2$-inner product by dot with $\p_x^k \vr $ and integrate by parts over $x \in \T^3$. It infers that
\begin{equation}\label{J1J2J3J4}
  \begin{aligned}
    \tfrac{1}{2} \tfrac{\d}{\d t} \| \p^k_x \vr \|^2_{L^2_x} + & \| \nabla_x \p^k_x \vr \|^2_{L^2_x (\widetilde{D})} \\
    = & \underbrace{ - \l \vr \nabla_x \partial_x^k \widetilde{D} (h) , \nabla_x \p^k_x \vr \r_{L^2_x} }_{J_1} + \underbrace{ \l W_k (\vr, h) , \nabla_x \p^k_x \vr \r_{L^2_x} }_{J_2} + \underbrace{ \gamma \l \p^k_x (n \vr) , \p^k_x \vr \r_{L^2_x} }_{J_3} \,,
  \end{aligned}
\end{equation}
where $W_k (\vr, h) = \vr \nabla_x \partial_x^k \widetilde{D} (h) - [\nabla_x \partial_x^k, \widetilde{D} (h)] \vr$ and the communicator operator $[X,Y]=XY-YX$ is employed. By the H\"older inequality, the lower bound $\widetilde{D} (h) \geq d > 0$ in \eqref{equation 1.12} and the Sobolev embedding theory, the quantity $J_1$ can be bounded by
\begin{equation}\label{J1}
	\begin{aligned}
		J_1 \leq \tfrac{1}{4} \| \nabla_x \partial_x^k \vr \|^2_{L^2_x (\widetilde{D})} + C \| \vr \|_{H^2_x}^2 \| \nabla_x \widetilde{D} (h) \|^2_{H^s_x} \leq \tfrac{1}{4} \| \nabla_x \partial_x^k \vr \|^2_{L^2_x (\widetilde{D})} + C E_l (t) \mathbb{D}_l (t) \,.
	\end{aligned}
\end{equation}
Here the symbols $E_l (t)$ and $\mathbb{D}_l (t)$ are defined in \eqref{AD-Loc-Ener}. Note that the $L^2$ norm of $W_k (\vr , h)$ with $1 \leq |k| \leq s$ can be bounded by
\begin{equation*}
	\begin{aligned}
		\| W_k (\vr, h) \|_{L^2_x} \lesssim \| \vr \|_{H^s_x} (1 + \| \nabla_x h \|^{s-1}_{H^{s-1}_x}) \| \nabla_x h \|_{H^{s-1}_x} \lesssim (1 + E_l^\frac{s-1}{2} (t) ) E_l (t) \,,
	\end{aligned}
\end{equation*}
where Lemma \ref{Lmm-D(h)}, the H\"older inequality, the Sobolev embedding theory and the Young's inequality are utilized. There therefore hold
\begin{equation}\label{J2}
  \begin{aligned}
    J_2 \leq & \tfrac{1}{4} \| \nabla_x \partial_x^k \vr \|^2_{L^2_x} + C \| W_k (\vr, h) \|_{L^2_x}^2 \\
    \leq & \tfrac{1}{4} \| \nabla_x \p^k_x \vr \|^2_{L^2_x (\widetilde{D})} + C (1 + E_l^{s-1} (t) ) E_l^2 (t) \,,
  \end{aligned}
\end{equation}
where the lower bound $\widetilde{D} (h) \geq d > 0$ in \eqref{equation 1.12} is also used. Moreover, the term $J_3$ can be dominated by
\begin{equation}\label{J3}
  \begin{aligned}
    J_3 \lesssim ( \| \vr \|_{H^s_x} + \| n \|_{H^s_x} ) ( \| \nabla_x n \|^2_{H^{s-1}_x} + \| \nabla_x \vr \|^2_{H^{s-1}_x} ) \lesssim E_l^\frac{3}{2} (t) \,,
  \end{aligned}
\end{equation}
where the H\"older inequality and the calculus inequality $\| f g \|_{H^s_x} \lesssim \| f \|_{H^s_x} \| g \|_{H^s_x} \ (\forall s \geq 2)$ have been utilized. It is derived from plugging the bounds \eqref{J1}, \eqref{J2} and \eqref{J3} into \eqref{J1J2J3J4} that
\begin{equation}\label{equation 5.12}
  \begin{aligned}
    \tfrac{1}{2} \tfrac{\d}{\d t} \| \p^k_x \vr \|^2_{L^2_x} + \tfrac{1}{2} \| \nabla_x \p^k_x \vr \|^2_{L^2_x (\widetilde{D})} \lesssim E_l (t) \mathbb{D}_l (t) + (1 + E_l^s (t)) E_l (t)
  \end{aligned}
\end{equation}
for all $1 \leq |k| \leq s$.

Next we act $k$-order derivative operator on the second equation of system \eqref{AD-EECP} for all $1 \leq |k| \leq s$, take $L^2$-inner product by dot with $\p_x^k h$ and integrate by parts over $x \in \T^3$. We thereby have
\begin{equation*}
  \begin{aligned}
    \tfrac{1}{2} \tfrac{\d}{\d t} \| \p^k_x h \|^2_{L^2_x} + D_h \| \nabla_x \p^k_x h \|^2_{L^2_x} + \beta \| \p^k_x h \|^2_{L^2_x} = \alpha \l \p^k_x \vr , \p^k_x h \r_{L^2_x} \,.
  \end{aligned}
\end{equation*}
The H\"older inequality and the Young's inequality reduce to
\begin{equation*}
  \begin{aligned}
    \alpha \l \p^k_x \vr , \p^k_x h \r_{L^2_x} \leq \tfrac{\beta}{2} \| \p^k_x h \|^2_{L^2_x} + \tfrac{\alpha^2}{2 \beta} \| \p^k_x \vr \|^2_{L^2_x} \leq \tfrac{\beta}{2} \| \p^k_x h \|^2_{L^2_x} + C \| \nabla_x \vr \|^2_{H^{s-1}_x} \,.
  \end{aligned}
\end{equation*}
Consequently, we have
\begin{equation}\label{equation 5.13}
  \begin{aligned}
    \tfrac{1}{2} \tfrac{\d}{\d t} \| \p^k_x h \|^2_{L^2_x} + D_h \| \nabla_x \p^k_x h \|^2_{L^2_x} + \tfrac{\beta}{2} \| \p^k_x h \|^2_{L^2_x} \lesssim \| \nabla_x \vr \|^2_{H^{s-1}_x} \lesssim E_l (t)
  \end{aligned}
\end{equation}
for all $1 \leq |k| \leq s$.

Finally, from acting $\p_x^k$ on the third equation of system \eqref{AD-EECP} for all $1 \leq |k| \leq s $, taking $L^2$-inner product by dot with $\p_x^k n$ and integrating by parts over $x \in \T^3$, we derive
\begin{equation}\label{l5}
  \begin{aligned}
    \tfrac{1}{2} \tfrac{\d}{\d t} \| \p^k_x n \|^2_{L^2_x} + D_n \| \nabla_x \p^k_x n \|^2_{L^2_x} = - \xi \l \p^k_x (\vr n) , \p^k_x n \r_{L^2_x} \,.
  \end{aligned}
\end{equation}
It is derived from the H\"older inequality and the calculus inequality $\| f g \|_{H^s_x} \lesssim \| f \|_{H^s_x} \| g \|_{H^s_x} (\forall s \geq 2)$ that
\begin{equation*}\label{l6}
  \begin{aligned}
    - \xi \l \p^k_x (\vr n) , \p^k_x n \r_{L^2_x} \lesssim ( \| n \|_{H^s_x} + \| \vr \|_{H^s_x} ) ( \| \nabla_x n \|^2_{H^{s-1}_x} + \| \nabla_x \vr \|^2_{H^{s-1}_x} ) \lesssim E_l^\frac{3}{2} (t) \,.
  \end{aligned}
\end{equation*}
We therefore obtain
\begin{equation}\label{equation 5.14}
  \begin{aligned}
    & \tfrac{1}{2} \tfrac{\d}{\d t} \| n \|^2_{H^s_x} + D_n \| \nabla_x n \|^2_{H^s_x} \lesssim E_l^\frac{3}{2} (t)
  \end{aligned}
\end{equation}
for all $1 \leq |k| \leq s$. Recalling the definitions of $E_l (t)$, $D_l (t)$ and $\mathbb{D}_l (t)$ in \eqref{AD-Loc-Ener}, together with the inequalities \eqref{equation 5.4}, $\eqref{equation 5.12}$, $\eqref{equation 5.13}$ and $\eqref{equation 5.14}$, one has
\begin{equation}\label{equation 5.15}
  \begin{aligned}
    \tfrac{1}{2} \tfrac{\d}{\d t} E_l (t) + D_l (t) \lesssim E_l (t) \mathbb{D}_l (t) + \big( 1 + E_l^s (t) \big) E_l (t) \,.
  \end{aligned}
\end{equation}

One next will deal with the quantity $\mathbb{D}_l (t) : = D_h \| \nabla_x \widetilde{D} (h) \|^2_{H^s_x}$. From applying the $k$-order derivative operator to the evolution of $\widetilde{D} (h)$ in \eqref{equation 5.5} for all $1 \leq |k| \leq s$, taking $L^2$-inner product by dot with $\p_x^k \widetilde{D} (h)$ and integrating by parts over $x\in \T^3$, we derive
\begin{equation}\label{K1K2K3K4}
  \begin{aligned}
    & \tfrac{1}{2} \tfrac{\d}{\d t} \| \p^k_x \widetilde{D} (h) \|^2_{L^2_x} + D_h \| \nabla_x \p^k_x \widetilde{D} (h) \|^2_{L^2_x} \\
    = & \underbrace{ - D_h \sum_{0 \neq a \leq k} \l \p^a_x \widetilde{D}'' (h) \p^{k-a}_x |\nabla_x h|^2, \p^k_x \widetilde{D} (h) \r_{L^2_x} }_{K_1} \ \underbrace{ - \beta \l \p^k_x (\widetilde{D}' (h) h), \p^k_x \widetilde{D} (h) \r_{L^2_x}}_{K_2} \\
    & + \underbrace{ \alpha \l \p^k_x ( \vr \widetilde{D}' (h) ), \p^k_x \widetilde{D} (h) \r_{L^2_x} }_{K_3} \ \underbrace{- D_h \l \widetilde{D}'' (h) \p_x^k | \nabla_x h |^2, \p_x^k \widetilde{D} (h) \r_{L^2_x} }_{K_4} \,.
  \end{aligned}
\end{equation}

By applying the H\"older inequality, Lemma \ref{Lmm-D(h)} and the calculus inequality $\| f g \|_{H^N_x} \lesssim \| f \|_{H^N_x} \| g \|_{H^N_x} \, (\forall N \geq 2)$, the quantity $K_1$ can be bounded by
\begin{equation}\label{K1}
  \begin{aligned}
    K_1 \lesssim ( 1 + \| h \|^{s-1}_{H^{s-1}_x} )^2 \| \nabla_x h \|^2_{H^{s-1}_x} \| | \nabla_x h |^2 \|_{H^{s-1}_x} \lesssim ( 1 + E_l^{s-1} (t) ) E_l^2 (t) \,.
  \end{aligned}
\end{equation}
It is derived from the H\"older inequality, Lemma \ref{Lmm-D(h)} and the conditions \eqref{equation 1.12} that
\begin{equation}\label{K2}
  \begin{aligned}
    K_2 + K_3 \lesssim & ( 1 + \| h \|^{s-1}_{H^{s-1}_x} ) \| \nabla_x h \|_{H^{s-1}_x} \| h \|_{H^s_x} + ( 1 + \| h \|^{s-1}_{H^{s-1}_x} )^2 \| \nabla_x h \|^2_{H^{s-1}_x} \| h \|_{H^s_x} \\
    \lesssim & (1 + E_l^{s-1} (t) ) E_l (t) \,.
  \end{aligned}
\end{equation}
Moreover, by integration by parts, the quantity $K_4$ can be decomposed as four parts:
\begin{equation}\label{K41+K42+K43+K44}
  \begin{aligned}
    K_4 = \underbrace{ 2 D_h \l \widetilde{D}''' (h) |\nabla_x h|^2 \p^k_x h, \p^k_x \widetilde{D} (h) \r_{L^2_x} }_{K_{41}} + \underbrace{ 2 D_h \l \widetilde{D}'' (h) \p^k_x h \nabla_x h, \nabla_x \p^k_x \widetilde{D} (h) \r_{L^2_x} }_{K_{42}} \\
    + \underbrace{ 2 D_h \l \widetilde{D}'' (h) \p^k_x h \Delta_x h, \p^k_x \widetilde{D} (h) \r_{L^2_x} }_{K_{43}} \underbrace{ - \sum_{0 \neq a < k} D_h \l \widetilde{D}' (h) \nabla_x \p^a_x h \cdot \nabla_x \p^{k-a}_x h, \p^k_x \widetilde{D} (h) \r_{L^2_x} }_{K_{44}} \,.
  \end{aligned}
\end{equation}
The term $K_{41}$ can by bounded by
\begin{equation}\label{K41}
  \begin{aligned}
    K_{41} \lesssim & \| \nabla_x h \|^2_{H^2_x} \| h \|_{H^s_x} ( 1 + \| h \|^{s-1}_{H^{s-1}_x} ) \| \nabla_x h \|_{H^{s-1}_x} \lesssim ( 1 + E_l^\frac{s-1}{2} ) E_l^2 (t) \,,
  \end{aligned}
\end{equation}
where the H\"older inequality, the Sobolev embedding $H^2_x (\T^3) \hookrightarrow L^\infty_x (\T^3)$, the condition \eqref{equation 1.12} and Lemma \ref{Lmm-D(h)} have been used. Moreover, the term $K_{42}$ can be controlled by
\begin{equation}\label{K42}
  \begin{aligned}
    K_{42} \leq & C \| \nabla_x h \|_{H^2_x} \| h \|_{H^s_x} \| \nabla_x \p^k_x \widetilde{D} (h) \|_{L^2_x} \leq \tfrac{D_h}{2} \| \nabla_x \p^k_x \widetilde{D} (h) \|_{L^2_x}^2 + C E_l^2 (t) \,.
  \end{aligned}
\end{equation}
For the term $K_{43}$,
\begin{equation}\label{K43}
  \begin{aligned}
    K_{43} \lesssim & \| h \|_{H^s_x} \| \Delta_x h \|_{H^2_x} ( 1 + \| h \|^{s-1}_{H^{s-1}_x} ) \| \nabla_x h \|_{H^{s-1}_x} \lesssim ( 1 + E_l^\frac{s-1}{2} (t) ) E_l^\frac{3}{2} (t) \,,
  \end{aligned}
\end{equation}
where the inequality $\| \Delta_x h \|_{H^2_x} \leq \| h \|_{H^s_x}$ has been used in the last inequality, so that $s \geq 4$ is required. Furthermore, it is derived from the H\"older inequality, the condition \eqref{equation 1.12}, the Sobolev theory and Lemma \ref{Lmm-D(h)} that
\begin{equation}\label{K44}
  \begin{aligned}
    K_{44} \lesssim & \| h \|_{H^4_x} \| h \|_{H^s_x} ( 1 + \| h \|^{s-1}_{H^{s-1}_x} ) \| \nabla_x h \|_{H^{s-1}_x} + ( 1 + \| h \|^{s-1}_{H^{s-1}_x} ) \| \nabla_x h \|_{H^{s-1}_x}^3 \\
    \lesssim & ( 1 + E_l^\frac{s-1}{2} (t) ) E_l^\frac{3}{2} (t) \,,
  \end{aligned}
\end{equation}
where $s \geq 4$ is also required. Then we deduce from plugging \eqref{K41}, \eqref{K42}, \eqref{K43} and \eqref{K44} into \eqref{K41+K42+K43+K44} that
\begin{equation}\label{K4}
  \begin{aligned}
    K_4 \leq \tfrac{D_h}{2} \| \nabla_x \p^k_x \widetilde{D} (h) \|_{L^2_x}^2 + C ( 1 + E_l^\frac{s-1}{2} (t) ) E_l^2 (t) \,.
  \end{aligned}
\end{equation}
Together with the bounds \eqref{K1}, \eqref{K2} and \eqref{K4}, the relations \eqref{equation 5.6} and \eqref{K1K2K3K4} imply that
\begin{equation*}
  \begin{aligned}
    \tfrac{\d}{\d t} \| \widetilde{D} (h) \|^2_{H^s_x} + D_h \| \nabla_x \widetilde{D} (h) \|^2_{H^s_x} \lesssim ( 1 + E_l^s (t) ) E_l (t) + E_l^\frac{1}{2} \,.
  \end{aligned}
\end{equation*}
Recalling the definitions of $E_l (t)$, $\mathbb{E}_l (t)$ and $\mathbb{D}_l (t)$ in \eqref{AD-Loc-Ener}, one obtains
\begin{equation*}
  \begin{aligned}
    \tfrac{\d}{\d t} \mathbb{E}_l (t) + \mathbb{D}_l (t) \lesssim 1 + E_l^{s+1} (t) \,.
  \end{aligned}
\end{equation*}
Thus the proof of Lemma \ref{Lmm-AD} is finished.
\end{proof}

\subsection{Local well-posedness of \eqref{AD-EECP} system with large initial data}\label{Sec: Local-Result to the limit equation}

In this subsection, the goal is to justify the local well-posedness of the system \eqref{AD-EECP} with large initial data \eqref{IC-AD-EECP}, namely to prove the part (1) of Theorem \ref{Thm3}. The linear iterative approximate systems of \eqref{AD-EECP} are initially constructed as follows: for all integers $l \geq 0$,
\begin{align}\label{equation 6.1}
  \left\{
    \begin{array}{ll}
      \p_t \vr^{l+1} - \Delta_x ( \widetilde{D} (h^l) \vr^{l+1} ) = \gamma n^l \vr^{l+1} \,, \\
      \p_t h^{l+1} = D_h \Delta_x h^{l+1} - \beta h^{l+1} + \alpha \vr^{l+1} \,, \\
      \p_t n^{l+1} = D_n \Delta_x n^{l+1} - \xi \vr^l n^{l+1} \,, \\
      ( \vr^{l+1}, h^{l+1}, n^{l+1})|_{t=0} = ( \vr^{in}_a, h^{in}_a, n^{in}_a ) \,.
    \end{array}
  \right.
\end{align}
The iteration starts from
\begin{equation}\label{equation 6.2}
  ( \vr^0 (t,x), h^0 (t,x), n^0 (t,x) ) = ( \vr^{in}_a (x), h^{in}_a (x), n^{in}_a (x) ) \,.
\end{equation}

In the arguments proving the convergence ($l \rightarrow \infty$) of the approximate solutions \eqref{equation 3.1}, it is essential to obtain uniform (in $l \geq 0$) energy estimates of \eqref{equation 6.1} in a uniform lower bound lifespan time, whose derivations are the almost same as the derivations of the a priori estimates for the system \eqref{AD-EECP} with initial data \eqref{IC-AD-EECP} in Lemma \ref{Lmm-AD}. The arguments of the uniform lower bound lifespan time can be referred to \cite{JL-SIAM-2019}, for instance. The convergence arguments are a standard process. For simplicity, we will only consider the a priori estimates in Lemma \ref{Lmm-AD} for the smooth solutions of \eqref{AD-EECP}-\eqref{IC-AD-EECP} on some time interval.

From Lemma \ref{Lmm-AD},
\begin{equation}\label{AD-Loc-1}
  \left\{
    \begin{array}{l}
      \tfrac{\d}{\d t} E_l (t) + D_l (t) \leq C E_l (t) \mathbb{D}_l (t) + C ( 1 + E_l^s (t) ) E_l (t) \,, \\
      \tfrac{\d}{\d t} \mathbb{E}_l (t) + \mathbb{D}_l (t) \leq C ( 1 + E_l^{s+1} (t) ) \,.
    \end{array}
  \right.
\end{equation}
Observe that
\begin{equation*}
  \begin{aligned}
    & E_l (0) = E_l^{in} \,, \ \mathbb{E}_l (0) = \mathbb{E}_l^{in} : = \| \widetilde{D} (h^{in}_a) \|^2_{H^s_x} \leq C_1 (1 + (E_l^{in})^s) \,.
  \end{aligned}
\end{equation*}
Here $E_l^{in}$ is given in \eqref{IC-A1}. For some constant $\mathcal{C}_3 > \max \{ 2, e^{C \mathbb{E}_l^{in}} \} \geq 2$, we define a number
\begin{equation*}
  \begin{aligned}
    T^\star = \sup \{ \tau \geq 0, \sup_{0 \leq t \leq \tau} E_l (t) \leq \mathcal{C}_3 E_l^{in} \} \geq 0 \,.
  \end{aligned}
\end{equation*}
The continuity of $E_l (t)$ implies that $T^\star > 0$. Then the first inequality of \eqref{AD-Loc-1} implies that
\begin{equation}\label{AD-Loc-2}
  \begin{aligned}
    \tfrac{\d}{\d t} E_l (t) + D_l (t) \leq C \big[ 1 + \mathcal{C}_3^s ( E_l^{in} )^s + \mathbb{D}_l (t) \big] E_l (t)
  \end{aligned}
\end{equation}
for all $t \in [0, T^\star]$. Then we deduce from the Gr\"onwall inequality that
\begin{equation}\label{AD-Loc-3}
  \begin{aligned}
    E_l (t) \leq E_l^{in} \exp \big\{ C [ (1 + \mathcal{C}_3^s ( E_l^{in} )^s) t + \int_0^t \mathbb{D}_l (\tau) d \tau ] \big\} \,.
  \end{aligned}
\end{equation}
Furthermore, the second inequality of \eqref{AD-Loc-1} reduces to
\begin{equation}\label{AD-Loc-4}
  \begin{aligned}
    \mathbb{E}_l (t) + \int_0^t \mathbb{D}_l (\tau) d \tau \leq \mathbb{E}_l^{in} + C \big( 1 + \mathcal{C}_3^{s+1} ( E_l^{in} )^{s+1} \big) t
  \end{aligned}
\end{equation}
for $0 \leq t \leq T^\star$. It is implied by \eqref{AD-Loc-3} and \eqref{AD-Loc-4} that for $0 \leq t \leq T^\star$,
\begin{equation*}
  \begin{aligned}
    E_l (t) \leq E_l^{in} \exp \big\{ C \mathbb{E}_l^{in} + C [ 1 + \mathcal{C}_3^s ( E_l^{in} )^s + C ( 1 + \mathcal{C}_3^{s+1} (E_l^{in})^{s+1} ) ] t \big\} \,.
  \end{aligned}
\end{equation*}
If we require $ E_l^{in} \exp \big\{ C \mathbb{E}_l^{in} + C [ 1 + \mathcal{C}_3^s ( E_l^{in} )^s + C ( 1 + \mathcal{C}_3^{s+1} (E_l^{in})^{s+1} ) ] t \big\} < \mathcal{C}_3 E_l^{in} $, which means that
\begin{equation}\label{T0}
  \begin{aligned}
    t < T_0 : = \tfrac{\ln \mathcal{C}_3 - C \mathbb{E}_l^{in}}{ C [ 1 + \mathcal{C}_3^s ( E_l^{in} )^s + C ( 1 + \mathcal{C}_3^{s+1} (E_l^{in})^{s+1} ) ] } \,,
  \end{aligned}
\end{equation}
then the definition of $T^\star$ tells us that $ T^\star \geq T_0 > 0 $. Here the choice of $\mathcal{C}_3$ is such that $T_0 > 0$.

Therefore, for any fixed $0 < T < T_0$, $E_l (t) \leq \mathcal{C}_3 E_l^{in}$ holds for all $t \in [0, T]$, combining with \eqref{AD-Loc-1}, which means that
\begin{equation*}
  \begin{aligned}
    \sup_{0 \leq t \leq T} \mathbb{E}_l (t) + \int_0^T \mathbb{D}_l (t) d t \leq \mathbb{E}_l^{in} + C ( 1 + \mathcal{C}_3^{s+1} (E_l^{in})^{s+1} ) T \,,
  \end{aligned}
\end{equation*}
and
\begin{equation*}
  \begin{aligned}
    \sup_{0 \leq t \leq T} E_l (t) & + \int_0^T D_l (t) d t \leq E_l^{in} + C (1 + \mathcal{C}_3^s (E_l^{in})^s ) \mathcal{C}_3 E_l^{in} T \\
    & \leq E_l^{in} +  C \mathcal{C}_3 E_l^{in} \big[ \mathbb{E}_l^{in} + C ( 1 + \mathcal{C}_3^{s+1} (E_l^{in})^{s+1} ) T \big] + C \mathcal{C}_3 (1 + \mathcal{C}_3^s (E_l^{in})^s ) E_l^{in} T \,.
  \end{aligned}
\end{equation*}
Consequently, the proof of Part (1) of Theorem \ref{Thm3} is finished.

\subsection{Positivity and conservation laws: Proof of part (2) of Theorem \ref{Thm3}}

The goal of this subsection is to prove the positivity of the solution $(\vr, h, n)$ to \eqref{AD-EECP} system constructed in part (1) of Theorem \ref{Thm3} and to derive the conservation laws of $(\vr, n)$ under the initial assumptions \eqref{IC-A2} and \eqref{IC-Average}. Thanks to the similar arguments in proving part (2) of Theorem \ref{Thm1}, one can conclude the corresponding results of part (2) of Theorem \ref{Thm3}. For simplicity of presentation, the details are omitted here.

\subsection{Long time existence of \eqref{AD-EECP} system around $(\vr_a, h_a, 0)$}\label{Sec: Global-Result to the limit equation}

In this subsection, we will prove the global-in-time existence of the \eqref{AD-EECP} system around the steady state $(\vr_a, h_a, 0)$ with $h_a = \frac{\alpha}{\beta} \vr_a \geq 0$ under the initial hypotheses \eqref{IC-A1}-\eqref{IC-A2}-\eqref{IC-Average-a} with $\vr_0 = \vr_a$. As explained in Subsection \ref{Subsec2.4}, there is a form of energy differential inequality
\begin{equation*}
	\begin{aligned}
		\tfrac{\d}{\d t} \mathcal{E}_{nergy} (t) + \mathcal{D}_{issipative} (t) \leq \mathbb{P} ( \mathcal{E}_{nergy} (t) ) \mathcal{D}_{issipative} (t) \,,
	\end{aligned}
\end{equation*}
required to be derived, while proving the global-in-time solution around the steady state $(\vr_a, h_a, 0)$ before. In the \eqref{AD-EECP} system, as similarly shown in Lemma \ref{Lmm-AD} above,
\begin{equation}
	\begin{aligned}
		\mathcal{E}_{nergy} (t) \thicksim & \| \vr - \vr_a \|^2_{H^s_x} + \| h - h_a \|^2_{H^s_x} + \| n \|^2_{H^s_x} \,, \\
		\mathcal{D}_{issipative} (t) \thicksim & \| \nabla_x \vr \|^2_{H^s_x (\widetilde{D})} + \| \nabla_x h \|^2_{H^s_x} + \| h \|^2_{H^s_x} + \| \nabla_x n \|^2_{H^s_x} \,.
	\end{aligned}
\end{equation}
While doing $L^2$ estimates on the $\vr$ and $n$ equations in \eqref{AD-EECP}, there are two quantities $\gamma \l n (\vr - \vr_a), \vr - \vr_a \r_{L^2_x}$ and $- \xi \l n (\vr - \vr_a), n \r_{L^2_x}$, which should be bounded by $\| n \|_{H^s_x} \| \nabla_x \vr \|^2_{H^s_x(\widetilde{D})}$ and $\| \vr - \vr_a \|_{H^s_x} \| \nabla_x \vr \|_{H^s_x(\widetilde{D})} \| \nabla_x n \|^2_{H^s_x}$, respectively, so that they can be both dominated by $[\mathcal{E}_{nergy} (t) ]^\frac{1}{2} \mathcal{D}_{issipative} (t)$. These require the Poincar\'e inequality with zero-average conditions $\int_{\T^3} (\vr - \vr_a) \d x = \int_{\T^3} n \d x = 0$, which hold provided that the initial assumptions \eqref{IC-A2} and \eqref{IC-Average-a} with $\vr_0 = \vr_a$ are satisfied. Observe that the part (2) of Theorem \ref{Thm3} implies that continuous functions $\vr , n \geq 0$ and $\int_{\T^3} (\vr - \vr_a) \d x = \int_{\T^3} n \d x = 0$. These facts indicate that $\int_{\T^3} (\vr - \vr_a) \d x = 0$ and $n (t,x) \equiv 0$. Consequently, under \eqref{IC-A2} and \eqref{IC-Average-a} with $\vr_0 = \vr_a$, the \eqref{AD-EECP} system reduces to
\begin{equation}\label{AD-EECP-rho-h}
	\left\{
	\begin{array}{l}
		\p_t \vr = \Delta_x ( \widetilde{D} (h) \vr ) \,, \\[1.5mm]
		\p_t h = D_h \Delta_x h + \alpha \vr - \beta h \,,
	\end{array}
	\right.
\end{equation}
Therefore, the problem transforms to the global existence of \eqref{AD-EECP-rho-h} around the steady state $(\vr_a, h_a)$ with $h_a = \frac{\alpha}{\beta} \vr_a \geq 0$. Denote by
\begin{equation}\label{Fluctuation}
	\begin{aligned}
		\vr (t,x) = \vr_a + \Phi (t,x) \,, \ h (t,x) = h_a + \Psi (t,x) \,.
	\end{aligned}
\end{equation}
Then $(\Phi, \Psi)$ subjects to
\begin{equation}\label{AD-PhiPsi}
	\left\{
	  \begin{aligned}
	  	\partial_t \Phi = & \Delta_x \big[ \widetilde{D} (h) \Phi \big] + \vr_a \Delta_x \widetilde{D} (h) \,, \\
	  	\partial_t \Psi = & D_h \Delta_x \Psi + \alpha \Phi - \beta \Psi \,, \\
	  	h = & h_a + \Psi \,, \ \int_{ \T^3 } \Phi (t,x) \d x = 0 \,.
	  \end{aligned}
	\right.
\end{equation}
Define the global energy functional $E_g (t)$ and energy dissipative rate $D_g (t)$ as
\begin{equation}\label{EgDg}
  \begin{aligned}
    E_g (t) & = \| \Phi \|_{H^s_x}^2 + \| \Psi \|_{H^s_x}^2 + \| \widetilde{D} (h) - \widetilde{D} (h_a) \|_{H^s_x}^2 \,, \\
    D_g (t) & = \| \nabla_x \Phi \|_{H^s_x (\widetilde{D}) }^2 + D_h \| \nabla_x \Psi \|_{H^s_x}^2 +  \tfrac{\beta}{2}\| \Psi \|_{H^s_x}^2 + \|\nabla_x \widetilde{D} (h) \|_{H^s_x}^2 \,.
  \end{aligned}
\end{equation}

We first give a useful lemma to clarify the steady states $(\vr_a, h_a, 0)$ such that the \eqref{AD-EECP} system under initial assumptions \eqref{IC-A2}-\eqref{IC-Average-a} with $\vr_0 = \vr_a$ is dissipative around these steady states.
\begin{lemma}\label{Lmm-Dissi}
	Let $\mathcal{C}_p > 0$ be the optimal constant such that $\| f \|_{L^2_x} \leq \mathcal{C}_p \| \nabla_x f \|_{L^2_x}$ for any $f \in H^1_x$ with $\int_{ \T^3 } f \d x = 0$. If $0 \leq \vr_a < \Lambda_b: = \tfrac{4 \beta d^\frac{3}{2} D_h}{\alpha^2 \mathcal{C}_p^2 b}$, where $b, d > 0$ are given in \eqref{equation 1.12} and $D_h, \alpha, \beta > 0$ are mentioned in \eqref{AD-EECP} system, then there are some constants $\lambda, \mu > 0$ such that the matrix
	\begin{equation*}
		A = \left(
		  \begin{array}{ccc}
		  	\lambda & - \frac{\lambda b \vr_a}{2 \sqrt{d}} & - \tfrac{b \mathcal{C}_p \alpha \mu}{2 d} \\[2mm]
		  	 - \frac{\lambda b \vr_a}{2 \sqrt{d}} & D_h \mu & 0 \\[2mm]
		  	 - \tfrac{b \mathcal{C}_p \alpha \mu}{2 d} & 0 & \beta \mu \\
		  \end{array}
		\right)
	\end{equation*}
    is positive definite. This further means that there some constant $\alpha_1, \alpha_2, \alpha_3 > 0$ such that
    \begin{equation}\label{Coercivity}
    	\begin{aligned}
    		\lambda ( X^2 - \tfrac{b \vr_a}{\sqrt{d}} XY ) + \mu ( D_h Y^2 + \beta Z^2 - \tfrac{\mathcal{C}_p \alpha}{\sqrt{d}} YZ ) \geq \alpha_1 X^2 + \alpha_2 Y^2 + \alpha_3 Z^2
    	\end{aligned}
    \end{equation}
    for all $X,Y,Z \in \R$.
\end{lemma}

The proof of Lemma \ref{Lmm-Dissi} will be given later.

\begin{lemma}\label{Lmm-AD-PhiPsi}
	Let integer $s \geq 4$ be given in Lemma \ref{Lmm-AD}, $h_a = \frac{\alpha}{\beta} \vr_a$ and $\vr_a \in [0, \Lambda_b) $ be given in Lemma \ref{Lmm-Dissi}. Assume that $$( \vr (t,x), h(t,x), n(t,x)) = (\vr_a + \Phi (t,x), h_a + \Psi (t,x), 0)$$ is the solution to system \eqref{AD-EECP} under the initial conditions \eqref{IC-A1}-\eqref{IC-A2}-\eqref{IC-Average-a} with $\vr_0 = \vr_a$ on the interval $[0,T]$ constructed in part (1) of Theorem \ref{Thm3}. Then there are two quantities
	\begin{equation*}
		\begin{aligned}
			E_g (t) = & \sum_{0 \leq j \leq s} \big( c_j^\Phi \| \Phi \|^2_{H^j_x} + c_j^\Psi \| \Psi \|^2_{H^j_x} \big) \,, \\
			D_g (t) = & \sum_{0 \leq j \leq s} \big( e_j^\Phi \| \nabla_x \Phi \|^2_{H^j_x (\widetilde{D})} + e_j^{\Psi 1} \| \nabla_x \Psi \|^2_{H^j_x} + e_j^{\Psi 2} \| \Psi \|^2_{H^j_x} \big)
		\end{aligned}
	\end{equation*}
	such that
	\begin{align*}
	  \tfrac{1}{2} \tfrac{\d}{\d t} E_g (t) + D_g (t) \lesssim \big( 1 + E_g^\frac{s-1}{2} (t) \big) E_g^\frac{1}{2} (t) D_g (t)
	\end{align*}
	holds for all $t \in [0,T]$. Here $c_j^\Phi, c_j^\Psi, e_j^\Phi, e_j^{\Psi 1}, e_j^{\Psi 2} > 0$ ($j = 0,1,2,\cdots,s$) are the constants depending only on the all coefficients, $\vr_a$ and $s$.
\end{lemma}

Once the conclusions in Lemma \ref{Lmm-AD-PhiPsi} hold for $0 \leq \vr_a < \Lambda_b$ and $h_a = \frac{\alpha}{\beta} \vr_a$, the standard continuity arguments imply that
\begin{equation}\label{Global}
	\begin{aligned}
		E_g (t) + \int_0^t D_g (t') \d t' \leq C E_g (0) \quad \forall \ t \geq 0
	\end{aligned}
\end{equation}
provided that the initial energy $E_g (0) \leq \xi_0$ for some sufficiently small constant $\xi_0 > 0$. For simplicity, the details are omitted here. Notice that
\begin{equation*}
	\begin{aligned}
		E_g (t) \thicksim \mathbf{E}_s (\vr- \vr_a, h - h_a, 0) \,, \ D_g (t) \thicksim \mathbf{D}_s (\vr- \vr_a, h - h_a, 0) \,.
	\end{aligned}
\end{equation*}
Thus, one can first choose a small $\vartheta_0 > 0$ such that if $\mathbf{E}_s^{in} : = \mathbf{E}_s (\vr_a^{in} - \vr_a, h_a^{in} - h_a, 0) \leq \vartheta_0$, then $E_g (0) \leq \xi_0$. Therefore, the bound \eqref{Global} concludes the part (3) of Theorem \ref{Thm3}.

\begin{proof}[Proof of Lemma \ref{Lmm-AD-PhiPsi}]
	One only requires to prove that for any $0 \leq r \leq s$ there are some constants $c_j^\Phi, c_j^\Psi, e_j^\Phi, e_j^{\Psi 1}, e_j^{\Psi 2} > 0$ ($j = 0,1,2,\cdots,r$) such that
	\begin{equation}\label{Claim}
		\begin{aligned}
			\frac{\d}{\d t} \sum_{0 \leq j \leq r} & \big( c_j^\Phi \| \Phi \|^2_{H^j_x} + c_j^\Psi \| \Psi \|^2_{H^j_x} \big) \\
			& + \sum_{0 \leq j \leq r} \big( e_j^\Phi \| \nabla_x \Phi \|^2_{H^j_x (\widetilde{D})} + e_j^{\Psi 1} \| \nabla_x \Psi \|^2_{H^j_x} + e_j^{\Psi 2} \| \Psi \|^2_{H^j_x} \big) \lesssim \Xi_s (\Phi, \Psi) \,,
		\end{aligned}
	\end{equation}
	where $H^0_x : = L^2_x$ and
	\begin{equation}
		\begin{aligned}
			\Xi_s (\Phi, \Psi) = (1 + \| (\Phi, \Psi) \|^{s-1}_{H^s_x}) \| (\Phi, \Psi) \|_{H^s_x} \big( \| \nabla_x \Phi \|^2_{H^s_x (\widetilde{D})} + \| \nabla_x \Psi \|^2_{H^s_x} + \| \Psi \|^2_{H^s_x} \big) \,.
		\end{aligned}
	\end{equation}
	We employ the induction for $0 \leq r \leq s$ to justify the previous claim.
	
	First, for $r = 0$, from taking $L^2$-inner product by dot with $\Phi$ in the first $\Phi$-equation of \eqref{AD-PhiPsi},
	\begin{equation}
		\begin{aligned}
			\tfrac{1}{2} \tfrac{\d}{\d t} \| \Phi \|^2_{L^2_x} + \| \nabla_x \Phi \|^2_{L^2_x (\widetilde{D})} = - \vr_a \l \nabla_x \widetilde{D} (h) , \nabla_x \Phi \r_{L^2_x} - \l \Phi \nabla_x \widetilde{D} (h) , \nabla_x \Phi \r_{L^2_x} \,.
		\end{aligned}
	\end{equation}
    Together with $h = h_a + \Psi$ and \eqref{equation 1.12},
    \begin{equation}
    	\begin{aligned}
    		- \vr_a \l \nabla_x \widetilde{D} (h) , & \nabla_x \Phi \r_{L^2_x} =  - \vr_a \l \tfrac{\widetilde{D}' (h)}{\sqrt{\widetilde{D} (h)}} \nabla_x \Psi , \nabla_x \Phi \sqrt{\widetilde{D} (h)} \r_{L^2_x} \\
    		\leq & \vr_a \| \tfrac{\widetilde{D}' (h)}{\sqrt{\widetilde{D} (h)}} \|_{L^\infty_x} \| \nabla_x \Psi \|_{L^2_x} \| \nabla_x \Phi \|_{L^2_x (\widetilde{D})} \leq \tfrac{b \vr_a}{\sqrt{d}} \| \nabla_x \Psi \|_{L^2_x} \| \nabla_x \Phi \|_{L^2_x (\widetilde{D})} \,.
    	\end{aligned}
    \end{equation}
	By the Sobolev theory, $h = h_a + \Psi$ and \eqref{equation 1.12},
	\begin{equation}
		\begin{aligned}
			- \l \Phi \nabla_x \widetilde{D} (h) , \nabla_x \Phi \r_{L^2_x} \leq & \| \widetilde{D}' (h) \|_{L^\infty_x} \| \Phi \|_{L^\infty_x} \| \nabla_x \Phi \|_{L^2_x} \| \nabla_x \Psi \|^2_{L^2_x} \\
			\leq & C \| \Phi \|_{H^2_x} \| \nabla_x \Phi \|_{L^2_x (\widetilde{D})} \| \nabla_x \Psi \|_{L^2_x} \,.
		\end{aligned}
	\end{equation}
    It is therefore derived from collecting the above estimates that
    \begin{equation}\label{Phi-0L2}
    	\begin{aligned}
    		\tfrac{1}{2} \tfrac{\d}{\d t} \| \Phi \|^2_{L^2_x} + \| \nabla_x \Phi \|^2_{L^2_x (\widetilde{D})} - \tfrac{b \vr_a}{\sqrt{d}} \| \nabla_x \Psi \|_{L^2_x} \| \nabla_x \Phi \|_{L^2_x (\widetilde{D})} \lesssim \| \Phi \|_{H^2_x} \| \nabla_x \Phi \|_{L^2_x (\widetilde{D})} \| \nabla_x \Psi \|^2_{L^2_x} \,.
    	\end{aligned}
    \end{equation}
	
	Taking $L^2$-inner product by dot with $\Psi$ in the $\Psi$-equation of \eqref{AD-PhiPsi}, one has
	\begin{equation}\label{Psi-0L2}
		\begin{aligned}
			\tfrac{1}{2} \tfrac{\d}{\d t} \| \Psi \|^2_{L^2_x} + D_h \| \nabla_x \Psi \|^2_{L^2_x} + \beta \| \Psi \|^2_{L^2_x} = \alpha \l \Phi, \Psi \r_{L^2_x} \\
			\leq \alpha \| \Phi \|_{L^2_x} \| \Psi \|_{L^2_x} \leq \tfrac{\mathcal{C}_p \alpha}{\sqrt{d}} \| \nabla_x \Phi \|_{L^2_x (\widetilde{D})} \| \Psi \|_{L^2_x} \,,
		\end{aligned}
	\end{equation}
	where the Poincar\'e inequality (since $\int_{ \T^3 } \Phi \d x = 0$) and $\widetilde{D} (h) \geq d > 0$ are used. For $\lambda, \mu > 0$ given in Lemma \ref{Lmm-Dissi}, it is derived from adding $\lambda$ times of \eqref{Phi-0L2} to $\mu$ times of \eqref{Psi-0L2} that
	\begin{equation*}
		\begin{aligned}
			\tfrac{1}{2} \tfrac{\d}{\d t} \big( \lambda \| \Phi \|^2_{L^2_x} + \mu \| \Psi \|^2_{L^2_x} \big) + \lambda ( \| \nabla_x \Phi \|_{L^2_x (\widetilde{D})}^2 - \tfrac{b \vr_a}{\sqrt{d}} \|_{L^2_x (\widetilde{D})} \| \nabla_x \Psi \|_{L^2_x} ) \\
			+ \mu ( D_h \| \nabla_x \Psi \|_{L^2_x}^2 + \beta \| \Psi \|_{L^2_x}^2 - \tfrac{\mathcal{C}_p \alpha}{\sqrt{d}} \| \nabla_x \Psi \|_{L^2_x} \| \Psi \|_{L^2_x} ) \\
			\lesssim \| \Phi \|_{H^2_x} \| \nabla_x \Phi \|_{L^2_x (\widetilde{D})} \| \nabla_x \Psi \|_{L^2_x} \lesssim \Xi_s (\Phi, \Psi) \,,
		\end{aligned}
	\end{equation*}
	which imply by Lemma \ref{Lmm-Dissi}
	\begin{equation}
		\begin{aligned}
			\tfrac{1}{2} \tfrac{\d}{\d t} \big( \lambda \| \Phi \|^2_{L^2_x} + \mu \| \Psi \|^2_{L^2_x} \big) + \alpha_1 \| \nabla_x \Phi \|^2_{L^2_x (\widetilde{D})} + \alpha_2 \| \nabla_x \Psi \|^2_{L^2_x} + \alpha_3 \| \Psi \|^2_{L^2_x} \lesssim \Xi_s (\Phi, \Psi) \,.
		\end{aligned}
	\end{equation}
	Hence the claim \eqref{Claim} holds for $r = 0$.
	
	Second, we assume the claim \eqref{Claim} holds for $r$. Our goal is to prove that \eqref{Claim} is valid for $r + 1$. For any multi-index $k \in \mathbb{N}^3$ with $ |k| = r + 1 $, one first apply $\partial_x^k$ to the first $\Phi$-equation in \eqref{AD-PhiPsi} and take $L^2$-inner product by dot with $\partial^k_x \Phi$. There thereby holds
	\begin{equation}\label{Phi-L2}
		\begin{aligned}
			\tfrac{1}{2} \tfrac{\d}{\d t} \| \partial_x^k \Phi \|^2_{L^2_x} & + \| \nabla_x \partial^k_x \Phi \|^2_{L^2_x (\widetilde{D})} = - \vr_a \l \widetilde{D}' (h) \nabla_x \partial^k_x \Psi , \nabla_x \partial^k_x \Phi \r_{L^2_x} \\
			& - \l [\nabla_x \partial^k_x , \widetilde{D} (h)] \Phi , \nabla_x \partial^k_x \Phi \r_{L^2_x} - \vr_a \l [\partial^k_x , \widetilde{D}' (h) \nabla_x ] \Psi , \nabla_x \partial^k_x \Phi \r_{L^2_x} \,.
		\end{aligned}
	\end{equation}
    By \eqref{equation 1.12} and the H\"older inequality, the first term in the right-hand side of \eqref{Phi-L2} can be bounded by
    \begin{equation*}
    	\begin{aligned}
    		\tfrac{1}{2} \| \nabla_x \partial^k_x \Phi \|^2_{L^2_x (\widetilde{D})} + C_b \| \nabla_x \partial^k_x \Psi \|^2_{L^2_x}
    	\end{aligned}
    \end{equation*}
    for some constant $C_b > 0$. Moreover, from Lemma \ref{Lmm-D(h)}, $h = h_a + \Psi$, \eqref{equation 1.12} and the fact $\widetilde{D} (h) \geq d > 0$, the last two terms in the right-hand side of \eqref{Phi-L2} can be controlled by
    \begin{equation*}
    	\begin{aligned}
    		C (1 + \| \Psi \|^{s-1}_{H^s_x}) ( \| \Phi \|_{H^s_x} + \| \Psi \|_{H^s_x} ) \| \nabla_x \Psi \|_{H^s_x} \| \nabla_x \Phi \|_{H^s_x (\widetilde{D})} \leq C \Xi_s (\Phi, \Psi) \,.
    	\end{aligned}
    \end{equation*}
	As a result, after summing up for $|k| = r+1$, the above estimates imply
	\begin{equation}\label{Phi-Hs-r+1}
		\begin{aligned}
			\tfrac{1}{2} \tfrac{\d}{\d t} \sum_{|k|=r+1} \| \partial^k_x \Phi \|^2_{L^2_x} + \tfrac{1}{2} \sum_{|k|=r+1} \| \nabla_x \partial^k_x \Phi \|^2_{L^2_x (\widetilde{D})} - C_b \sum_{|k|=r+1} \| \nabla_x \partial^k_x \Psi \|^2_{L^2_x} \lesssim \Xi_s (\Phi, \Psi) \,.
		\end{aligned}
	\end{equation}
	
	By applying $\partial^k_x$ to the $\Psi$-equation in \eqref{AD-PhiPsi} and taking $L^2$-inner product by dot with $\partial^k_x \Psi$, one has
	\begin{equation*}
		\begin{aligned}
			\tfrac{1}{2} \tfrac{\d}{\d t} \| \partial^k_x \Psi \|^2_{L^2_x} + D_h \| \nabla_x \partial^k_x \Psi \|^2_{L^2_x} + \beta \| \partial^k_x \Psi \|^2_{L^2_x} = & \alpha \l \partial^k_x \Phi , \partial^k_x \Psi \r_{L^2_x} \\
			\leq & \tfrac{\beta}{2} \| \partial^k_x \Psi \|^2_{L^2_x} + \tfrac{\alpha^2}{2 \beta d} \| \partial^k_x \Phi \|^2_{L^2_x (\widetilde{D})} \,.
		\end{aligned}
	\end{equation*}
    Summing up for all $|k| = r + 1$ implies
	\begin{equation}\label{Psi-L2-r+1}
		\begin{aligned}
			\tfrac{1}{2} \tfrac{\d}{\d t} \sum_{|k|=r+1} \| \partial^k_x \Psi \|^2_{L^2_x} + D_h \sum_{|k|=r+1} \| \nabla_x \partial^k_x \Psi \|^2_{L^2_x} + \tfrac{\beta}{2} \sum_{|k|=r+1} & \| \partial^k_x \Psi \|^2_{L^2_x} \\
			\leq & C_* \sum_{|k| = r} \| \nabla_x \partial^k_x \Phi \|^2_{L^2_x (\widetilde{D})} \,.
		\end{aligned}
	\end{equation}
    By adjusting the balance coefficients among \eqref{Phi-Hs-r+1}, \eqref{Psi-L2-r+1} and the induction hypothesis for the case $r$ in \eqref{Claim}, one yields that the claim \eqref{Claim} holds for the case $r+1$. Observe that $\Xi_s (\Phi, \Psi) \thicksim (1 + E_g^\frac{s-1}{2} (t) ) E_g^\frac{1}{2} (t) D_g (t)$. Then the proof of Lemma \ref{Lmm-AD-PhiPsi} is finished.
\end{proof}

\begin{proof}[Proof of Lemma \ref{Lmm-Dissi}]
	Observe that \eqref{Coercivity} holds if and only if the matrix $A$ is positive definite. In order to prove the positivity of the matrix $A$, it is equivalent to show the all order principal minor determinants of $A$ is positive. More precisely, one shall show that there are constants $\lambda, \mu > 0$ such that
	\begin{equation*}
		  \begin{aligned}
		  	\lambda > 0 \,, \quad \det
		  	\left(
		  	  \begin{array}{cc}
		  	    \lambda & - \tfrac{\lambda b \vr_a}{2 \sqrt{d}} \\[2mm]
		  	    - \tfrac{\lambda b \vr_a}{2 \sqrt{d}} & D_h \mu \\
		  	  \end{array}
		  	\right) > 0 \,, \det A > 0 \,,
		  \end{aligned}
	\end{equation*}
    which is equivalent to
    \begin{equation}
    	\begin{aligned}
    		\tfrac{\mu}{\lambda} > \tfrac{b \vr_a}{4 D_h \sqrt{d}} \,, \quad \left( \tfrac{\mu}{\lambda} \right)^2 - \tfrac{4 \beta d}{\alpha^2 \mathcal{C}_p^2} \left( \tfrac{\mu}{\lambda} \right) + \tfrac{\beta b \sqrt{d} \vr_a}{\alpha^2 \mathcal{C}_p^2 D_h} < 0 \,.
    	\end{aligned}
    \end{equation}
    hold for some $\lambda, \mu > 0$. This is equivalent to
    \begin{equation*}
    	\left\{
    	\begin{aligned}
    		& \Gamma = \left( - \tfrac{4 \beta d}{\alpha^2 \mathcal{C}_p^2} \right)^2 - 4 \cdot \tfrac{\beta b \sqrt{d} \vr_a}{\alpha^2 \mathcal{C}_p^2 D_h} > 0 \,, \\
    		& \tfrac{b \vr_a}{2 D_h \sqrt{d}} < \tfrac{4 \beta d}{\alpha^2 \mathcal{C}_p^2} + \tfrac{1}{2} \sqrt{\Gamma} \,.
    	\end{aligned}
        \right.
    \end{equation*}
    which is solved by $\vr_a < \tfrac{4 \beta d^\frac{3}{2} D_h}{\alpha^2 \mathcal{C}_p^2 b}$. The proof of Lemma \ref{Lmm-Dissi} is therefore completed.
\end{proof}

\appendix

\section{Instability of linearized models around steady states $(0,0, n_0)$}\label{Appendix}

Observe that $(0, 0, n_0)$ with the constants $n_0 > 0$ is a steady solution of both \eqref{EECPK} and \eqref{AD-EECP} systems. The main goal of  this appendix is to study the instability of the linearized \eqref{EECPK} and \eqref{AD-EECP} systems around the steady state $(0,0,n_0)$. The results motivate us to prove the global existence around the steady states of \eqref{EECPK} and \eqref{AD-EECP} systems excluded the form of $(0,0,n_0)$ with $n_0 > 0$.

\subsection{Linear instability of the \eqref{EECPK} model}
 In this subsection, we consider the linear instability of the \eqref{EECPK} model around $(0, 0, n_0)$. Set
 \begin{align}\label{decomposition1}
 	\rho =  \delta \rho,
 	h = \delta h, n = n_0 + \delta n.
 \end{align}
Plugging \eqref{decomposition1} into the \eqref{EECPK} system yields the following linearized K-EECP model
\begin{equation}\label{linearized}
   \partial_t
   \left(
     \begin{array}{c}
     	\delta \rho \\
     	\delta h \\
     	\delta n
     \end{array}
   \right) =
   \left(
     \begin{array}{ccc}
     	D(z) \Delta_x - \tfrac{1}{\eps} \partial_z (g (z,0) \cdot \, ) + \gamma n_0 & 0 & 0 \\
     	\alpha \int_{ \T_w } \cdot \, \d z & D_h \Delta_x - \beta & 0 \\
     	- \xi n_0 \int_{ \T_w } \cdot \, \d z & 0 & D_n \Delta_x \\
     \end{array}
   \right)
   \left(
     \begin{array}{c}
   	    \delta \rho \\
   	    \delta h \\
   	    \delta n
     \end{array}
   \right) \,,
\end{equation}
where the symbol with the form $F(\cdot) g$ means $F (g)$ and $g (z, h)$ is given in \eqref{g(z,w)}. Let $\widehat{\delta \rho} (t, m, m_z)$, $\widehat{\delta h} (t, m)$ and $\widehat{\delta n} (t, m)$ be the Fourier transform of $\delta \rho (t,x,z)$, $\delta h (t,x)$ and  $\delta n (t,x)$, respectively. Let $i = \sqrt{-1}$ here. From substituting the corresponding normal nodes $\widehat{\delta \rho} \exp[\lambda t +im \cdot x+i m_z z]$, $\widehat{\delta h} \exp[\lambda t +i m \cdot x]$ and $\widehat{\delta n} \exp[\lambda t +i m \cdot x]$ of $\delta \rho$, $\delta h$ and $\delta n$ into \eqref{linearized}, the characteristic matrix of linear system \eqref{linearized} is derived that
\[
\begin{pmatrix*}[c]
  -D(z)|m|^2+\tfrac{k_V}{\eps}-\tfrac{g(z,0)m_z}{\eps}i+\gamma n_0-\lambda & 0 & 0\\
 \alpha \int_{\T_w} \exp[im_z z] \d z  & -D_h|m|^2-\beta-\lambda & 0 \\
 -\xi n_0 \int_{\T_w}\exp[im_z z] \d z & 0 & -D_n|m|^2-\lambda
\end{pmatrix*} \,.
\]
Here $m = (m_1, m_2, m_3) \in \mathbb{Z}^3$ and $m_z \in \mathbb{Z}$. Then the eigenvalues of the linear system \eqref{linearized} are
  \begin{equation*}
     \lambda_1=-D_n|m|^2, \ \lambda_2=-D_h|m|^2-\beta, \ \lambda_3=-D(z)|m|^2+\tfrac{k_V}{\eps}+\gamma n_0-\tfrac{g(z,0)m_z}{\eps}i \,.
   \end{equation*}
Observe that if $m = 0$, one has $\lambda_1 = 0$ and $\lambda_2 = - \beta < 0$. Moreover, if $|m|^2$ is sufficiently small, the real part $\Re \lambda_3 = - D(z) |m|^2 + \tfrac{k_V}{\eps} + \gamma n_0 > 0$. It is therefore concluded that
\begin{proposition}
	The linearized system \eqref{linearized} is unstable.
\end{proposition}

\begin{remark}\label{Rmk-A1}
	The instability of \eqref{linearized} can also be understood by another way. The zero frequency $m = 0$ corresponds to the quantities of ``total mass", namely, $$A_\rho (t) = \iint_{ \T^3 \times \T_w } \delta \rho (t,x,z) \d z \d x \,, \ A_h (t) = \int_{ \T^3 } \delta h (t,x) \d x \,, \ A_n (t) = \int_{ \T^3 } \delta n (t,x) \d x \,. $$
	From \eqref{linearized}, there hold
	\begin{equation}\label{ODE-1}
		\left\{
		  \begin{aligned}
		  	& \tfrac{\d}{\d t} A_\rho (t) = \gamma n_0 A_\rho (t) \,, \\
		  	& \tfrac{\d}{\d t} A_h (t) = \alpha A_\rho (t) - \beta A_h (t) \,, \\
		  	& \tfrac{\d}{\d t} A_n (t) = - \xi n_0 A_\rho (t) \,.
		  \end{aligned}
		\right.
	\end{equation}
    Then the ODE system \eqref{ODE-1} can be solved by
    \begin{equation}
    	\left\{
    	  \begin{aligned}
    	  	A_\rho (t) = & A_\rho (0) e^{\gamma n_0 t} \,, \\
    	  	A_h (t) = & [ A_h (0) - \tfrac{\alpha}{\gamma n_0} A_\rho (0) ] e^{- \beta t} + \tfrac{\alpha}{\gamma n_0} A_\rho (0) e^{(\gamma n_0 - \beta) t} \,, \\
    	  	A_n (t) = & A_n (0) + \tfrac{\xi}{\gamma} A_\rho (0) - \tfrac{\xi}{\gamma} A_\rho (0) e^{\gamma n_0 t} \,.
    	  \end{aligned}
    	\right.
    \end{equation}
	Once $n_0 > 0$ and $A_\rho (0) \neq 0$, $A_\rho (t)$ and $A_n (t)$ are unstable. Moreover, if $\gamma n_0 > \beta$, $A_h (t)$ is also unstable.
\end{remark}

\subsection{Linear instability of the \eqref{AD-EECP} model}
 In this subsection, the main goal is to consider the instability of the linearized \eqref{AD-EECP} system around $(0, 0, n_0)$ with constant $n_0 > 0$. Define a fluctuation
 \begin{align}\label{decomposition2}
   \vr (t,x) = \delta \vr (t,x) \,, \ h (t,x) = \delta h (t,x) \,, \ n (t,x) = n_0 + \delta n (t,x) \,.
  \end{align}
Then the linearized \eqref{AD-EECP} system reads
\begin{equation}\label{linearizedA2}
  \partial_t
  \left(
    \begin{array}{c}
    	\delta \vr \\
    	\delta h \\
    	\delta n
    \end{array}
  \right) =
  \left(
    \begin{array}{ccc}
    	\widetilde{D} (0) \Delta_x + \gamma n_0 & 0 & 0 \\
    	\alpha & D_h \Delta_x - \beta & 0 \\
    	- \xi n_0 & 0 & D_n \Delta_x \\
    \end{array}
  \right)
  \left(
    \begin{array}{c}
  	  \delta \vr \\
  	  \delta h \\
  	  \delta n
    \end{array}
  \right) \,.
\end{equation}
It is easy to see that the corresponding characteristic matrix is
\[
\begin{pmatrix*}[c]
  -\widetilde{D}(0)|m|^2+\gamma n_0-\lambda & 0 & 0\\
 \alpha & -D_h|m|^2-\beta-\lambda & 0 \\
 -\xi n_0 & 0 & -D_n|m|^2-\lambda
\end{pmatrix*} \,.
\]
The eigenvalues thereby are
\begin{equation*}
 \lambda_1 = - D_n |m|^2 \leq 0 \,, \ \lambda_2 = - D_h |m|^2 - \beta < 0 \,, \ \lambda_3 = - \widetilde{D}(0) |m|^2 + \gamma n_0 \,.
\end{equation*}
Note that if $|m|^2$ is small enough, $\lambda_3 > 0$. It thereby infers that
\begin{proposition}
	The linearized AD-EECP system \eqref{linearizedA2} is unstable.
\end{proposition}

\begin{remark}
	As similarly discussed in Remark \ref{Rmk-A1}, one also can focus on the so-called ``total mass"
	\begin{equation*}
		\begin{aligned}
			M_\vr (t) = \int_{ \T^3 } \delta \vr (t,x) \d x \,, \ M_h (t) = \int_{ \T^3 } \delta h (t,x) \d x \,, \ M_n (t) = \int_{ \T^3 } \delta n (t,x) \d x \,,
		\end{aligned}
	\end{equation*}
    which correspond to the zero frequency case. Observe that $(M_\vr, M_h, M_n)$ solves the ODE system \eqref{ODE-1}. Then the same analysis in Remark \ref{Rmk-A1} concludes that \eqref{linearizedA2} is unstable if $n_0 > 0$ and $M_\vr (0) \neq 0$.
\end{remark}

\section*{Acknowledgments}

The author N. Jiang is supported by grants from the National NSFC under contract Nos. 11971360 and 11731008. Y.-L. Luo is supported by the Starting Research Fund from South China University of Technology (Double First-Class Construction Project) under contract No. D6211300. M. Tang and Y. M. Zhang are partially supported by NSFC 11871340.

\bibliography{reference}

\begin{thebibliography}{99}
	
    \bibitem{BGCAW-2005} S. Basu, Y. Gerchman, C. H. Collins, F. H. Arnold, and R. Weiss.
    A synthetic multicellular system for programmed pattern formation. {\em Nature.} 2005, {\bf 434} (7037), 1130.

    \bibitem{BSM-2009} R. E. Baker, S. Schnell, and P. K. Maini.
    Waves and patterning in developmental biology: vertebrate segmentation and
           feather bud formation as case studies. {\em Int. J. Dev. Biol.} 2009, {\bf 53}(0), 783-794.

    \bibitem{EL-2010} M. Elowitz and W. A. Lim.
       Build life to understand it. {\em Nature.} 2010, {\bf 468} (7326), 889-890.

    \bibitem{PRL2012} X. Fu, L. Tang, C. Liu, J. Huang, T. Hwa, and P. Lenz.
    Stripe formation in bacterial systems with density-suppressed motility. {\em Physical Review Letters.} 2012; {\bf 108} (19), 198102.


    \bibitem{Guo-Indiana} Y. Guo,
    The Boltzmann equation in the whole space. {\em Indiana Univ. Math. J.} {\bf 53} (2004), no. 4, 1081-1094.

%    \bibitem{GJL-2019} M. M. Guo, N. Jiang and Y.-L. Luo, From Vlasov-Poisson-Boltzmann system to incompressible Navier-Stokes-Fourier-Poisson system: convergence for classical solutions. {\em arXiv:2006.16514}

    \bibitem{H-1980} B. Hobom, Gene surgery: on the threshold of synthetic biology. {\em Med Klin.} 1980, {\bf 75} (24), 834-841.

    \bibitem{JL-SIAM-2019} N. Jiang and Y.-L. Luo, On well-posedness of Ericksen-Leslie's hyperbolic incompressible liquid crystal
            model. {\em SIAM J. Math. Anal.}, 2019, {\bf 51 }(1), 403-434.

%    \bibitem{JL-2020} N. Jiang and Y.-L. Luo, From Vlasov-Maxwell-Boltzmann system to two-fluid incompressible Navier-Stokes-Fourier-Maxwell system with Ohm's law: convergence for classical solutions. {\em arXiv:1905.04739}.

    \bibitem{JLZ-2021-CMS} N. Jiang, Y.-L. Luo and X. Zhang, Stability of equilibria to the model for non-isothermal electrokinetics. {\em Commun. Math. Sci.} {\bf 19} (2021), no. 3, 687-720.

    \bibitem{JXZ-Indiana} N. Jiang, C-J. Xu and H. J. Zhao,
    Incompressible Navier-Stokes-Fourier limit from the Boltzmann equation: classical solutions. {\em Indiana Univ. Math. J.}, {\bf 67} (2018), no. 5, 1817-1855.

    \bibitem{JKW-SIAM2018} H-Y. Jin, Y-J.  Kim, and Z-A. Wang,
    Boundedness, stabilization, and pattern formation driven by density-suppressed motility. {\em SIAM J. Appl. Math.} {\bf 78} (2018), no. 3, 1632-1657.

    \bibitem{JSW-JDE2020} H-Y. Jin, S. J.  Shi, and Z-A. Wang,
    Boundedness and asymptotics of a reaction-diffusion system with density-dependent motility. {\em J. Differential Equations} {\bf 269} (2020), no. 9, 6758-6793.

%    \bibitem{KC-2010} A. S. Khalil and J. J. Collins.
%    Synthetic biology: applications come of age. {\em Nature reviews Genetics.} 2010, {\bf
%    11} (5), 367.

    \bibitem{LFLRCL-2011-S} C. Liu, X. Fu, L. Liu, X. Ren, C. K. L. Chau, S. Li, L. Xiang, H. Zeng, G. Chen, L. H. Tang, P. Lenz, X.
            Cui, W. Huang, T. Hwa and J. D. Huang. Sequential Establishment of Stripe Patterns in an Expanding Cell Population. {\em
            Science}, 2011, {\bf 334} (6053), 238-241.

    \bibitem{MPW-PhysicaD} M. J. Ma,  R. Peng, and Z-A. Wang,
    Stationary and non-stationary patterns of the density-suppressed motility model. {\em Phys. D} {\bf 402} (2020), 132259.

%    \bibitem{Majda-Book} A. Majda and A. Bertozzi, {\em Vorticity and incompressible flow.} Cambridge Texts in Applied
%            Mathematics,{\bf 27}. Cambridge University Press, Cambridge, 2002.

    \bibitem{M-2002} J. D. Murray. {\em Mathematical biology.} vol. 2. Springer; 2002.

%    \bibitem{MHHPSP-2010} A. Marrocco, H. Henry, I. B. Holland, M. Plapp, S. J. Seror, and B. Perthame. Models of Self-Organizing Bacterial Communities and Comparisons with Experimental Observations. {\em Mathematical Modelling of Natural Phenomena.} 2010; {\bf 5} (1), 148-162.

    \bibitem{MV-2009} S. Mukherji, and A. Van Oudenaarden.
    Synthetic biology: understanding biological design from synthetic circuits.
         {\em Nature reviews Genetics.} 2009; {\bf 10} (12), 859.

%    \bibitem{PMO-1999} K. J. Painter, P. K. Maini and H. G. Othmer. Stripe formation in juvenile Pomacanthus explained by a generalized turing mechanism with chemotaxis. {\em Proc Natl Acad Sci U S A.} 1999; {\bf 96} (10), 5549-5554.

    \bibitem{Perthame-2015-BOOK} B. Perthame, {\em Parabolic equations in biology. Growth, reaction, movement and diffusion.} Lecture Notes on Mathematical Modelling in the Life Sciences. Springer, Cham, 2015. xii+199 pp.

%    \bibitem{PKMM-2010} P. Rugbjerga, K. Sarup-Lytzena, M. Nagya, and M. O. A. Sommera,
%    Synthetic addiction extends the productive life time of engineered {\em Escherichia coli} populations. {\em PNAS}, 2018; {\bf 115} (10), 2347-2352.
%
%    \bibitem{STY-2014} G. Si, M. Tang and X. Yang. A pathway-based mean-field model for E. coli chemotaxis: Mathematical derivation and
%           Keller-Segel limit. {\em Multiscale Modeling and Simulations}. 2014; {\bf 12} (2), 907-926.
%
%    \bibitem{T-1952} A. M. Turing. The Chemical Basis of Morphogenesis. {\em Philosophical Transactions of the Royal Society B.} 1952; {\bf 237} (641), 37-72.
%
%    \bibitem{VS-2015} A. Volkening , B. Sandstede. Modelling stripe formation in zebrafish: an agent-based approach. {\em Journal of the
%           Royal Society Interface.} 2015; {\bf 12} (112), 20150812.
%
%    \bibitem{WOLDPR-2017} Q. Wang, J. W. Oh, H. L. Lee, A. Dhar, T. Peng, R. Ramos, et al. A multi-scale model for hair follicles reveals
%          heterogeneous domains driving rapid spatiotemporal hair growth patterning. {\em eLife}. 2017; 6.


    \bibitem{XXT-2018-CB} X. Xue, C. Xue and M. Tang. The role of intracellular signaling in the stripe formation in engineered
          Escherichia coli populations. {\em PLOS Comput. Biol.}, 2018, {\bf 14} (6).

%    \bibitem{X-2015} C. Xue. Macroscopic equations for bacterial chemotaxis: integration of detailed biochemistry of cell signaling.
%          {\em J Math Biol.} 2015; {\bf 70} (1-2), 1-44.
%
%    \bibitem{XO-2012} X. Xin, H. G. Othmer, A ``trimer of dimers"-based model for the chemotactic signal transduction network in bacterial chemotaxis. {\em Bull Math Biol.} 2012; {\bf 74} (10), 2339-2382.


\end{thebibliography}

\end{document}